%% file: spmeas.rv.tex
\documentclass[10pt, a4paper]{amsart}

\usepackage{amssymb,amsmath,graphics,verbatim}
\usepackage{latexsym}
\usepackage{eucal}

\usepackage{epsfig}
\usepackage{color}

\makeatletter
\@addtoreset{equation}{section}\makeatother

\newtheorem{theo}{Theorem}[section]
\newtheorem{lem}[theo]{Lemma}
\newtheorem{prop}[theo]{Proposition}
\newtheorem{cor}[theo]{Corollary}

\theoremstyle{remark} \newtheorem{remark}[theo]{Remark}
\newtheorem{defn}[theo]{Definition}
\newcommand{\mc}{\mathcal}
\newcommand{\rr}{\mathbb{R}}
\newcommand{\nn}{\mathbb{N}}

\newcommand{\la}{\lambda}
\newcommand{\eps}{\epsilon}

\newcommand{\pl}{\partial}
\newcommand{\x}{\times}

\newcommand{\demi}{\frac{1}{2}}

\newcommand{\indic}{\operatorname{1\negthinspace l}}

\newcommand{\zf}{\mathrm{zf}}
\newcommand{\bfo}{\mathrm{bf}_0}
\newcommand{\rbo}{\mathrm{rb}_0}
\newcommand{\lbo}{\mathrm{lb}_0}
\newcommand{\lb}{\mathrm{lb}}
\newcommand{\rb}{\mathrm{rb}}
\newcommand{\bfa}{\mathrm{bf}}
\newcommand{\bfc}{\mathrm{bf}}
\newcommand{\sca}{\mathrm{sc}}
\newcommand{\ff}{\mathrm{ff}}
\newcommand\Id{\operatorname{Id}}
\newcommand\conic{\operatorname{conic}}

\newcommand\Pconic{P_{\conic}}

\newcommand\RR{\mathbb{R}}
\newcommand\CC{\mathbb{C}}
\newcommand\NN{\mathbb{N}}
\newcommand\extunion{\overline{\cup}}
\newcommand\Omegab{\tilde \Omega_b^{1/2}}
\newcommand\MMksc{M^2_{k, \sca}}
\newcommand\Mkb{M_{k, b}}
\newcommand\Tkb{{}^{k, b} T}
\newcommand\Tkbstar{{}^{k, b} T^*}
\newcommand\MMkb{M^2_{k, b}}

\newcommand{\SC}{\ensuremath{\mathrm{sc}}}
\newcommand{\SF}{\ensuremath{\mathrm{s}\Phi}}

\newcommand{\Tscstar}[1][\mbox{}]{{^{\SC}T^*_{#1}}}

\newcommand\Lsharp{L^\#}

\def\scOh{{}^{\operatorname{sc}} \Omega^{1/2}}

\newcommand\ZZb{Z^2_b}
\newcommand\ZZZb{Z^3_b}
\newcommand\ZZbsc{Z^2_{b, \sca}}

\newcommand\phg{\operatorname{phg}}
\newcommand\Omegabht{\tilde \Omega_b^{1/2}}
\newcommand\Omegahbsc{\Omega_{b, sc}^{1/2}}
\renewcommand\Re{\operatorname{Re}}
\renewcommand\Im{\operatorname{Im}}
\def\ang#1{\langle #1 \rangle}
\newcommand\CIdot{\dot C^\infty}

\newcommand\Vsc{\mathcal{V}_{\SC}}
\newcommand\Vb{\mathcal{V}_{b}}
\newcommand\Vkb{\mathcal{V}_{k,b}(\MMkb)}

\newcommand\Lbf{L^{\bfc}}
\newcommand\MMb{M^2_b}

\newcommand{\Nscstar}[1][\mbox{}]{{^{\SC}N^*_{#1}}}
\newcommand\Diagb{\text{Diag}_b}

\newcommand{\Nsfstar}[1][\mbox{}]{{^{\SF}N^*_{#1}}}
\newcommand\CI{C^\infty}
\newcommand\Omegakb{\Omega_{k,b}}
\newcommand\Omegakbh{\Omega_{k,b}^{1/2}}
\newcommand\Omegabh{\Omega_{b}^{1/2}}
\newcommand\Omegasch{\Omega_{\SC}^{1/2}}
\newcommand\mcA{\mathcal{B}}
\newcommand\mcB{\mathcal{B}}
\newcommand\mcE{\mathcal{E}}
\newcommand\Ebar{\overline{\mathcal{E}}}
\newcommand\half{{1/2}}
\newcommand\Lambdat{\tilde \Lambda}
\newcommand\Lambdas{\Lambda^\sharp}
\newcommand\Lambdah{\hat \Lambda}

\newcommand\phgc{\mathcal{A}}

\newcommand\Gb{{ G_{b}}}
\newcommand\Gsc{G^{\SC}}
\newcommand\Esc{E^{\SC} }

\newcommand\gconic{g_{\conic}}
\newcommand\gcyl{g_{\operatorname{cyl}}}
\newcommand\Diagbsc{\mathrm{Diag}_{b, sc}}
\newcommand\Omegabsc{\Omega_{b,sc}^{1/2}}
\newcommand\phit{\tilde \phi}
\def\dbyd#1#2{\frac{\partial #1}{\partial #2}}

\newcommand\spn{\operatorname{span}}
\newcommand\mcC{\mathcal{C}}
\newcommand\Pbconic{{P_{b, \conic}}}
\newcommand\Ha{\operatorname{Ha}}

\begin{document}
\title[Resolvent and spectral measure at low energy]{Resolvent at low energy III: the spectral measure}
\author{Colin Guillarmou}
\address{DMA, U.M.R. 8553 CNRS\\
Ecole Normale Sup\'erieure\\
45 rue d'Ulm\\ 
F 75230 Paris cedex 05 \\France}
\email{cguillar@dma.ens.fr}
\author{Andrew Hassell}
\address{Department of Mathematics, Australian National University \\ Canberra ACT 0200 \\ AUSTRALIA}
\email{Andrew.Hassell@anu.edu.au}
\author{Adam Sikora}
\address{Department of Mathematics, Australian National University \\ Canberra ACT 0200 \\ AUSTRALIA, and \ Department of Mathematics, Macquarie University \\ NSW 2109 \\ AUSTRALIA}
\email{sikora@mq.edu.au}
\thanks{This research was supported by Australian Research Council Discovery grants DP0771826  and DP1095448 (A.H., A.S.) and a Future Fellowship (A.H.). C.G. is partially supported by ANR grant ANR-09-JCJC-0099-01 and by the PICS-CNRS Progress in Geometric
Analysis and Applications, and thanks the math department of  ANU for 
its hospitality. C.G thanks M.Tohaneanu for useful discussions.}

\begin{abstract} Let $M^\circ$ be a complete noncompact manifold and $g$ an asymptotically conic Riemaniann metric on $M^\circ$, in the sense that  $M^\circ$ compactifies to a manifold with boundary $M$ in such a way that $g$ becomes a scattering metric on $M$. Let $\Delta$ be the positive Laplacian associated to $g$, and $P = \Delta + V$, where $V$ is a potential function obeying certain conditions. 
We analyze the asymptotics of the spectral measure $dE(\lambda) = (\lambda/\pi i) \big( R(\lambda+i0) - R(\lambda - i0) \big)$ of $P_+^{1/2}$, where $R(\lambda) = (P - \lambda^2)^{-1}$, as $\lambda \to 0$, in a manner similar to that done by the second author and Vasy in \cite{HV2}, and by the first two authors in \cite{GH1, GH2}. The main result is that the spectral measure  has a simple, `conormal-Legendrian'  singularity structure on a space which was introduced in \cite{GH1} and is obtained from $M^2 \times [0, \lambda_0)$ by blowing up a certain number of boundary faces. We use this to deduce results about the asymptotics of the wave solution operators $\cos(t \sqrt{P_+})$ and $\sin(t \sqrt{P_+})/\sqrt{P_+}$, and the Schr\"odinger  propagator $e^{itP_+}$, as $t \to \infty$. In particular, we prove the analogue of Price's law for odd-dimensional asymptotically conic manifolds. 

In future articles, this result on the spectral measure will be  used to (i) prove restriction and spectral multiplier estimates on asymptotically conic manifolds, and (ii)  prove long-time dispersion and Strichartz estimates for solutions of the Schr\"odinger equation on $M$, provided $M$ is nontrapping. 
\end{abstract}
\maketitle
\tableofcontents

\section{Introduction}
This paper continues the investigations carried out in \cite{HV1}, \cite{HV2}, \cite{HW}, \cite{GH1} and \cite{GH2} concerning the Schwartz kernel of the boundary values of the resolvent $(P - (\lambda \pm i0)^2)^{-1}$, where $P$ is the (positive) Laplacian $\Delta_g$ on an asymptotically conic manifold $(M^\circ,g)$, or more generally a Schr\"odinger operator $P = \Delta_g + V$ where $V$ is a suitable potential function. This was done for a fixed real $\lambda$ in \cite{HV1} and \cite{HV2} (and valid uniformly for $\lambda$ in a compact interval of $(0, \infty)$), for $\lambda \to \infty$ in \cite{HW} and for $\lambda = ik$, with $k$ real and tending to zero, that is, inside the resolvent set but approaching the bottom of the spectrum, $0$, in \cite{GH1} (and also in \cite{GH2}, where zero eigenvalues and zero-resonances were treated). Here we treat the case $\lambda$ real and tending to zero. 

One of the main reasons for doing this is to obtain results about the spectral measure, which can be expressed in terms of the difference between the outgoing and incoming resolvents $R(\lambda \pm i0)$, where $R(\sigma) = (P - \sigma^2)^{-1}$. 
Very complete results about the singularity structure of the spectral measure are known from \cite{HV1}, \cite{HV2} and \cite{HW} for $\lambda \in [\lambda_0, \infty)$ for any $\lambda_0 > 0$. To complete the picture we derive the asymptotics as $\lambda \to 0$ here. We can then, at least in principle, analyze the Schwartz kernel of any function of $P$ by integrating over the spectral measure. In the present paper we use our result on the low-energy asymptotics of the spectral measure to deduce long-time asymptotics of the wave and Schr\"odinger propagators determined by $P$. 
In future articles, we will 
treat two aspects of functional calculus for Laplace-type operators on an asymptotically conic manifold: (i) restriction estimates, that is $L^p \to L^{p'}$ estimates on the spectral measure, and $L^p \to L^p$ estimates for fairly general functions of the Laplacian, and (ii) 
long-time dispersion and Strichartz estimates for solutions of the Schr\"odinger equation on nontrapping asymptotically conic manifolds. 

\subsection{Geometric setting}\label{sect:Geometricsetting}
The geometric setting for our analysis is the same as in \cite{GH1}, which we now recall. Let $(M^\circ, g)$ be a complete  noncompact Riemannian manifold of dimension $n \geq 2$ with one end, 
diffeomorphic to $S\x(0,\infty)$ where $S$ is a smooth compact connected manifold without boundary. (The assumption that $S$ is connected is for simplicity of exposition; see Remark~\ref{manyends}.) 
We assume that $M^\circ$ admits a compactification $M$ to a smooth compact manifold with boundary,  with $\pl M=S$,
such that the metric $g$ becomes a \emph{scattering metric} or \emph{asymptotically conic metric} on $M$. That means that there
is a boundary defining function $x$ for $\pl M$ (i.e.\ $\pl M=\{x=0\}$ and $dx$ does not vanish on $\pl M$)  such that in a collar neighbourhood 
$[0,\eps)_x\x\pl M$ near $\pl M$, $g$ takes the form 
\begin{equation}\label{metricconic}
g=\frac{dx^2}{x^4}+\frac{h(x)}{x^2} = \frac{dx^2}{x^4}+\frac{\sum h_{ij}(x,y) dy^i dy^j}{x^2}
\end{equation}
where $h(x)$ is a smooth family of metrics on $S$. We then call $M^\circ$ an asymptotically conic manifold, or a scattering manifold. 
Notice that if $h(x) = h$ is independent of $x$ for small $x < x_0$, then setting $r = 1/x$ the metric reads 
$$
dr^2+ r^2 h(0), \qquad r > x_0^{-1},
$$
which is a conic metric; in this sense, the metric $g$ is asymptotically conic. For a general scattering metric taking the form \eqref{metricconic}, we view $r = 1/x$ as a generalized `radial coordinate', as the distance to any fixed point of $M$ is given by $r + O(1)$ as $r \to \infty$.
A metric cone itself is not an example of an asymptotically conic manifold, since cone points are not allowed, except in the case of Euclidean space, where the cone point is a removable singularity. 
In spite of this, the methods of this paper apply to metric cones, and in fact we analyze the resolvent kernel on a metric cone as an ingredient of our analysis on asymptotically conic manifolds.

We let $V$ be a real potential function on $M$ such that 
\begin{equation}\label{hyp2}
\begin{gathered}
V\in C^{\infty}(M), \quad V(x,y)=O(x^2) \textrm{ as }x\to 0,\\ 
\textrm{ with }\Delta_{\pl M}+\frac{(n-2)^2}{4}+V_0>0 \textrm{ on } L^2(\partial M, h(0)), \  \textrm{ where } V_0=(x^{-2}V)|_{\pl M}.
\end{gathered}
\end{equation}
Here, $\Delta_{\partial M}$ is the (positive) Laplacian with respect to the metric $h(0)$, we let $(\nu_j^2)_{j=0}^\infty$ 
be the set of increasing eigenvalues of $\Delta_{\pl M}+\frac{(n-2)^2}{4}+V_0$ and the condition in \eqref{hyp2} is that the lowest eigenvalue $\nu_0^2$  is \emph{strictly} positive. Notice that  $V_0 \equiv 0$ is not allowed if $n=2$, but is allowed for $n \geq 3$, and indeed then $V_0$ could be somewhat negative: for example, any negative constant greater than $(n/2 - 1)^2$. 
We shall further assume that 
\begin{equation}
\Delta_g + V \text{ has no zero eigenvalue or zero-resonance.}
\label{hyp3}\end{equation}
(Recall that a zero-resonance of $P$ is a solution $u$ to $Pu = 0$ where $u \notin L^2(M)$, but $u \to 0$ at infinity. Zero-resonances do not exist when $V \geq 0$.)

Let $P = \Delta_g + V$. 
Then $P$, with domain $H^2(M, dg)$, 
is self-adjoint on $L^2(M,dg)$ and a consequence of \eqref{hyp2} and \eqref{hyp3} is that its spectrum is the union of absolutely continuous spectrum on $[0, \infty)$ together with possibly a finite number of negative eigenvalues. It is known that, for $\lambda > 0$, the limits 
$$
R(\lambda \pm i0) := \lim_{\eta \downarrow 0} (P - (\lambda \pm i\eta)^2)^{-1}$$ exist  as bounded operators from $x^{1/2 + \eps} L^2(M, dg)$ to $x^{-(1/2 + \eps)} L^2(M, dg)$, for any $\epsilon > 0$ \cite{Yafaev}. The Schwartz kernels of these operators determines that of the spectral measure of $P_+^{1/2}$, where $P_+ = \indic_{(0, \infty)}(P) \circ P$ denotes the positive part of $P$,  according to Stone's formula
\begin{equation}
dE_{P_+}(\lambda) = \frac{\lambda}{\pi i} \Big( R(\lambda + i0) - R(\lambda - i0) \Big) \, d\lambda, \quad \lambda \geq 0. 
\label{Stone}\end{equation}

\subsection{Asymptotics}
We just mentioned above that, for any positive $\lambda$, the boundary values $R(\lambda \pm i0)$ of the resolvent exist as  bounded operators from $x^{1/2 + \eps} L^2(M, dg)$ to $x^{-(1/2 + \eps)} L^2(M, dg)$, for any $\epsilon > 0$. However, this is not true uniformly down to $\lambda = 0$ \cite{BH2}.  Indeed, the limit $\lambda \to 0$ of the resolvent kernel is a singular limit, which can be seen e.g. from explicit formulae for the resolvent kernel on flat Euclidean space. The outgoing resolvent kernel on $\RR^n$  has (modulo constant factors)  asymptotics
$$
\lambda^{(n-3)/2} e^{i\lambda |z-z'|} |z-z'|^{-(n-1)/2} + O(|z-z'|^{-(n+1)/2}), 
$$
for fixed $\lambda > 0$ and $|z-z'| \to \infty$, and 
$$
|z-z'|^{-(n-2)} + O(\lambda),
$$
for $\lambda \to 0$ and $z, z'$ fixed (provided $n \geq 3$). These asymptotics do not match, and there  is a transitional asymptotic regime in which we send $\lambda \to 0$ while holding $\lambda |z-z'|$ fixed. In the special case of $\RR^n$ the resolvent kernel is given by 
$$
\lambda^{n-2} (\lambda |z - z'|)^{(n-2)/2} \operatorname{Ha}^{1}_{(n-2)/2}(\lambda |z-z'|), \quad z, z' \in \RR^n,
$$
where $\operatorname{Ha}^{1}_{(n-2)/2}$ is the Hankel function of the first kind and order $(n-2)/2$. Thus in this case we can see explicitly the transitional asymptotic regime, interpolating between the oscillatory behaviour of the kernel for positive $\lambda$ and the polyhomogeneous behaviour at $\lambda = 0$. 

In this paper, following \cite{GH1} and more generally Melrose's program \cite{Kyoto}, we analyze the different asymptotic regimes of the resolvent kernel by working on a compactified and blown-up version, denoted\footnote{In this notation, $k$ stands for the parameter $\lambda$. We write $k$ rather than $\lambda$ in this notation to agree with the notation of \cite{GH1}, where the same space was used to construct $(\Delta_g + k^2)^{-1}$, $k \geq 0$.} $\MMksc$,  of the space 
\begin{equation}
M^\circ \times M^\circ \times (0, \lambda_0]
\label{naturaldomain}\end{equation}
 which is the natural domain of definition of the kernel $R(\lambda \pm i0)$ for $0 < \lambda \leq \lambda_0$. The idea is to realize asymptotic regimes geometrically so that each regime corresponds to a boundary hypersurface, and we consider the space \eqref{naturaldomain} to be ``sufficiently blown up" when the resolvent kernel lifts to be conormal at the lifted diagonal and either Legendrian, or polyhomogeneous conormal, at each boundary hypersurface. That means, in particular, that there is nothing ``hidden" at any of the corners, or in other words that if we have two intersecting hypersurfaces $H_1$ and $H_2$, that the expansion at $H_1 \cap H_2$ can be obtained by taking the expansion at $H_1$ and restricting the coefficients, which are functions on $H_1$,  to $H_1 \cap H_2$, or conversely by taking the  expansion at $H_2$ and restricting the coefficients to $H_2 \cap H_1$. In the example above, the expansions for $\lambda \to 0$ for fixed $z,z'$ and at $|z-z'| \to \infty$ for fixed $\lambda$ do not match, and this requires (in our approach) that the corner in between be blown up, to create a hypersurface on which the transitional asymptotics take place. 
 
\subsection{Main results and relation to previous literature}
Expansions of the resolvent as $\lambda \to 0$ were first considered by Jensen-Kato \cite{JK} for Schr\"odinger operators on $\RR^3$ and generalized by Murata \cite{Murata} to general dimension and general constant coefficient operators. More recently, there have been several studies by Wang \cite{XPW}, Bouclet \cite{Bouclet, Bou}, Bony-H\"afner \cite{BH1, BH2} and Vasy-Wunsch \cite{VW} on resolvent estimates (based on commutator estimates and Mourre theory) at low energy, for asymptotically Euclidean or asymptotically conic metrics. 
Wang's paper \cite{XPW} is particularly close in spirit to the current paper, and we discuss it further in Remark~\ref{Wang}.

To describe previous results from \cite{HV1}, \cite{HV2} and \cite{GH1}, we refer to figures \ref{mmkb} and \ref{fig:mmksc} which are  illustrations of the manifolds $\MMksc$ and $M^2_{k,b}$, two blown-up versions of $M \times M \times [0, \lambda_0]$. Here, $\lb$, $\rb$ and $\zf$ are respectively the boundary hypersurfaces of $M \times M \times  [0, \lambda_0]$ corresponding to $x = 0$, $x' = 0$ and $\lambda = 0$ (we use unprimed variables to refer to the left copy of $M$ and primed variables to refer to the right copy of $M$). The other hypersurfaces are created by blowup. Notice that $\zf, \lbo, \rbo$ and $\bfo$ are ``at $\lambda = 0$'', while the others are ``at positive $\lambda$''. 

In \cite{HV1}, \cite{HV2} the boundary value of the resolvent, $R(\lambda \pm i0)$, for fixed $\lambda > 0$, was shown to be the sum of a pseudodifferential operator (in the scattering calculus of Melrose \cite{scatmet}) and a Legendre distribution of a certain specific type, with respect to several Legendre submanifolds associated to the diagonal and to the geodesic flow on $M$, or more precisely to a limiting flow at `infinity'. In terms of the picture in Figure~\ref{fig:mmksc} this means that, on a fixed $\lambda > 0$ slice,  the kernel is oscillatory at the boundaries $\bfc, \lb, \rb$ and can be written as an oscillatory function or oscillatory integral with respect to phase functions determined by geodesic flow on $M$. 

On the other hand, in \cite{GH1} the resolvent $(P + k^2)^{-1}$ was analyzed for real $k \to 0$ on the space $\MMksc$. Because $k$ is in the resolvent set whenever $0 < k < k_1$ where $-k_1^2$ is the largest negative eigenvalue of $P$, the kernel of the resolvent has exponential decay away from the diagonal, and hence vanishes exponentially at the faces $\bfc, \lb$ and $\rb$. However, the rate of exponential decay vanishes as $k \to 0$ and, consequently, the kernel has nontrivial expansions at $\lbo, \rbo$ as well of course at $\zf$ and $\bfo$ (which meet the diagonal), and the focus of \cite{GH1} was the precise analysis of these (polyhomogeneous) expansions. 

The point of the current paper is to unify the two constructions. 
A precise statement of the result is given in Theorem~\ref{mainres}, after definitions of Legendre distributions on the space $\MMkb$ have been given. For now, let us say that a kernel is conormal-Legendrian on the space $\MMkb$  if it lies in the calculus of Legendre distributions given in Section~\ref{Leg}; roughly this means that it is oscillatory at the faces $\bfc, \lb, \rb$ and polyhomogeneous conormal at the other faces, on which $\lambda = 0$. 

\begin{theo}
The boundary value of the resolvent kernel, $R(\lambda \pm i0)$, is the sum of a pseudodifferential operator, i.e.\ a kernel on $\MMksc$, supported close to and conormal to the diagonal $\Delta_{k, sc}$, and a conormal-Legendrian on $\MMkb$. 
\end{theo}

We determine the structure of the spectral measure by subtracting the incoming from the outgoing resolvent. There are two different cancellations that occur when we do this. First the singularity along the diagonal disappears (not surprisingly, since the spectral measure solves an elliptic equation) and secondly there is cancellation in the asymptotic expansion for fixed $z, z' \in M^\circ \times M^\circ$ as $\lambda$ goes to zero. The second cancellation is quite important in applications, such as in understanding the decay of the heat kernel or propagator for long time. 

\begin{theo}\label{mainth} 
The kernel of the spectral projection \eqref{Stone} is conormal-Legendrian on $\MMkb$, and vanishes to order $2\nu_0 + 1$ as $\lambda \to 0$ with $z, z' \in M^\circ$ fixed, where $\nu_0^2$ is the lowest eigenvalue of the operator \eqref{hyp2}. In particular, if $V = 0$ (or even if just $V_0 = 0$), the spectral projection vanishes to order $n-1$ as $\lambda \to 0$ with $z, z' \in M^\circ$ fixed.
More precisely, there is a nontrivial solution $w$ to $Pw = 0$, with $w = O(x^{n/2 - 1 -\nu_0})$ as $x \to 0$, such that the expansion at $\la=0$ is given by
\[dE(\la)= \Big(\la^{2\nu_0+1} w(z) w(z') |dg dg'|^{1/2} + O(\la^{\min(2\nu_0+2,2\nu_1+1)}) \Big)d\la.\]
where $\nu_1^2>\nu_0^2$ is the second eigenvalue of the operator \eqref{hyp2}. Moreover, if $V$ is identically zero, then $w$ is constant. 
\end{theo}
 
 See Theorem~\ref{mainsm} for a more precise statement of this result. 
 
We now give two corollaries of Theorem~\ref{mainth} concerning the long-time behaviour of the wave and Schr\"odinger kernels associated to $P$. 
We write $P_+$ for $\indic_{(0, \infty)}(P) \circ P$. 

 \begin{cor}\label{cor:waves} 
 Let $P = \Delta_g + V$ be as above, and 
 let $\chi \in C_c^\infty(\RR)$, with $\chi(t) \equiv 1$ for $t$ near $0$.
 Let $w$ and $\nu_0, \nu_1$ be as in Theorem~\ref{mainth}. 
Then the 
 solution operators for the wave equation, localized to low energy, satisfy as $t \to \infty$
\begin{equation}
\begin{gathered}
\indic_{(0, \infty)}(P) \chi(P) \frac{\sin(t\sqrt{P_+})}{\sqrt{P_+}}(z,z') = - \Gamma(2\nu_0 + 1) \cos(\pi(\nu_0 + 1)) t^{-(2\nu_0 + 1)} w(z) w(z') \\ + O(t^{-\min(2\nu_0 + 2,2\nu_1+1) }) ,\\
 \indic_{(0, \infty)}(P)   \chi(P) \cos (t \sqrt{P_+})(z,z') = \Gamma(2\nu_0 + 2) \cos(\pi(\nu_0 + 1))  t^{-(2\nu_0+2)}w(z)w(z')\\ + O(t^{-\min(2\nu_0+3,2\nu_1+2)}) .
\end{gathered}\label{waveexp1}
\end{equation}
Notice that the coefficient $ \cos(\pi(\nu_0 + 1))$ vanishes 
when $2(\nu_0+1)$ is an odd integer. 
In particular if $\pl M=S^{n-1}$ and $V_0=0$, then waves decay to order $t^{-(n-1)}$ if $n$ is even and $O(t^{-n})$ is $n$ is odd.  
The implied constant in the remainders  are uniform on compact subsets of $M^\circ \times M^\circ$. Moreover, if $(M,g)$ is nontrapping, then we can remove the energy cutoff $\chi(P)$: the Schwartz kernels of $\indic_{(0, \infty)}(P) \sin(t\sqrt{P_+})/\sqrt{P_+}$ and $\indic_{(0, \infty)}(P) \cos(t \sqrt{P_+})$ are given by the right hand side of \eqref{waveexp1}. \end{cor}

\begin{remark}
This result is closely related to Price's law, which is the statement that waves on a Schwarzschild spacetime, starting with localized initial data, decay to order $t^{-3}$ (outside the event horizon) as $t \to \infty$. This $t^{-3}$ decay was predicted in \cite{Pr1,Pr2} and has been proved recently by  Donninger-Schlag-Soffer \cite{DSS,DSS1} for exact Schwarzschild using separation of variables and by Tataru \cite{Ta} for more general settings. 
Although our result does not apply directly to the Schwarzschild case, it does apply to asymptotically flat manifolds 
which are isometric to Schwarzschild near infinity, or more generally to asymptotically conic manifolds with a `gravitational' type metric at infinity, that is, of the form near $x=0$
\begin{equation}
(1 - 2Mx) \frac{dx^2}{x^4} + \frac{h(x)}{x^2}.
\label{grav}\end{equation}
The case $M \neq 0$ requires a minor extension to the analysis of Section~\ref{lerc} given in \cite[Section 5]{HV2}.\footnote{Note that, in \cite[Section 5]{HV2}, for a metric of the form \eqref{grav}, we have $q_l = q_r = 2M\lambda^2$, so the imaginary powers that show up have an exponent $i\alpha$ with $\alpha = O(\lambda)$ vanishing at $\bfo$ and $\zf$.} 
\end{remark}

The corresponding result for the Schr\"odinger propagator $e^{itP_+}$ is as follows:

 \begin{cor}\label{propagators} 
The Schwartz kernel of the propagator $e^{itP_+}$, localized to low energy, satisfies
\begin{equation}
\indic_{(0, \infty)}(P)  \chi(P) e^{itP_+}(z,z') =  C t^{-(\nu_0 + 1)} w(z) w(z') + O(t^{-\min(\nu_0 + 3/2,\nu_1+1) }) , \quad t \geq 1,
  \label{propexp}
\end{equation}
for some $C\not=0$.  The implied constant in the remainder term is uniform on compact subsets of $M^\circ \times M^\circ$. Moreover, if $(M,g)$ is nontrapping, then we can remove the energy cutoff $\chi(P)$: the Schwartz kernel of $\indic_{(0, \infty)}(P) e^{itP_+}$ is given by the right hand side of \eqref{propexp}.
 \end{cor}
 
\begin{remark} 
Indeed, using results of \cite{HW}, if the metric $g$ is nontrapping, then the propagator localized away from low energy satisfies
$$
\big| (\Id - \chi(P)) e^{itP_+}(z,z') \big| = O(t^{-\infty}), \quad t \to \infty. 
$$
Intuitively this is because when $z, z'$ are fixed and $t \to \infty$, then no signal starting at $z$ can end at $z'$ when there is a lower bound on the velocity. The same is true with $e^{itP_+}$ replaced by either of the wave solution operators. 
We also remark that, instead of localizing the variables $(z,z')$ in compact sets, we could work instead in appropriately weighted Sobolev spaces. 
\end{remark}

\begin{remark}\label{Wang} X. P. Wang's paper \cite{XPW} is quite close in spirit to the present paper. He also studies manifolds with (exactly) conical ends, and derives low energy asymptotics for the resolvent, as well as large time expansions for the propagator similar to the Corollaries above. In fact, his results are more general than ours in some respects, as he treats higher order asymptotic terms, and also allows zero modes and zero resonances. On the other hand, our results are more complete in that we consider expansions at all boundary hypersurfaces of $\MMksc$, while Wang only considers (in our terminology) expansions at the $\zf$ boundary hypersurface. Our expansions are also more explicit: for example, it does not seem easy to see from \cite{XPW} that the leading asymptotic for the propagator is a rank one operator (under our assumptions), as in \eqref{propexp}.  
\end{remark}

 Corollaries~\ref{cor:waves} and \ref{propagators} only use the expansion of the spectral measure at the $\zf$ face of $\MMkb$. 
In the sequel, \cite{GHS2}, to this paper we shall prove the following 
consequences of Theorem~\ref{mainth} that exploit the full regularity of the spectral measure, in particular its Legendrian nature at the ``positive $\lambda$" boundary hypersurfaces:  
 
 \begin{itemize}
  \item
For any $\lambda_0 > 0$ there exists a constant $C$ such that the generalized spectral projections $dE(\lambda)$
for $\sqrt{\Delta}$ satisfy
\begin{equation}
\| dE(\lambda) \|_{L^p(M) \to L^{p'}(M)} \leq C \lambda^{n(1/p - 1/p') - 1}, \quad 0 \leq \lambda \leq \lambda_0
\label{restr}\end{equation}
for $1 \leq p \leq 2(n+1)/(n+3)$. Moreover, if $(M,g)$ is nontrapping, then there exists $C$ such that \eqref{restr} holds for all $\lambda > 0$ and the same range of $p$. 

\item
Assume that $F\in C_c(0, T)$   and that for some $s>\max\{n(1/p-1/2),1/2\}$
$$
\|F\|_{H^s}< \infty,
$$
where $H^s$ is a Sobolev space of order $s$. Then there exists $C$ depending only on $T$, $p$ and $s$  such that
\begin{equation}
\|F(\sqrt{\Delta})\|_{p\to p} \leq  C\|F\|_{H^s}  .
\label{specmult}\end{equation}
Moreover, if $(M,g)$ is nontrapping, then \eqref{specmult} can be improved to 
\begin{equation*}
\sup_{t > 0} \|F(t\sqrt{\Delta})\|_{p\to p} \leq  C\|F\|_{H^s}  .
\end{equation*}
\end{itemize}


\section{Geometric Preliminaries}
\subsection{The spaces $\MMkb$ and $\MMksc$}\label{5.1}

The construction of the Schwartz kernel of $(\Delta-\la^2)^{-1}$ takes place on a desingularized version of 
the manifold $[0,1]\x M\x M$ where $[0,1]$ is the range of the spectral parameter $\la$. In geometric terms, this 
corresponds to an iterated sequence of blow-up of corners of $[0,1]\x M\x M$; it was introduced in \cite{MS} and heavily
used in \cite{GH1}. For the convenience of the reader we recall quickly its definition but we refer to Section 2.2 in \cite{GH1} for 
a detailed description of this manifold. 
We denote by $[X;Y_1,\dots,Y_N]$ the iterated real blow-up of $X$ around 
$N$ submanifolds $Y_i$ if $Y_1$ is a p-submanifold, the lift of $Y_2$ to $[X; Y_1]$ is a p-submanifold, and so on. 
We shall denote by $\rho_H$ an arbitrary boundary defining function for a boundary hypersurface $H$ of $X$. 
We now define the space $\MMksc$. Consider in $[0,1]\x M\x M$ the codimension 3 corner $C_3:=\{0\}\x\pl M\x\pl M$ 
and the codimension $2$ corners 
\[C_{2,L}:=\{0\}\x\pl M\x  M, \quad C_{2,R}:=\{0\}\x M\x\pl M,\quad C_{2,C}:=[0,1]\x\pl M\x\pl M.\]
We consider first the blow-up 
\[M_{k,b}^2:=\big[[0,1]\x M\x M; C_3, C_{2,R},C_{2,L},C_{2,C}\big]\]
with blow-down map $\beta_b:M^2_{k,b}\to [0,1]\x M\x M$. 
We have $7$ faces on $M^2_{k,b}$, the right, left, and zero faces
\[\rb=\mathrm{clos }\beta_b^{-1}([0,1]\x M\x\pl M),\quad 
\lb:=\mathrm{clos }\beta_b^{-1}([0,1]\x\pl M\x M),\] 
\[\zf:=\mathrm{clos }\beta_b^{-1}(\{0\}\x M\x M),\]
the `b-face' (so-called because of its use in the b-calculus) $\bfc:=\mathrm{clos }\beta_b^{-1}(C_{2,C}\setminus C_3)$, and the three faces corresponding to $\bfc, \rb, \lb$ at zero energy:
\[\bfo:=\beta_b^{-1}(C_3), \quad 
\rbo:=\mathrm{clos }\beta_b^{-1}(C_{2,R}\setminus C_3), \quad 
\lbo:=\mathrm{clos }\beta_b^{-1}(C_{2,L}\setminus C_3).\]
\begin{figure}[ht!]
\begin{center}
\input{Mk2b2.pstex_t}
\caption{The manifold $M^2_{k,b}$. The arrows show the direction in which the indicated function increases from $0$ to $\infty$.}
\label{mmkb}
\end{center}
\end{figure}
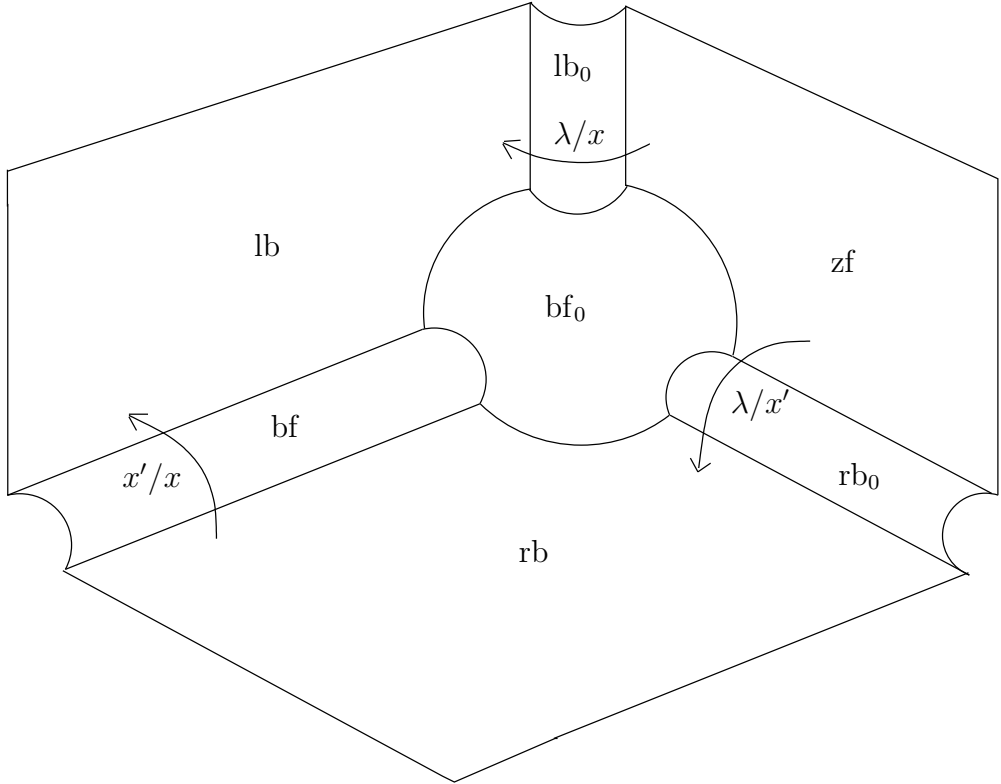
The closed lifted diagonal $\Delta_{k,b}=\mathrm{clos }\beta_b^{-1}([0,1]\x \{(m,m);m\in M^\circ\})$
intersects the face $\bfc$ in a p-submanifold denoted $\pl_{\bfc}\Delta_{k,b}$. We then define
the final blow-up
\begin{equation}
M^2_{k,\sca}:=\big[M^2_{k,b}; \pl_{\bfc}\Delta_{k,b}\big],
\label{scblowup}\end{equation}
and denote the new boundary hypersurface created by this blowup $\sca$, for `scattering face'. 

\begin{figure}[ht!]
\begin{center}
\input{M2ksc.pstex_t}
\caption{The manifold $M^2_{k,\sca}$; the dashed line is the boundary of the lifted diagonal $\Delta_{k, \sca}$}
\label{fig:mmksc}
\end{center}
\end{figure}
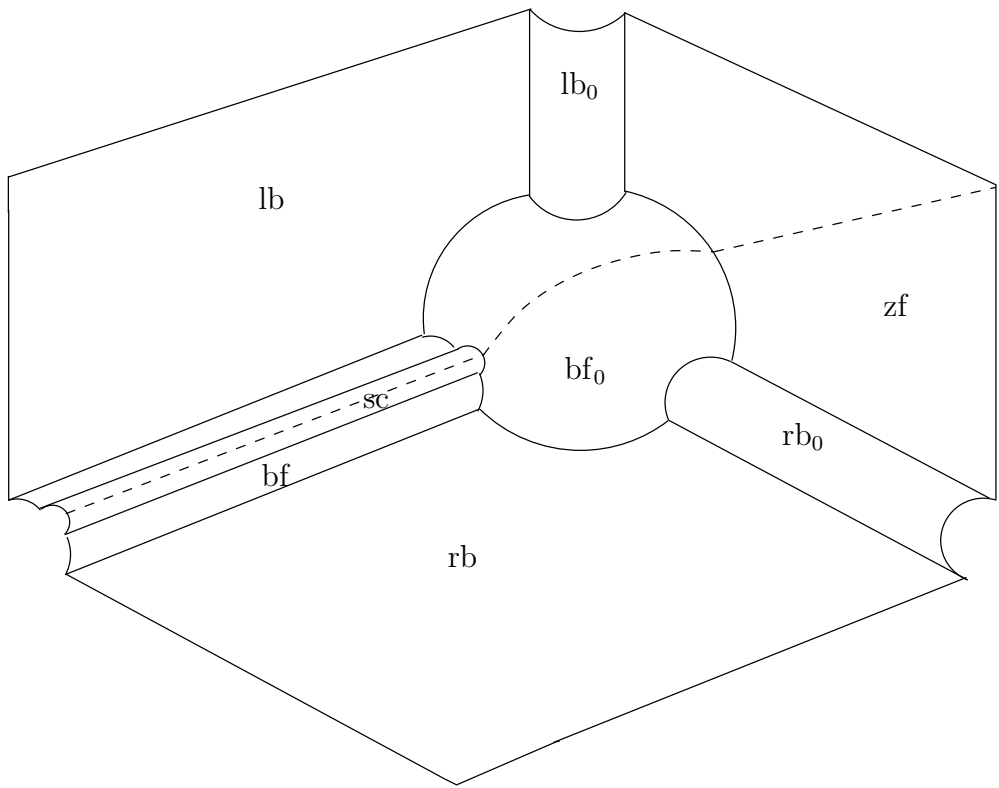

\subsection{Polyhomogeneous conormal functions and index sets}\label{sect:phg}
Below we use spaces of polyhomogeneous conormal functions. These are
defined on any manifold with corners $X$. Let  $\mc{F}$ denote its set of boundary hypersurfaces. An  \emph{index family} $\mathcal{E}$ consists of a subset $\mathcal{E}_H$ of $\CC \times \NN$ (an \emph{index set}) for each $H$ in the set  $\mc{F}$ of boundary hypersurfaces of $X$, satisfying two conditions: (a) for each $K \in \RR$, the number of points $(\beta, j) \in \mc{E}_H$ with $\Re \beta \leq K$ is finite and (b)  if $(\beta, j) \in \mc{E}_H$ then $(\beta + 1, j)\in \mc{E}_H $ and if $j > 0$ then also $(\beta, j-1) \in \mc{E}_H$. Then the space of polyhomogeneous conormal functions with index family $\mathcal{E}$, denoted $\mathcal{A}_\mcE(X)$, is the space of functions $f$ that are smooth in the interior of $X$ and possess expansions in powers and logarithms of the form 
$$
f = \sum_{{z,p} \in \mc{E}_H \text{ s.t. } \Re z \leq s} a_{(z,p)} \rho_H^z (\log \rho_H)^p +O(\rho_H^s)
$$
where $\rho_H$ is a boundary defining function for the boundary hypersurface $H$. See \cite{GH1} or \cite{cocdmc} for a precise definition. Condition (a) on the index set ensures that the sum on the left hand side is finite, and condition (b) ensures that the form of the sum is independent of the choice of local coordinates. 

Let us recall from \cite{cocdmc} the operations of addition and extended union on two index sets $E_1$ and $E_2$, denoted $E_1 + E_2$ and $E_1 \extunion E_2$ respectively:
\begin{equation}\begin{gathered}
E_1 + E_2 = \{ (\beta_1 + \beta_2, j_1 + j_2) \mid (\beta_1, j_1) \in E_1 \text{ and } (\beta_2 , j_2) \in E_2 \} \\
E_1 \extunion E_2 = E_1 \cup E_2 \cup \{ (\beta, j) \mid \exists (\beta, j_1) \in E_1, (\beta, j_2) \in E_2 \text{ with } j = j_1 + j_2 + 1 \}.
\end{gathered}\end{equation}

We write $q$ for the index set 
\begin{equation}
\{ (q + n, 0) \mid n = 0, 1, 2, \dots \}
\label{short}\end{equation} for any $q \in \RR$; note that in this notation, $0$ denotes the $\CI$ index set $\NN = \{ (0,0), (1,0), (2,0), \dots \}$. 
For any index set $E$ and $q \in \RR$, we write $E \geq q$ if $\Re \beta \geq q$ for all $(\beta, j) \in E$ and if $(\beta, j) \in E$ and $\Re \beta = q$ implies $j = 0$. We write $E > q$ if there exists $\epsilon > 0$ so that $E \geq q + \epsilon$. We shall say that $E$ is integral if $(\beta, j) \in E$ implies that $\beta \in \mathbb{Z}$, and one-step if $E$ is such that $E = E' + (\alpha,0)$ for some $\alpha \in \CC$ and some integral index set $E'$. We write 
$$\min E = \min \{ \beta \mid \exists \ (\beta, j) \in E \}.
$$
We  say that $E'$ is a logarithmic extension of $E$ if $E \subset E'$ and if $(\beta, j) \in E'$ implies that $(\beta, 0) \in E$.

\subsection{Compressed cotangent bundle}\label{sect:compcotbundle}
We define a `compressed tangent bundle' and a `compressed cotangent bundle' on $\MMkb$. To define the compressed tangent bundle, denoted $\Tkb \MMkb$, it is helpful first to define the single space version. Thus we define the single space $\Mkb$ to be 
$$
\Mkb = \big[ M \times [0, \lambda_0]; \partial M \times \{ 0 \} \big],
$$
that is, $M \times [0, \lambda_0]$ with the corner $\partial M \times \{ 0 \}$ blown up. (We always ignore the boundary at $\lambda = \lambda_0$.) We denote the boundary hypersurfaces of $\Mkb$ by $\bfc$, the `b-face', the lift of $\partial M \times [0, \lambda_0]$; $\zf$, the `zero face', the lift of $M \times \{ 0 \}$, and  $\ff$, the `front face', created by the blowup. 

Recall that the space of scattering vector fields on $M$, denoted $\Vsc(M)$, are those of the form $x W$, where $x$ is a boundary defining function for $\partial M$ and $W$ is tangent to the boundary (i.e.\ is a b-vector field, $W \in \Vb(M)$). Let $\rho_{\ff}$ denote a boundary defining function for $\ff \subset \Mkb$. 
We define the space (and Lie algebra) $\mathcal{V}_{k,b}(\Mkb)$ to be those smooth vector fields 
generated by $\rho_{\ff}^{-1}$ times the lift of $\Vsc(M)$ to $\Mkb$, together with the vector field $\lambda \partial_\lambda + W$, where $W \in \Vb(M)$ is equal to $ x \partial_x$ near $\partial M$. Thus near $\bfc \cap \ff$, using coordinates\footnote{We use $\rho$ to denote $x/\lambda$, and $r = \lambda/x$ throughout this paper.} $(\rho = x/\lambda, y, \lambda)$, such vector fields are smooth $\CI(\Mkb)$-linear combinations of 
\begin{equation}
\rho^2 \partial_\rho,\quad \rho \partial_{y_i}, \quad \lambda \partial_\lambda \Big|_{\rho, y},
\end{equation}
while near $\ff \cap \zf$, using coordinates $(x, y, r = \lambda/x)$, 
such vector fields are smooth $\CI(\Mkb)$-linear combinations of 
\begin{equation}
x \partial_x \Big|_{r, y}, \quad \partial_{y_i}, \quad r \partial_{r}  
\end{equation}
which in particular restrict to $\zf$ to give the b-vector fields on $M$. 
This definition is independent of coordinates. These vector fields define a bundle $\Tkb \Mkb$, whose space of smooth sections is precisely $\mathcal{V}_{k,b}(\Mkb)$. 

We observe some geometric properties of $\Mkb$ and $\Tkb \Mkb$:
\begin{itemize}
\item
To see why these vector fields are chosen, note that both $H$ and $\lambda^2$ vanish to second order at $\ff$, in terms of $\Tkb\Mkb$. 
It is natural, then, to divide the operator by a factor of $\rho_{\ff}^2$; 
we find that $\rho_{\ff}^{-2}(H - \lambda^2)$, which we can take to be $\lambda^{-2} H - 1$ near $\bfc$ and $x^{-2} H - (\lambda/x)^2$ near $\zf$, is `built' out of an elliptic combination of sections of $\Tkb\Mkb$.  Moreover, for Euclidean space, we have
$[\lambda\partial_\lambda + x\partial_x, \Delta - \lambda^2 ] = 2 (\Delta - \lambda^2)$, which shows that the resolvent and spectral measure are (in the Euclidean case) homogeneous with respect to this vector field. 
\item A scattering metric \eqref{metricconic} on $M$ blows up at $\ff$ to second order. If we multiply by $\lambda^2$ and restrict to $\ff$, then it is easy to check that we get the \emph{exact} conic metric 
\begin{equation}
\frac{d\rho^2}{\rho^4} + \frac{h(0)}{\rho^2}.
\label{exactconicmetric}\end{equation}
Thus a scattering metric on $M$ induces an exact conic structure on $\ff$. 

\item At $\zf$, we have the Lie algebra of b-vector fields on $M$, but for a fixed positive $\lambda$, the vector fields tangent to $M \times \{ \lambda \}$ are the \emph{scattering} vector fields. Thus this Lie algebra interpolates between the b-calculus at $\lambda = 0$ and the scattering calculus for positive $\lambda$. This Lie Algebra was `microlocalized' to a calculus of operators in \cite{GH1}, with kernels defined on $\MMksc$. This calculus interpolates between the b-calculus at $\lambda = 0$ and the scattering calculus for fixed positive $\lambda$. 
\end{itemize}

Now to define $\Tkb \MMkb$, we note that there are stretched projections $\pi_L, \pi_R : \MMkb \to \Mkb$; this can be proved 
by noting that $\MMkb$ can be constructed from $\Mkb \times M$
by blowing up $\ff \times \partial M$, $\ff \times M$ and $\bfc \times \partial M$ using Lemma 2.7 of \cite{asatet}. We define $\Tkb \MMkb$ to be that vector bundle generated over $\CI(\MMkb)$ by 
\begin{equation}
\pi_L^*(\rho_{\ff}^{-1} \Vsc(M)), \quad \pi_R^*(\rho_{\ff}^{-1} \Vsc(M)) \text{ and }\lambda \partial_\lambda + x \partial_x + x' \partial_{x'}. 
\label{VkbMMkb}\end{equation}

It is straightforward to check that the $\CI(\MMkb)$-span of these vector fields is closed under Lie bracket. 
We denote this Lie Algebra by $\Vkb$. 

We now define the compressed cotangent bundle $\Tkbstar \MMkb$. This is the dual bundle to $\Tkb \MMkb$. Near $\bfc$ and $\bfo$, but away from $\zf$, a basis of sections is given by singular one-forms of the form 
\begin{equation}
\frac{d\rho}{\rho^2}, \quad \frac{d\rho'}{{\rho'}^2}, \quad \frac{ dy_i}{\rho}, \quad \frac{dy'_i}{\rho'}, \quad \frac{d\lambda}{\lambda}
\label{Tkbstar-basis}\end{equation}
where primed variables are coordinates on the right factor of $M$, and unprimed variables on the left factor of $M$ (lifted to $\MMkb$);
near $\bfo$ and $\zf$, but away from $\bfc$, a basis is given by 
\begin{equation}
\frac{dx}{x}, \quad \frac{dx'}{{x'}}, \quad {dy_i}, \quad dy'_i, \quad \frac{d\lambda}{\lambda};
\end{equation}
and near $\zf$, and away from other boundary hypersurfaces, 
\begin{equation}
dz_i , \quad dz'_i, \quad \frac{d\lambda}{\lambda},
\end{equation}
where $z = (z_1, \dots, z_n)$ are local coordinates on the interior of $M$. 
Therefore, any point in $\Tkbstar \MMkb$ can be written 
\begin{equation}
\nu \frac{d\rho}{\rho^2}  +\nu' \frac{d\rho'}{{\rho'}^2}+ \mu_i \frac{dy_i}{\rho} + \mu'_i \frac{dy'_i}{\rho'} + T \frac{d\lambda}{\lambda}
\label{T}\end{equation}
in the first region, 
$$
\tau \frac{dx}{x} + \tau' \frac{dx'}{{x'}} + \eta_i {dy_i} + \eta'_i dy'_i +  T' \frac{d\lambda}{\lambda}
$$
in the second region, and 
$$
\zeta_i dz_i + \zeta'_i z'_i + T'' \frac{d\lambda}{\lambda}
$$
in the third region. These expressions define local coordinates on $\Tkbstar \MMkb$ in each region.

\begin{remark} Here the coordinates $\mu, \nu,  \mu', \nu'$ have a different meaning to that used in \cite{HV2}, due to the scaling in $\lambda$, since for example here $\mu_i$ is dual to $\lambda dy_i/x$ rather than $dy_i/x$. It is similar to how in the semiclassical calculus, the variable $\eta$ is dual to $dy/h = \lambda dy$ rather than $dy$, so the frequency corresponding to $\eta$ scales as $\lambda$. However, here we have the opposite situation in that $\lambda \to 0$, rather than infinity, so in a sense we are giving a meaning to the `semiclassical calculus with $h \to \infty$'!
\end{remark}

\subsection{Densities}\label{sect:densities}
We define the compressed density bundle $\Omegakb(\MMkb)$ to be that line bundle whose smooth nonzero sections are given by the wedge product of a basis of sections for $\Tkbstar(\MMkb)$. For example, near $\bfc \cap \bfo \cap \Diagb$ this takes the form (using \eqref{Tkbstar-basis})
\begin{equation}
\Big| \frac{d\rho d\rho' dy dy' d\lambda}{\rho^{n+1} {\rho'}^{n+1} \lambda} \Big| \sim \lambda^{2n} \Big| \frac{dg dg' d\lambda}{\lambda} \Big|
\end{equation}
where $dg$, resp. $dg'$ denotes the Riemannian density with respect to $g$, lifted to $\MMkb$ by the left, resp. right projection; 
near $\zf$, a smooth nonzero section takes the form 
\begin{equation}
\Big| \frac{dx dx' dy dy' d\lambda}{x x' \lambda} \Big| \sim \Big| \frac{dg_b dg_b' d\lambda}{\lambda} \Big|
\end{equation}
where $dg_b$ is the Riemannian density with respect to the b-metric $$g_b = x^2 g$$ on $M$; near $\lb \cap \lbo$, a smooth nonzero section takes the form 
\begin{equation}
\Big| \frac{d\rho dy dx' dy' d\lambda}{\rho^{n+1} x' \lambda} \Big| \sim \lambda^n \Big| \frac{dg dg_b' d\lambda}{\lambda} \Big|,
\label{lblbo}\end{equation}
and so on.

\begin{remark} This differs from the density bundle used in \cite{GH1}. The density bundle defined here is more convenient; for example, it absorbs the $\rho_{\sca}^{n/2}$ factors put in `by hand' in Definition 2.7 of \cite{GH1}.
\end{remark}

\subsection{Fibrations and contact structures}\label{facs}
The Lie algebra $\Vkb$ gives rise to a fibration at each  boundary hypersurface $\bullet$ of $\MMkb$, in the following way: the leaves of this fibration at $\bullet$ are precisely the maximal submanifolds on which the restriction of $\Vkb$ to $\bullet$ is transitive. 
These fibrations are trivial (and therefore unimportant) on $\bfo, \lbo, \rbo$ and $\zf$, i.e.\ the Lie algebra is transitive on these faces. We now describe the fibration at the remaining boundary hypersurfaces, namely $\bfc$, $\lb$ and $\rb$. 

At $\bfc$, the Lie Algebra restricted to this face is given by multiples of $\lambda \partial_\lambda$, and therefore, the fibration is given by projection off the $\lambda$ factor. That is, in local coordinates $(y, y', \sigma, \lambda)$ on this face, where\footnote{We use the notation $\sigma = x/x' = \rho/\rho'$ throughout the paper} $\sigma = x/x'$, the fibration takes the form 
$$
(y, y', \sigma, \lambda) \mapsto (y, y', \sigma).
$$
Thus $Z_{\bfc}$ is the base of this fibration, i.e.\ $Z_{\bfc}$ is the lift of the corner $(\partial M)^2$ to $M^2_b = [M^2; (\partial M)^2]$ (which can be identified with $\bfc \cap \bfo$ in Figure~\ref{mmkb}).  

At $\lb$, the Lie algebra restricts to the span of vector fields $\la\pl_\la,\pl_{z'}$. Hence the fibration takes the form 
$$
(y, z', \lambda) \mapsto y.
$$
Similarly at $\rb$  the fibration takes the form 
$$
(z, y', \lambda) \mapsto y'.
$$
We let $Z_{\lb} = Z_{\rb} = \partial M$ denote the base of these fibrations. 

In the interior of $\MMkb$, the compressed cotangent bundle is canonically isomorphic to the usual cotangent bundle, and hence the canonical symplectic form on $T^*(\MMkb)$ induces a canonical symplectic form $\omega$ on $\Tkbstar \MMkb$. In turn, $\omega$ induces a contact structure at each boundary hypersurface $\bullet$, where $\bullet = \bfc, \lb$, or $\rb$. In fact, the contact structure lives on a bundle over $Z_{\bullet}$, denoted $\Nsfstar Z_\bullet$ defined in \cite{HV1}, which we recall here. Here and below we use $\bullet$ to denote one of $\bfc, \lb$ or $\rb$. 

There is a subbundle of $\Tkb_\bullet \MMkb$ whose fibre at $p \in \bullet$ consists of the span of those vector fields that vanish (as an element of $T_p \MMkb$) at $p$. (This is not the trivial subbundle, since a vector field can vanish as an element of $T_p \MMkb$ while being nonzero as an element of $\Tkb_p \MMkb$; for example, $x^2 \partial_x$ is nonzero as an element of $\Tkb \MMkb$ at $\lb$.)
Dually, we define the annihilator subbundle
$\Tkbstar(F; \bullet)\MMkb$, a subbundle of $\Tkbstar_\bullet \MMkb$. Here the $F$ stands for `fibre'. The quotient
bundle, $\Tkbstar_\bullet \MMkb / \Tkbstar(F; \bullet)\MMkb$, turns out  to be the lift of a bundle $\Nsfstar Z_\bullet$ over $Z_\bullet$  to $\bullet$ \cite{HW}. For example, at $\bullet = \bfc$, the vector fields vanishing as an element of $T_p \MMkb$, $p \in \bfc$, are spanned by all but the last vector field in \eqref{VkbMMkb}, the annihilator subbundle is spanned by $d\lambda/\lambda$, and the quotient bundle is spanned by the  remaining elements of \eqref{Tkbstar-basis}. Thus, local coordinates on $\Nscstar Z_\bfc$ are $(y, y', \sigma; \nu, \mu, \nu', \mu')$. Similarly, for $\bullet = \lb$, the annihilator subbundle is spanned by $\lambda dz'_i$ and $d\lambda/\lambda$, and the fibres of $\Nscstar Z_{\lb}$ are spanned by 
$$
\frac{d\rho}{\rho^2} \, , \ \frac{dy_j}{\rho} \, , \ 1 \leq j \leq n-1,
$$
showing that $(y, \nu, \mu)$ furnish local coordinates on $\Nscstar Z_{\lb}$. 

Note that $Z_\bfc$ is a manifold with boundary; we denote its two boundary hypersurfaces by $\partial_{\lb} Z_\bfc$ and $\partial_{\rb} Z_\bfc$. Similarly, 
$\Nsfstar Z_\bfc$ is a manifold with boundary, with boundary hypersurfaces $\partial_{\lb} \Nsfstar Z_\bfc$ and $\partial_{\rb} \Nsfstar Z_\bfc$. There is a fibration $\phi_{\bfc, \lb}$ from $\partial_{\lb} Z_\bfc$ to $Z_{\lb}$ given in local coordinates by $(y, y') \mapsto y$. Similarly there is an induced fibration $\phit_{\bfc, \lb}$ from $\partial_{\lb} \Nsfstar Z_\bfc$ to $\Nsfstar Z_{\lb}$ given in local coordinates by $(y, y', \nu, \mu, \nu', \mu') \mapsto (y, \nu, \mu)$. There are of course analogous fibrations at the right boundary $\rb$. 

Next we show how $\omega$ induces a contact structure on $\Nsfstar Z_\bullet$. Contracting $\omega$ with $\rho_{\bullet}^2 \partial_{\rho_{\bullet}}$, where $\rho_{\bullet}$ denotes a boundary defining function for $\bullet$, and restricting to $\bullet$ gives a one-form on $\Tkbstar_{\bullet} \MMkb$ that is well-defined up to scalar multiples. It can be checked that it induces a form on $\Nsfstar Z_{\bullet}$ that is nondegenerate in the sense of contact geometry (at least in the interior of $\Nsfstar Z_{\bullet}$), and therefore determines a well-defined contact structure on $\Nsfstar Z_{\bullet}$. 

In the case $\bullet = \bfc$, we compute that $\omega$ is given by 
\begin{equation}\begin{gathered}
\omega = d \Big( \nu \frac{d\rho}{\rho^2}  +\nu' \frac{d\rho'}{{\rho'}^2}+ \mu_i \frac{dy_i}{\rho} + \mu'_i \frac{dy'_i}{\rho'} + T \frac{d\lambda}{\lambda} \Big) \\
= d\nu \wedge \frac{d\rho}{\rho^2}  +d\nu' \wedge\frac{d\rho'}{{\rho'}^2}
+ d\mu_i \wedge\frac{\lambda dy_i}{\rho} - \mu_i dy_i \wedge\frac{d\rho}{\rho^2} \\ +  \ 
d \mu'_i \wedge\frac{\lambda dy'_i}{\rho'} - \mu'_i dy'_i \wedge\frac{d\rho'}{{\rho'}^2}
 + dT \wedge\frac{d\lambda}{\lambda}.
\end{gathered}\end{equation}
We can use $\rho \partial_\rho + \rho' \partial_{\rho'}$ (where these are lifted from the left and right factors respectively) as a b-normal vector field at $\bfc$, and then using $\rho$ as a boundary defining function we obtain 
$$
\iota_{\rho(\rho \partial_\rho + \rho' \partial_{\rho'})}  \omega = 
\mu_i dy_i -d\nu + \sigma (\mu_i' dy_i'  - d\nu')
$$
as the contact 1-form on $\Nsfstar Z_{\bfc}$. This is precisely the same 
contact form that one gets from the manifold $\MMkb$ at a fixed energy level, as was done in \cite{HV2}. 

This contact 1-form degenerates at $\partial_{\lb} \Nsfstar Z_\bfc$: the contact form becomes degenerate in the fibre directions of $\phit_{\bfc, \bullet}$ but remains nondegerate in the `base' directions. In fact, over $\partial_{\lb} \Nsfstar Z_{\bfc}$, the contact 1-form is the lift of the contact 1-form on $\Nsfstar Z_{\lb}$ (given in local coordinates by $\mu_i dy_i - d\nu$) with respect to this fibration. Moreover, the fibres of $\phit_{\bfc, \lb}$ have a natural contact structure given in local coordinates by $\mu'_i dy'_i - d\nu'$. Of course, similar statements are true at the intersection with $\rb$. 
This is all explained in more detail in \cite{HV1} and \cite{HW}.


\section{Legendre distributions on the space $\MMkb$}\label{Leg}

\subsection{Legendre submanifolds}\label{sec:legsub}
We recall from \cite{HV1} definitions concerning Legendre submanifolds of $\Nsfstar Z_{\bfc}$. Let $n = \dim M$, so that $\dim \Nsfstar Z_{\bfc} = 4n-1$. We define a Legendre submanifold $\Lambda$ of $\Nsfstar Z_{\bfc}$ to be a smooth submanifold of dimension $2n-1$ such that
\begin{itemize}
\item the contact form vanishes on $\Lambda$.

\item $\Lambda$ is transversal to  $\partial_{\lb} \Nsfstar Z_\bfc$ and $\partial_{\rb} \Nsfstar Z_\bfc$, and therefore is a smooth manifold with boundary. The boundary hypersurfaces $\Lambda \cap \partial_{\lb} \Nsfstar Z_\bfc$ and $\Lambda \cap \partial_{\rb} \Nsfstar Z_\bfc$ will be denoted $\partial_{\lb} \Lambda$, respectively $\partial_{\rb} \Lambda$.
\end{itemize}

\begin{remark} As shown in \cite[Section 4]{HW}, this definition is equivalent to the more elaborate definition given in \cite{HV1}, \cite{HV2}.
\end{remark}

We also recall two definitions concerning two Legendre submanifolds that intersect. The first applies away from $\lb$ and $\rb$. 
Suppose that 
$\Lambda_0$ and $\Lambda$ are two smooth  Legendre submanifolds that intersect cleanly in a submanifold of dimension $2n-2$ (disjoint from $\partial_{\lb} \Nsfstar Z_\bfc$ and $\partial_{\rb} \Nsfstar Z_\bfc$). In that case, each submanifold divides the other into two parts. Let $\Lambda_+$ and $\Lambda_-$ denote the two pieces of $\Lambda_1$. Then both $(\Lambda_0, \Lambda_+)$ and $(\Lambda_0, \Lambda_-)$ are said to form an intersecting pair of Legendre submanifolds. 

The second definition concerns two Legendre submanifolds $\Lambda^\sharp$ and $\Lambda$ where $\Lambda^\sharp$ is smooth, and $\Lambda$ is smooth except at $\Lambda^\sharp$ where it has a conic singularity. We say that $(\Lambda, \Lambda^\sharp)$ form an intersecting pair of Legendre submanifolds with conic points if
\begin{itemize}
\item $\Lambda^\sharp$ projects diffeomorphically to the base $\bfc$ and does not meet the zero section of $\Nsfstar Z_\bfc$ (so that  $\spn \Lambda^\sharp = \{ tq \mid q \in \Lambda^\sharp, t \in \RR \}$ is a submanifold of dimension $2n$);

\item the lift $\hat \Lambda$ of $\Lambda$ to the blown-up manifold
\begin{equation}
\big[ \Nsfstar Z_{\bfc}; \ \spn \Lambda^\sharp \big]
\label{lambdah}\end{equation}
is a smooth submanifold that meets the boundary hypersurfaces of \eqref{lambdah} (in particular, the lift of $\spn \Lambda^\sharp$) transversally.

\end{itemize}
All this is explained in more detail in \cite{HV1} and \cite{HW}. 

In subsequent sections of this paper there are three Legendre submanifolds of particular interest: the boundary (at $\bfc$) of the conormal bundle to the diagonal $\Diagb$, which following \cite{HV1}, \cite{HV2} we denote $\Nscstar \Diagb$; the `propagating Legendrian' $L^{\bfc}$; and the incoming/outgoing Legendrian $L^\sharp_\pm$. We now define and describe these three submanifolds.

The conormal to the diagonal $\Nscstar \Diagb$ is easy to describe: it is the Legendre submanifold of $\Nsfstar Z_{\bfc}$ given in local coordinates $(\sigma, y, y', \nu, \nu', \mu, \mu')$ by $\sigma = 1, y=y', \nu = -\nu', \mu = -\mu'$. Analytically, it is related to pseudodifferential operators on $\MMkb$ of (differential) order $-\infty$: kernels on $\MMksc$ that are smooth and rapidly vanishing at $\bfc$ are conormal to $\Nsfstar \Diagb$ when viewed on $\MMkb$. 

The incoming/outgoing Legendrian $L^\sharp_\pm$ is also easy to describe in local coordinates: it is given by $\nu = \nu' = \pm 1, \mu = \mu' = 0$. It is clear that this projects diffeomorphically to the base $Z_{\bfc}$. Analytically this corresponds to pure incoming/outgoing oscillations $a(y,y',\sigma) e^{\pm i\lambda/x} e^{\pm i\lambda/x'} = a(y,y',\sigma) e^{\pm i/\rho} e^{\pm i/\rho'}$. We write $L^\sharp = L^\sharp_+ \cup L^\sharp_-$.

The propagating Legendrian $L^{\bfc}$ is more interesting and geometrically intricate. It is related to the limit at $\partial M$ of geodesic flow on $M$ or, what is the same thing, the Hamilton flow of the symbol of $P - \lambda^2$ at $\bfc$. Since this occurs purely at $x=x' = 0$ (where the potential vanishes) and the metric $g$ is asymptotically conic, it is related to geodesic flow on an exact conic metric. 

One way to describe $L^{\bfc}$ is to start with the intersection of $\Nscstar \Diagb$ and the characteristic variety of $P - \lambda^2$, which is the submanifold $\{ \sigma = 1, y=y', \nu = -\nu', \mu = -\mu'; \nu^2 + |\mu|^2_h = 1 \}$, and take the flowout by the (rescaled) Hamilton vector field associated to the operator $P - \lambda^2$ acting on either the left or the right variables. The Hamilton vector field for this operator vanishes to first order at $\Tkbstar_{\bfc} \MMkb$; after dividing by $\rho$, we obtain a nonzero vector field on $\Tkbstar_{\bfc} \MMkb$ that descends to a contact vector field on $\Nsfstar Z_{\bfc}$.  In local coordinates, the symbol of $P - \lambda^2$ acting on the left is 
\begin{equation}
\sigma_l(P - \lambda^2) = \nu^2 + h - \lambda^2, \quad h = \sum_{i,j} h^{ij}(y) \mu_i \mu_j,
\label{sigmal}\end{equation}
and  the left Hamilton vector field takes the form (after dividing by $\rho$)
\begin{equation}
V_l = -\nu (\sigma \dbyd{}{\sigma} + \mu \dbyd{}{\mu}) + h \dbyd{}{\nu} + \dbyd{h}{\mu_i}\dbyd{}{y_i} - \dbyd{h}{y_i}\dbyd{}{\mu_i},
\label{Vl}\end{equation}
while the symbol of $P - \lambda^2$ acting on the right is 
\begin{equation}
\sigma_r(P - \lambda^2) = {\nu'}^2 + h' - \lambda^2, \quad h' = \sum_{i,j} h^{ij}(y') \mu'_i \mu'_j,
\label{sigmar}\end{equation}
and
 the right Hamilton vector field takes the form (after dividing by ${\rho'}$)

\begin{equation}
V_r = \nu' (\sigma \dbyd{}{\sigma} - \mu \dbyd{}{\mu}) + h' \dbyd{}{\nu'} + \dbyd{h'}{\mu'_i}\dbyd{}{y'_i} - \dbyd{h'}{y'_i}\dbyd{}{\mu'_i}.
\label{Vr}\end{equation}
Let $L^{\bfc}_\pm$ denote the flowout in the positive, resp. negative directions by the vector field $V_l$ from $\Nscstar \Diagb \cap \{ \sigma_l(P - \lambda^2)= 0 \}$. 
In \cite{HV1} it was proved

\begin{prop}
(i) Locally near $\Nscstar \Diagb$, the pairs $(\Nscstar \Diagb, L^{\bfc}_\pm)$ form an intersecting pair of Legendre submanifolds.

(ii) Locally near $L^\sharp_\pm$, the pair $(L^\bfc_\pm, L^\sharp_\pm)$ forms a pair of intersecting Legendre submanifolds with conic points. 
\end{prop}

A second way to describe $L^{\bfc}$ is directly in terms of geodesic flow on the metric cone with cross section $(\partial M, h)$. Let $g_{\conic}$ be the conic metric
$$
g_{\conic} = dr^2 + r^2 h.
$$
Write $Y = \partial M$. 
Then geodesic flow on $(C(Y), g_{\conic})$, the cone over $(Y, h)$, can be written explicitly in terms of geodesic flow on $(Y, h)$ as follows: 
Let $(y(s), \eta(s))$, where $s \in [0, \pi]$ is an arc-length parameter (so that $|\eta(s)|_{h(y(s))} = 1$), be a geodesic in $T^* Y$. Then every geodesic $\gamma$ for the exact conic metric  $dr^2 + r^2 h$, not hitting the cone tip, is of the form $y = y(s), \mu = \eta(s) \sin s, r = r_0 \csc s, \nu = -\cos s$, $s \in (0, \pi)$, where $r_0 > 0$ is the minimum distance to the cone tip and  $-\cot s/r_0 \in (-\infty, \infty)$ is arc length on $C(Y)$. 
We define $\gamma^2$ to be the submanifold
\begin{multline}
\gamma^2 = \big\{ (y, y', \sigma = x/x', \mu, \mu', \nu, \nu') \mid y = y(s), y' = y(s'), \mu = \eta(s) \sin s, \\ \mu' = -\eta(s') \sin s', \nu = -\cos s, \nu' = \cos s', \sigma = \sin s'/\sin s, (s, s') \in [0, \pi]^2. \big\} 
\label{gamma^2}\end{multline}
Then $\Lbf$ is  the union of the $\gamma^2$ over all geodesics of length $\pi$ in $T^*(\partial M)$. The vector fields $V_l$ and $V_r$ are tangent to each leaf $\gamma^2$, and are given in terms of the coordinates $s$ and $s'$ by $V_l = \sin s \partial_s$ and $V_r = \sin s' \partial_{s'}$. Also, the intersection of this leaf with $\Nscstar \Diagb$ is $\{ s = s'\}$, and the conic singularity is at  the two off-diagonal corners $s = 0, s' = \pi$ and $s=\pi, s' = 0$, which corresponds to the two different ends of the geodesic. The blowup of the span of $L^\sharp$ that desingularizes these conic singularities corresponds on the leaf $\gamma^2$ to blowing up these two corners.

\subsection{Legendre distributions on $\MMkb$}\label{sect:legdist}

Let $\Lambda \subset \Tscstar_{\bfc} \MMb$ be a Legendre submanifold. We define a space of (half-density) functions on $\MMkb$ associated to $\Lambda$. As usual $n$ denotes the dimension of $M$. 

Let $\mcA = (\mcA_{\bfo}, \mcA_{\lb},\mcA_{\rb},\mcA_{\zf})$ be an index family consisting of an index set for each of the hypersurfaces $\bfo, \lbo, \rbo, \zf$ of $\MMkb$. Also let $m, r_{\lb}, r_{\rb}$ be real numbers. We shall shortly define the set of Legendre (half-density) distributions $I^{m, r_{\lb}, r_{\rb}; \mcA}(\MMkb, \Lambda; \Omegakbh)$. First we give the intuitive idea: for $\lambda > 0$ it is a family of Legendre distributions on $\MMb$,  depending conormally on $\lambda$ as $\lambda \to 0$ with respect to the given family. Away from $\bfc \cup \lb \cup \rb$ it is polyhomogeneous conormal with respect to the given index family. 

We remark that the parametrizations of $\Lambda$ given in the definition below are defined in \cite{HV1} and \cite{HV2}. 

\begin{defn}\label{defn:legdist} The space  $ I^{m, r_{\lb}, r_{\rb}; \mcA}(\MMkb, \Lambda; \Omegakbh)$ consists of half-densities $u$ on $\MMkb$ that  can be written as a finite sum of terms $u = \sum_{j=0}^6 u_j$, where
\begin{itemize}

\item $u_0$ is supported in $\{ \lambda \geq \epsilon \}$ for some $\epsilon > 0$, and $u \otimes |d\lambda/\lambda|^{-1/2}$ is a family of 
Legendre distributions in $I^{m, r_{\lb}, r_{\rb}}(\MMb, \lambda \Lambda; \Omegasch)$ with symbol depending smoothly on $\lambda$;

\item $u_1$ is supported close to $\bfo \cap \bfc$ and away from $\lb \cup \rb$, and is given by a finite sum of expressions
\begin{equation}
\rho^{m-k/2+n/2} \lambda^n  \int_{\RR^k} e^{i\Psi(y, y', \sigma, v)/\rho} a(\lambda, \rho, y, y', \sigma,v) \, dv \Big| \frac{dg dg' d\lambda}{\lambda}\Big|^{1/2}
\label{u1}\end{equation}
where $\sigma = x/x'$, $\Psi$ locally parametrizes $\Lambda$ in the sense of \cite{HV1}, \cite{HV2} and $a$ is polyhomogeneous conormal  in $\lambda$, with respect to the index set $\mcA_{\bfo}$, at $\lambda = 0$ and is smooth in all other variables;

\item $u_2$ is supported close to $\bfo \cap \bfc \cap \lb$, and is given by a finite sum of expressions
\begin{multline}
{\rho'}^{m-(k+k')/2+n/2} \sigma^{r_{\lb} - k/2} \lambda^n  \\ \times \int_{\RR^{k+k'}} e^{i\big( \Psi_0(y,v) + \sigma \Psi_1(y, y', \sigma, v, v') \big)/\rho}  a(\lambda, \rho', y, y', \sigma, v, v') \, dv \, dv' \,  \Big| \frac{dg dg' d\lambda}{\lambda}\Big|^{1/2};
\label{u2}\end{multline}
where $\Psi_0 + \sigma \Psi_1$ locally parametrizes $\Lambda$ and $a$ is polyhomogeneous conormal  in $\lambda$, with respect to the index set $\mcA_{\bfo}$, at $\lambda = 0$ and is smooth in all other variables;

\item $u_3$ is supported close to $\bfo \cap \bfc \cap \rb$, and is given by a similar expression to $u_2$ with $(x,y)$ and $(x',y')$ interchanged, and $r_{\lb}$ replaced by $r_{\rb}$;

\item $u_4$ is supported close to $\lb \cap \bfo$ and away from $\bfc$, and is given by a finite sum of expressions of the form
\begin{equation}
\rho^{r_{\lb} - k/2} \lambda^n \int_{\RR^k}
e^{i\Psi_0(y,v) /\rho} a(\rho, x', 1/\rho', y, y',  v) \, dv \Big| \frac{dg dg' d\lambda}{\lambda}\Big|^{1/2}
\label{u4}\end{equation}
where $\Psi_0$ locally parametrizes $\Lambda_{\lb}$ and $a$ is polyhomogeneous conormal  in $(x', 1/\rho')$, with respect to the index sets $(\mcA_{\bfo}, \mcA_{\lbo})$ and is smooth in all other variables;

\item $u_5$ is supported close to $\rb \cap \bfo$ and away from $\bfc$, and is given by a similar expression to $u_4$ with  $(x,y)$ and $(x',y')$ interchanged, $r_{\lb}$ replaced by $r_{\rb}$, and $\mcA_{\lbo}$ replaced by $\mcA_{\rbo}$;

\item $u_6$ is supported away from $\bfc \cup \lb \cup \rb$ and is of the form $a \tau$ where $\tau$ is a smooth nonvanishing section of $\Omegakb^{1/2}$ and  $a$ is polyhomogeneous with index family $\mcA$ at $\bfo, \lbo, \rbo, \zf$. 
\end{itemize}
\end{defn}

\begin{remark} Recall that, away near $\bfc$, a smooth nonvanishing section of $\Omegakbh$ is given by $\lambda^n |dg dg' d\lambda/\lambda|^{1/2}$; this accounts for the  factors of $\lambda^n$ in \eqref{u1}, \eqref{u2} and \eqref{u4}. 
\end{remark}

\begin{remark} We have chosen here a different order convention from that used in \cite{HW}.  Our convention here has the advantage that if $u \in I^{m, r_{\lb}, r_{\rb}; \mcA}(\MMkb, \Lambda; \Omegakbh)$ then for $\lambda > 0$, $u$ is a smooth family of 
Legendre distributions in $I^{m, r_{\lb}, r_{\rb}}(\MMb, \lambda \Lambda; \Omegasch)$ (tensored with $d\lambda/\lambda|^{1/2}$); i.e., the orders do not change. Our orders $m, r_{\lb}, r_{\rb}$ here corresponds to orders $m+1/4, r_{\lb}+1/4, r_{\rb}+1/4$ in \cite{HW}. 
\end{remark}

In an analogous way, we can define intersecting Legendrian distributions and distributions associated to intersecting pairs of Legendre submanifolds with conic points on $\MMkb$, based on the definitions given in \cite{HV2} for such distributions on $\MMb$. 

\begin{defn}\label{defn:intlegdist} Let $(\Lambda_0, \Lambda_+)$ be an intersecting pair of Legendre submanifolds in $\Nsfstar Z_{\bfc}$, which do not meet the left and right boundaries of $\Nsfstar Z_{\bfc}$. 
The space  $ I^{m, r_{\lb}, r_{\rb}; \mcA}(\MMkb, (\Lambda_0, \Lambda_+); \Omegakbh)$ consists of half-density functions $u$ on $\MMkb$ that  can be written as a finite sum of terms $u = \sum_{j=0}^3 u_j$, where
\begin{itemize}

\item $u_0$ is supported in $\{ \lambda \geq \epsilon \}$ for some $\epsilon > 0$, and $u \otimes |d\lambda/\lambda|^{-1/2}$ is a family of 
Legendre distributions in $I^{m, r_{\lb}, r_{\rb}}(\MMb,  (\lambda\Lambda_0, \lambda\Lambda_+); \Omegasch)$ with symbol depending smoothly on $\lambda$;

\item $u_1$ is an element of $I^{m, r_{\lb}, r_{\rb}; \mcA}(\MMkb, \Lambda_0; \Omegakbh)$, microsupported away from $\Lambda_+$;

\item $u_2$ is an element of $I^{m-1/2, r_{\lb}, r_{\rb}; \mcA}(\MMkb, \Lambda_+; \Omegakbh)$, microsupported away from $\Lambda_0$;

\item $u_3$ is supported close to $\bfo \cap \bfc$ and away from $\lb \cup \rb$, and is given by a finite sum of expressions
\begin{equation}
\rho^{m-(k+1)/2+n/2} \lambda^n \int_0^\infty ds \int_{\RR^k} e^{i\Psi(y, y', \sigma, v,s)/\rho} a(\lambda, \rho, y, y', \sigma,v,s) \, dv \Big| \frac{dg dg' d\lambda}{\lambda}\Big|^{1/2}
\label{u3int}\end{equation}
where $\Psi$ locally parametrizes $(\Lambda_0, \Lambda_+)$ in the sense of \cite{HV1}, \cite{HV2} and $a$ is polyhomogeneous conormal  in $\lambda$, with respect to the index set $\mcA_{\bfo}$, at $\lambda = 0$ and is smooth in all other variables. 
\end{itemize}
\end{defn}

\begin{defn}\label{defn:coniclegdist} Let $(\Lambda, \Lambda^\sharp)$ be an pair of intersecting Legendre submanifolds with conic points in $\Nsfstar Z_{\bfc}$. 
The space  $ I^{m, p; r_{\lb}, r_{\rb}; \mcA}(\MMkb, (\Lambda, \Lambda^\sharp); \Omegakbh)$ consists of half-density functions $u$ on $\MMkb$ that  can be written as a finite sum of terms $u = \sum_{j=0}^5 u_j$, where
\begin{itemize}

\item $u_0$ is supported in $\{ \lambda \geq \epsilon \}$ for some $\epsilon > 0$, and $u \otimes |d\lambda/\lambda|^{-1/2}$ is a family of 
Legendre distributions in $I^{m, p; r_{\lb}, r_{\rb}}(\MMb, ( \lambda\Lambda,  \lambda\Lambda^\sharp); \Omegasch)$ with symbol depending smoothly on $\lambda$;

\item $u_1$ is an element of $I^{m, r_{\lb}, r_{\rb}; \mcA}(\MMkb, \Lambda; \Omegakbh)$, microsupported away from $\Lambda^\sharp$;

\item $u_2$ is an element of $I^{p, r_{\lb}, r_{\rb}; \mcA}(\MMkb, \Lambda^\sharp; \Omegakbh)$, microsupported away from $\Lambda$;

\item $u_3$ is supported close to $\bfo \cap \bfc$, and away from $\lb \cup \rb$,  and is given by a finite sum of expressions 
\begin{multline}
 \lambda^n \int_0^\infty ds \int_{\RR^k} e^{i\Psi(y, y', \sigma, v,s)/\rho} \big( \frac{\rho}{s} \big)^{m-(k+1)/2+n/2} s^{p+n/2-1} \\ 
 \times
 a(\lambda, \frac{\rho}{s}, y, y', \sigma,v,s) \, dv \Big| \frac{dg dg' d\lambda}{\lambda}\Big|^{1/2}
\label{u3con}\end{multline}
where $\Psi$ locally parametrizes $(\Lambda, \Lambda^\sharp)$ in the sense of \cite{HV1}, \cite{HV2} and $a$ is polyhomogeneous conormal  in $\lambda$, with respect to the index set $\mcA_{\bfo}$, at $\lambda = 0$ and is smooth in all other variables;

\item $u_4$ is supported close to $\bfo \cap \bfc \cap \lb$ and is given by a finite sum of expressions
\begin{multline}
 \lambda^n \int_0^\infty ds \int_{\RR^k} e^{i\Psi(y, y', \sigma, v,s)/\rho} \big( \frac{\rho'}{s} \big)^{m-(k+1)/2+n/2} s^{p+n/2-1} \sigma^{r_{\rb}} \\ \times
 a(\lambda, \frac{\rho'}{s}, y, y', \sigma,v,s) \, dv \Big| \frac{dg dg' d\lambda}{\lambda}\Big|^{1/2}
\label{u4con}\end{multline}
where $\Psi$ locally parametrizes $(\Lambda, \Lambda^\sharp)$ in the sense of \cite{HV1}, \cite{HV2} and $a$ is polyhomogeneous conormal  in $\lambda$, with respect to the index set $\mcA_{\bfo}$, at $\lambda = 0$ and is smooth in all other variables;

\item $u_5$ is supported close to $\bfo \cap \bfc \cap \rb$ and is given by a finite sum of expressions analogous to \eqref{u4con}, with $(x,y)$ and $(x',y')$ interchanged, and $r_{\lb}$ replaced by $r_{\rb}$.

\end{itemize}
\end{defn}

\begin{remark} There are typos in  the expression \cite[equation (2.23)]{HV2}  corresponding to \eqref{u4con}. These have been fixed here. For purposes of comparison, notice that in \cite[equation (2.23)]{HV2}, in the exponent of $x_1$, $N/4 - f_i/2$ vanishes (in the present situation).
\end{remark}

\subsection{The boundary hypersurface $\bfo$}\label{sec:bfo} The boundary hypersurface $\bfo$ of $\MMksc$ or $\MMkb$ plays a crucial role in our analysis. This is because it corresponds to the transitional asymptotics between  zero energy  and  positive energy behaviour. Section~\ref{exactcone} is devoted to the analysis of the 
model operator induced by $P$ on $\bfo$, namely the conic Schr\"odinger operator \eqref{conicSchr}. Here we note the geometric structures on $\bfo \subset \MMkb$ induced from $\MMkb$. (Unless specifically indicated, we work on $\MMkb$ rather than $\MMksc$ below.)

We first observe that $\bfo$ is a blown-up version of $\ff \times \ff$, where $\ff$ is the front (blown-up) face of $\Mkb$. The front face $\ff$ is given by $\partial M \times [0, \infty]_r$ where $r = 1/\rho = \lambda/x$. Indeed, the interior of $\bfo$ admits smooth coordinates $(r, y, r', y')$, where $y, y' \in \partial M$, $r, r' \in (0, \infty)$, and we can easily check that $\bfo$ is obtained from $\ff \times \ff$ by performing $b$-blowups at the diagonal corners $\{ r = r' = 0 \}$ and $\{ \rho = \rho' = 0 \}$.  
Moreover, if we work on $\MMksc$, then the scattering blowup \eqref{scblowup} has the effect of performing the scattering blowup on $\bfo$, i.e.\ blowing up $\{ \rho = \rho' = 0, y = y' \}$. (Recall that we have already observed in \eqref{exactconicmetric} that the metric $g$ induces an exact conic metric on the front face of $\Mkb$, hence a scattering metric at $\rho = 0$ and conformal to a b-metric at $r=0$.)

Next consider the vector fields $\mathcal{V}_{k,b}(\MMkb)$ restricted to $\bfo$. First, on the single space, the vector fields $\mathcal{V}_{k,b}(\Mkb)$ restrict to $\ff$ to be scattering vector fields near $\rho = 0$, and b-vector fields near $r=0$. These vector fields on $\ff$ can be lifted to $\bfo$ via either the left or right stretched projections $\bfo \to \ff$, and generate a space of vector fields on $\bfo$ that coincide with the restriction of $\mathcal{V}_{k,b}(\MMkb)$ to $\bfo$. In turn, this space of vector fields in $\bfo$ defines  a vector bundle over $\bfo$ for which such vector fields are the smooth sections. 
Its dual bundle ${}^{b,k} T^* \bfo$ can be identified with the subbundle
of ${}^{k,b}T^*_{\bfo}\MMkb$ annihilated by the vector field $\lambda \partial_\lambda + x \partial_x + x' \partial_{x'}$. In terms of coordinates \eqref{T}, this bundle is given by $\{ \lambda = 0, T = 0 \}$. As this is a symplectic reduction of the bundle ${}^{k,b} T^* \MMkb$, there is a symplectic form induced on ${}^{b,k} T^* \bfo$ by the form $\omega$, which, as in Section~\ref{facs}, induces a contact structure on ${}^{b,k} T^*_{\bfc \cap \bfo} \bfo$, i.e.\ when we restrict to the boundary hypersurface $\bfo \cap \bfc$. 
This restricted bundle is isomorphic (as a bundle and as a contact manifold) to $\Nsfstar Z_{\bfc}$. Therefore, we can define Legendre submanifolds, Legendre distributions, etc, for $\bfo$. However, this is nothing new --- this precisely reproduces the structure described in \cite{HV1} and \cite{HV2} for a manifold with boundary; we could alternatively derive it by treating $\ff$ as a scattering manifold by ignoring the `b'-boundary at $r = 0$, and working locally near the scattering boundary $\rho = 0$, or equivalently working locally near $\bfo \cap \bfc$. 

A consequence of the isomorphism between $\Nsfstar Z_{\bfc}$ 
and ${}^{b,k} T^*_{\bfc \cap \bfo} \bfo$ is that 
Legendre distributions, as defined above, on $\MMkb$ induce Legendre distributions on $\bfo$, essentially by restriction to $\bfo$ --- see Proposition~\ref{bfo-rest} for a precise statement.

\subsection{Statement of main results}

The main result of this paper is a rather precise  description of the resolvent kernel on the space $\MMksc$.

\begin{theo}\label{mainres} There is an index family $\mcB = (\mcB_{\zf}, \mcB_{\bfo}, \mcB_{\lbo}, \mcB_{\rbo})$ such that the
outgoing resolvent kernel $R(\lambda + i0)$, for $\lambda \leq \lambda_0$, can be represented as the sum of four terms $R_1 + R_2 + R_3 + R_4$, where 
\begin{itemize}
\item $R_1 \in \Psi^{-2, (-2, 0, 0), \*}(M, \Omegabht)$ is a pseudodifferential operator of order $-2$ in the calculus of operators defined in \cite{GH1};

\item $R_2 \in I^{m,\mcA}(\MMkb, (\Nscstar\Diagb, L^{\bfc}_+); \Omegakbh)$ is an intersecting Legendre distribution on $\MMkb$, microsupported close to $\Nscstar\Diagb$;

\item $R_3 \in I^{m,p; r_{\lb}, r_{\rb}; \mcA}(\MMkb, (L^{\bfc}_+, \Lsharp_+); \Omegakbh)$ is a Legendre distribution on $\MMkb$ associated to the intersecting pair of Legendre submanifolds with conic points $(L^{\bfc}_+, \Lsharp_+)$, microsupported away from $\Nscstar\Diagb$;

\item $R_4$ is supported away from $\bfc$ and is such that $e^{-i/\rho} e^{-i/\rho'} R_4$ is\footnote{Notice that $e^{i/\rho} = e^{i\lambda r}$ is the usual outgoing oscillation at infinity} 
polyhomogeneous conormal  on $\MMkb$ with index family $\mcB \cup (\mcB_{\lb}, \mcB_{\rb}, \mcB_{\bfc})$, where $\mcB_{\lb} = \mcB_{\rb} = (n-1)/2$ and $\mcB_{\bfc} = \emptyset$.
\end{itemize}
We have $m = -1/2$, $p = (n-2)/2$, $r_{\lb} = r_{\rb} = (n-1)/2$,  $\min \mcB_{\zf} = 0$, $\min \mcB_{\bfo} = -2$, $\min \mcB_{\lbo} = \min \mcB_{\rbo} = \nu_0 - 1$. Moreover, the leading asymptotics of $R(\lambda + i0)$ at $\bfo, \zf, \lbo, \rbo$ are given by \eqref{resmodels}.
\end{theo}

There is a corresponding statement about the incoming resolvent, with $L^{\bfc}_+$ and $\Lsharp_+$ replaced by $L^{\bfc}_-$ and $\Lsharp_-$. Subtracting the incoming from the outgoing resolvent we obtain our result about the spectral measure:

\begin{theo}\label{mainsm} The difference between the outgoing and incoming resolvents is a conormal-Legendre distribution on $\MMkb$ associated to the intersecting pair of Legendre submanifolds with conic points $(L^{\bfc}, \Lsharp)$. More precisely, we have
$$
R(\lambda + i0) - R(\lambda - i0) \in I^{m,p; r_{\lb}, r_{\rb}; \mcA'}(\MMkb, (L^{\bfc}, \Lsharp); \Omegakbh),$$ 
where $\mcA'_{\bullet} = \mcA_{\bullet}$ for $\bullet = \lbo, \rbo, \bfo$, but
$$\mcA'_{\zf} \subset \mcA_{\zf} \setminus \{ (\beta, j) \mid \beta < 2\nu_0 \}
$$
and $\mcA'_{\zf} \setminus\{2\nu_0\}\geq  \min(2\nu_0+1,2\nu_1)$. 
Consequently, the spectral measure \eqref{Stone} of $P_+^{1/2}$ satisfies
$$
dE_{P_+^{1/2}}(\lambda) \in  I^{m,p; r_{\lb}, r_{\rb}; \mcA''}(\MMkb, (L^{\bfc}, \Lsharp); \Omegakbh) \otimes {|\lambda d\lambda|^{1/2}}$$
where $\mcA''_{\bullet} = \mcA'_{\bullet} + (1,0)$; in particular, the spectral measure vanishes to order $2\nu_0 + 1$ for fixed $z, z' \in M^\circ$. The leading asymptotic of the spectral measure at $\zf$ is given by \eqref{smzf}. 
\end{theo}


\section{Symbol calculus for Legendre distributions}\label{sect:symbolcalc} 
The symbol calculus follows in a straightforward way from that given for Legendrian distributions on $\MMb$ given in \cite{HV2}. We state the results for $\MMkb$ here without proof. (One reason for stating the results here is to correct some typos in \cite{HV2}; for example, in the exact sequence of Proposition 3.4 of \cite{HV2}, $\rho^{m-{\bf r}}$ should be $\rho^{{\bf r}-m}$, and Proposition 3.5 has a similar typo.)

The principal symbol map is defined on the space $I^{m, r_{\lb}, r_{\rb}; \mcA}(\MMkb, \Lambda; \Omegakbh)$ of 
Legendre distributions associated to $\Lambda$ and maps to bundle-valued half-densities on $\Lambda \times [0, \lambda_0]$.
Here the bundle in question is the symbol bundle $S^{[m]}(\Lambda)$, pulled back to $\Lambda$ (which we continue to denote $S^{[m]}(\Lambda)$), defined in 
\cite{HV2} or \cite{HW} :
\begin{equation}
S^{[m]}(\Lambda) = M(\Lambda) \otimes E \otimes |N^*_{\bfc} (\partial \MMkb) |^{m-(2n+1)/4}\label{S[m]defn}
\end{equation}
(where $M(\Lambda)$ is the Maslov bundle, and the other bundles are defined in \cite{HV2} or \cite{HW}). Let $\mcC$ denote the index family for the boundary hypersurfaces of $\Lambda \times [0, \lambda_0)$ 
that assigns $\mcA_{\bfo}$ at $\Lambda \times \{ 0 \}$, the one-step index set $r_{\lb} - m$ at $\partial_{\lb} \Lambda \times [0, \lambda_0)$ and the  one-step index set $r_{\rb} - m$ at $\partial_{\rb} \Lambda \times [0, \lambda_0)$. Thus $\mcC$ depends on the data $(\mcA, m, r_{\lb}, r_{\rb})$. 
Then the principal symbol $\sigma^m(u)$ of $u \in I^{m, r_{\lb}, r_{\rb}; \mcA}(\MMkb, \Lambda; \Omegakbh)$ takes values in the polyhomogeneous space $ \phgc_{\mcC}(\Lambda \times [0, \lambda_0]; S^{[m]}(\Lambda) \otimes \Omegabh)$. It is defined by continuity from the symbol map given in \cite{HV2}. 

The following propositions follow straightforwardly from the corresponding results in Section 3 of \cite{HV2}.

\begin{prop}\label{ex} There is an exact sequence
\begin{multline*}
0 \to I^{m+1, r_{\lb}, r_{\rb}; \mcA}(\MMkb, \Lambda; \Omegakbh) \to I^{m, r_{\lb}, r_{\rb}; \mcA}(\MMkb, \Lambda; \Omegakbh) \to  \\
 \phgc_{\mcC}(\Lambda \times [0, \lambda_0], \Omega^\half_b \otimes S^{[m]}(\Lambda)) \to 0.
\end{multline*}
If $u \in I^{m, r_{\lb}, r_{\rb}; \mcA}(\MMkb, \Lambda; \Omegakbh)$, 
then $(P - \lambda^2)u \in I^{m, r_{\lb}, r_{\rb}; \mcA+2}(\MMkb, \Lambda; \Omegakbh)$ and 
$$
\sigma^m((P - \lambda^2)u) =  \sigma_l(P - \lambda^2) \sigma^m(u) .
$$
Thus, if $\sigma_l(P - \lambda^2)$ vanishes on $\Lambda$, $(P - \lambda^2)u \in I^{m+1, r_{\lb}, r_{\rb}; \mcA+2}(M, \Lambda;
\Omegakbh)$. The symbol of order $m+1$ of $(P - \lambda^2)u$ in this case is given by
\begin{equation}
\Big( -i \mathcal{L}_{V_l} -i \big(\half + m - \frac{2n+1}{4} \big) \nu + p_{\mathrm{sub}} \Big) \sigma^m(u) \otimes |dx|,
\label{transport}\end{equation}
where $V_l$ is the vector field \eqref{Vl} and $p_{\mathrm{sub}}$ is the boundary subprincipal symbol of $P - \lambda^2$. 
\end{prop}

\begin{remark} The boundary subprincipal symbol $p_{\mathrm{sub}}$ is defined in \cite[Section 2.1]{HV2}. Here it is sufficient to note that it is a smooth function on $\Nsfstar Z_{\bfc}$ which vanishes on $L^\sharp$. 
\end{remark}

In the next proposition, $ \tilde \Lambda = (\Lambda_0, \Lambda_1)$ is a pair of intersecting Legendre submanifolds as in Proposition 3.2 of \cite{HV2}. We are assume that they do not meet $\Tscstar_{\lb}\MMb$ or $\Tscstar_{\rb}\MMb$. Therefore we are left with the order,  $m$, at $\Lambda_0$, and the index family $\mcA$. We refer to \cite[Section 3.1, equation (3.8)]{HV2} for the definition of the bundle over $\Lambdat$. 

\begin{prop}\label{ex-int} The symbol map on $\Lambdat$ yields an exact sequence
\begin{multline*}
0 \to I^{m+1, \mcA}(\MMkb, \Lambdat; \Omegakbh) \to I^{m,\mcA}(\MMkb, \Lambdat; \Omegakbh) \to \\
\phgc_{\mcA_{\bfo}}(\Lambdat \times [0, \lambda_0], \Omega^\half_b \otimes S^{[m]}) \to 0. 
\end{multline*}
Moreover, if we consider just the symbol map to $\Lambda_1$, there is an exact
sequence
\begin{multline}
0 \to I^{m+1; \mcA}(\MMkb, \Lambdat; \Omegakbh) + I^{m+\half; \mcA}(\MMkb, \Lambda_0; \Omegakbh) \to I^{m;\mcA}(\MMkb, \Lambdat;
\Omegakbh)  \\  
\to  \phgc_{\mcA_{\bfo}}(\Lambda_1 \times [0, \lambda_0], \Omega^\half \otimes S^{[m]}) \to 0.
\label{ex-int-2}\end{multline}
If $u \in I^{m; \mcA}(\MMkb, \Lambdat; \Omegakbh)$, 
then $(P - \lambda^2) u \in I^{m;\mcA}(\MMkb, \Lambdat; \Omegakbh)$ and 
$$
\sigma^m((P - \lambda^2)u) = \sigma_l(P - \lambda^2) \sigma^m(u).
$$
Thus, if the symbol of $P - \lambda^2$ vanishes on $\Lambda_1 \times [0, \lambda_0]$, then $(P - \lambda^2)u$ is an element of
$I^{m+1; \mcA}(\MMkb, \Lambdat;
\Omegakbh) + I^{m+1/2;\mcA}(\MMkb, \Lambda_0; \Omegakbh)$. The symbol of order $m+1$ of $(P - \lambda^2)u$ on $\Lambda_1 \times [0, \lambda_0]$
in this case is given by \eqref{transport}. 
\end{prop}

In the last of these propositions, $\Lambdat$ is an intersecting pair of Legendre submanifolds $(\Lambda, \Lambdas)$ with conic points, as defined above. Now $\Lambdah$ is a manifold with codimension 2 corners since the blowup \eqref{lambdah} creates a new boundary hypersurface at the intersection with $\Lambda^\sharp$, which we denote $\partial_\sharp \Lambdah$. So $\Lambdah \times [0, \lambda_0]$ has codimension 3 corners. We define the index family $\mcC'$ for $\Lambdah \times [0, \lambda_0]$ to be that
which assigns $\mcA_{\bfo}$ at $\Lambdah \times \{ 0 \}$, the one-step index set $r_{\lb} - m$ at $\partial_{\lb} \Lambdah \times [0, \lambda_0)$, the one-step index set $r_{\rb} - m$ at $\partial_{\rb} \Lambdah \times [0, \lambda_0)$,
and the one-step index set $p - m$ at $\partial_{\sharp} \Lambdah$. 

\begin{prop}\label{ex-conic} There is an exact sequence
\begin{multline}
0 \to I^{m+1,p; r_{\lb}, r_{\rb}; \mcA}(\MMkb, \Lambdat; \Omegakbh) \to I^{m,p; r_{\lb}, r_{\rb}; \mcA}(\MMkb, \Lambdat; \Omegakbh) \\
\to \phgc_{\mcC'}(\Lambdah \times [0, \lambda_0], \Omega^\half_b \otimes S^{[m]}(\Lambdah))
\to 0. 
\end{multline}
If $u \in I^{m,p; r_{\lb}, r_{\rb}; \mcA}(\MMkb, \Lambdat; \Omegakbh)$, 
then $(P - \lambda^2)u \in I^{m,p; r_{\lb}, r_{\rb}; \mcA}(\MMkb, \Lambdat; \Omegakbh)$ and 
$$
\sigma^m((P - \lambda^2)u) = \sigma_l(P - \lambda^2) \sigma^m(u).
$$
If the symbol of $P - \lambda^2$ vanishes on $\Lambda$, then $(P - \lambda^2)u \in I^{m+1,p; r_{\lb}, r_{\rb}; \mcA}(\MMkb, \Lambdat;
\Omegakbh)$. The symbol of order $m+1$ of $(P - \lambda^2)u$ in this case is given by
\eqref{transport}. 
\end{prop}

We next consider the operation of restricting to $\bfo$. We are mainly interested in this in a neighbourhood of $\bfc$, so for simplicity we assume that the index family $\mcA$ is such that the index sets at $\lbo, \rbo$ and $\zf$ are empty, i.e.\ the half-densities vanish rapidly at these faces. We also assume for simplicity that the index set $\mcA_{\bfo}$ satisfies
$$
\mcA_{\bfo} = (b, 0) \cup \mcA'_{\bfo}$$ 
with $\mcA'_{\bfo} \geq b + \epsilon$ for some $\epsilon > 0$. 
Let $\mcA'$ denote $\mcA$ with $\mcA_{\bfo}'$ substituted for $\mcA_{\bfo}$. Also recall from Section~\ref{sec:bfo} that a Legendre submanifold 
$\Lambda$ for $\MMkb$ induces one, also denoted $\Lambda$, for $\bfo$. Here $\Lambda$ could be a smooth Legendre submanifold, an intersecting pair of Legendre submanifold, or a Legendre conic pair. 
\begin{prop}\label{bfo-rest} Assume that $\mcA$ satisfies the conditions above. Then there is an exact sequence
\begin{multline*}
0 \to I^{m, r_{\lb}, r_{\rb}; \mcA'}(\MMkb, \Lambda; \Omegakbh) \to I^{m, r_{\lb}, r_{\rb}; \mcA}(\MMkb, \Lambda; \Omegakbh) \to  \\
I^{m,  r_{\lb}, r_{\rb}; \emptyset}(\bfo, \Lambda; \Omegahbsc) \to 0
\end{multline*}
where the last map on the first line is multiplication by $\lambda^{-b}$ and restriction to $\bfo$, and the empty set in the exponent of $I^{m,  r_{\lb}, r_{\rb}; \emptyset}(\bfo, \Lambda; \Omegakbh)$ indicates rapid vanishing at $\lbo, \rbo, \zf$. 
\end{prop}


\section{The resolvent for a metric cone}\label{exactcone}
Let $(Y, h)$ be a closed Riemannian manifold of dimension $n-1$, and let $C(Y)$ denote the metric cone over $Y$; that is, the manifold $(0, \infty) \times Y$ with Riemannian metric $\gconic = dr^2 + r^2 h$. This metric is singular at $r=0$, except in the special case that $(Y, h)$ is  $S^{n-1}$ with its canonical metric, in which case $C(Y)$ is Euclidean space minus one point, with its standard metric (expressed `in polar coordinates'). 

In this section we analyze the operator $\Pconic = \Delta_{\conic} + V_0 r^{-2}$ on $C(Y)$, where $V_0$ is a smooth function of $y \in Y$ satisfying 
\begin{equation}
\Delta_{Y}+\frac{(n-2)^2}{4}+V_0>0, 
\label{posop}\end{equation}
as in \eqref{hyp2}. As is evident from \eqref{Pb}, under this condition $\Pconic$ is a positive operator. Acting initally with domain $C_c^\infty((0, \infty) \times Y)$, it is essentially self-adjoint in dimensions $n \geq 4$; for $n=3$ we use the Friedrichs extension of the corresponding quadratic form. 
We construct the resolvent kernel $(\Pconic - (1 + i0))^{-1}$. This problem has been considered previously, e.g.\ in \cite{Callias}, \cite{BS} and \cite{CT}; here we use an essentially microlocal approach based on the theory of Legendre distributions. 
\

The operator $\Pconic$ on $C(Y)$ as a differential operator has the form 
\begin{equation}
\Pconic = - \partial_r^2 - \frac{n-1}{r} \partial_r + \frac1{r^2} \Delta_Y + \frac{V_0(y)}{r^2} .
\label{conicSchr}\end{equation}
In terms of the variable $x = 1/r$, this reads
$$
-(x^2 \partial_x)^2 + (n-1) x^3 \partial_x + x^2 \Delta_Y + x^2 V_0(y) ,
$$
and is an elliptic scattering differential operator near $x=0$. In the remainder of this section we  regard this operator as acting on half-densities, using the flat connection on half-densities that annihilates the Riemannian half-density. Now let 
\begin{equation}
\Pbconic = r \, \Pconic \ r
\label{Pb1} \end{equation}
and compute
\begin{equation*}\begin{gathered}
\Big( r \Pconic r \Big)  \Big( f |d\gcyl|^{1/2} \Big) \\ = 
r^{-n/2} \Big( r^{1 + n/2} \big( - \partial_r^2 - \frac{n-1}{r} \partial_r + \frac1{r^2} \Delta_Y + \frac{V_0(y)}{r^2} \big) r^{1-n/2} \Big)  \Big( f |d\gconic|^{1/2} \Big) \\
= \Big( \big( -(r\partial_r)^2 + \Delta_Y + V_0 +(n/2 - 1)^2 \big) f \Big)  |d\gcyl|^{1/2}
\end{gathered}\end{equation*}
from which we deduce that, using the connection  on the half-density bundle which annihilates the \emph{cylindrical} half-density $|d\gcyl|^{1/2} = |dr/r dh|^{1/2}$, 
\begin{equation}
\Pbconic = -(r\partial_r)^2 + \Delta_Y + V_0 +(n/2 - 1)^2 .
\label{Pb}\end{equation}
Hence, after pre- and post-multiplying by $r$, our operator is equivalent to an elliptic b-differential operator $\Pbconic$ endowed with the flat connection that annihilates $|d\gcyl |^{1/2}$.  It follows from this formula for $\Pbconic$ that we can separate the $r$ and $y$ variables and express the resolvent kernel $(\Pbconic + k^2r^2)^{-1}$, for $k > 0$, as an infinite sum
\begin{equation}\label{metric-cone-k}
 \sum_{j=0}^\infty\Pi_{E_j}(y,y')\Big(I_{\nu_j}(kr)K_{\nu_j}(kr')H(r'-r)+I_{\nu_j}(kr')K_{\nu_j}(kr)H(r-r')\Big) \left|\frac{dr \, dr'}{rr'}\right|^\demi
\end{equation}
where $I_\nu, K_\nu$ are modified Bessel functions (see \cite[Section 4]{GH1}) and $H$ is the Heaviside function. This formula analytically continues to the imaginary axis; setting $k = -i$, and using the formulae
$$
I_\nu (-iz) = e^{-\nu \pi i/2} J_\nu (z), \quad 
K_\nu(-iz) = \frac{\pi i}{2} e^{\nu \pi i/2} \Ha^{(1)}_\nu(z),
$$
we see that the kernel of $(\Pconic - (1 + i0)^2)^{-1}$ is 
\begin{equation}\label{metric-cone}
\frac{\pi i  r r'}{2}  \sum_{j=0}^\infty\Pi_{E_j}(y,y')\Big(J_{\nu_j}(r)\Ha^{(1)}_{\nu_j}(r')H(r'-r)+J_{\nu_j}(r')\Ha^{(1)}_{\nu_j}(r)H(r-r')\Big) \left|\frac{dr \, dr'}{rr'}\right|^\demi
\end{equation}
where $\Pi_{E_j}$ is projection on the $j$th eigenspace $E_j$ of the operator $\Delta_Y + V_0 + (n/2 - 1)^2$ (on half-densities) on $Y$ and $\nu_j^2$ is the corresponding eigenvalue; also $J_\nu, \Ha_{\nu}^{(1)}$ are standard Bessel and Hankel functions. 
 This expression converges only distributionally, and is of very little help in revealing the asymptotic behaviour of the kernel, say as both $r$ and $r'$ tend to  $\infty$. For the purposes of this paper, we need very precise information on the kernel in this region. Therefore we give a different construction, based on the construction for scattering metrics in \cite{HV2} and \cite{GH1} together with the construction for b-metrics in \cite{APS}. First we define compactifications of $C(Y)$ and $C(Y)^2$, on which the construction takes place. 
 
\subsection{Compactifications of $C(Y)$ and $C(Y)^2$}\label{sect:coniccomp}
We begin by defining compactifications of $C(Y)$ and $C(Y)^2$. 
These constructions are parallel to those in Section~\ref{5.1} for $M \times [0, \lambda_0]$. 

Let us compactify $C(Y)$ to  $Z = [0, \infty]_r \times Y$, where we use $[0, \infty]_r$ to denote the one-point compactification of $[0, \infty)_r$ with boundary defining function $x = 1/r$ at $r = \infty$. As we have seen, $Z$ is the same as the boundary hypersurface $\ff$ of $\Mkb$.
We denote the boundary hypersurfaces of $Z$ at $r=0$ and $r=\infty$ by $\partial_0 Z$ and $\partial_\infty Z$, respectively. To define the double space, we start from $Z^2$ and perform a `b-blowup';
that is, we blow up the codimension 2 corners of $Z^2$ that meet the diagonal, yielding the b-double product $Z^2_b$:
\begin{equation}
Z^2_b = \big[Z^2; \partial_0 Z \times \partial_0 Z; \partial_\infty Z \times \partial_\infty Z \big].
\label{Z2bdefn}\end{equation}
Let $\Diagb(Z)$ denote the lift of the diagonal submanifold to $Z^2_b$. 
We then perform a `scattering blowup' near $r = \infty$. Specifically, we blow up the boundary $\partial_\infty \Diagb(Z)$ of $\Diagb(Z)$ lying over $r = r' = \infty$, obtaining a space we call $\ZZbsc$:
\begin{equation}
\ZZbsc = \big[Z^2; \partial_0 Z \times \partial_0 Z; \partial_\infty Z \times \partial_\infty Z; \partial_\infty \Diagb(Z) \big].
\label{Z2bscdefn}\end{equation}
\begin{figure}[ht!]
\begin{center}
\input{Z2bsc.pstex_t}
\caption{The manifold $Z^2_{b,\sca}$; the dashed line is the lifted diagonal of $Z^2$. The coordinate $r$ vanishes at $\zf$ and $\rbo$, while $r'$ vanishes at $\zf$ and $\lbo$. It is canonically isomorphic to the face $\bfo$ in Figure~\ref{fig:mmksc}.}
\label{Z2bsc}
\end{center}
\end{figure}
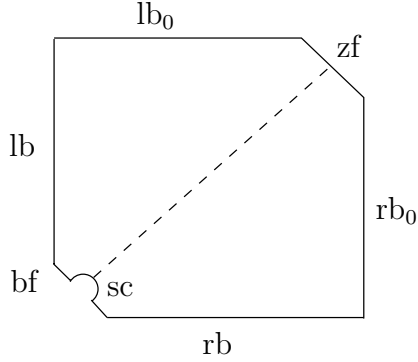

If $Y = \partial M$, then $\ZZb$ is canonically diffeomorphic to the boundary hypersurface $\bfo$ of $\MMkb$, and $\ZZbsc$ is canonically diffeomorphic to the boundary hypersurface $\bfo$ of $\MMksc$. Accordingly, we label the boundary hypersurfaces of $\ZZb$ and $\ZZbsc$ consistently with those of $\MMkb$ and $\MMksc$: the boundary hypersurfaces of $Z^2$ at $r'=0, r=\infty, r=0, r' = \infty$ will be denoted $\lbo, \lb, \rbo, \rb$ respectively,\footnote{This is not a typo; it is really the case that $\lbo$ corresponds to $r' = 0$ and $\rbo$ corresponds to $r=0$; see figure.} and the boundary hypersurfaces created by blowing up $\partial_0 Z \times \partial_0 Z$,  $\partial_\infty Z \times \partial_\infty Z$, and $ \partial_\infty \Diagb(Z)$ will be denoted by  $\zf, \bfc$ and (in the case of $\ZZbsc$) $\sca$, respectively. The lift of $\Diagb(Z)$ to $\ZZbsc$ we denote $\Diagbsc(Z)$. 

\subsection{Statement of main result for metric cones}
Recall (from Section~\ref{sec:bfo}) that there is a natural identification between
$\Nsfstar Z_{\bfc}$ and ${}^{b,k} T^*_{\bfc \cap \bfo} \bfo$ (where $\bfo$ here indicates the boundary  hypersurface of $\MMkb$) . Also, we noted above that $\bfo$ is naturally isomorphic to $\ZZb$. 
Consequently, the Legendre submanifolds $\Nsfstar \Diagb$, $L_{\pm}$ and $L^\sharp$ introduced in Section~\ref{sec:legsub} induce Legendre submanifolds in ${}^{b,k} T^*_{\bf} \ZZb$. To avoid excessive notation, these Legendre submanifolds of ${}^{b,k} T^*_{\bf} \ZZb$ will be denoted by the same symbols. In terms of these, the  main result of this section is 
 
\begin{theo}\label{conic} The kernel of $(\Delta_{\conic} + V_0 r^{-1} - (1 + i0))^{-1}$ is the sum of four terms $R_1 + R_2 + R_3 + R_4$, where 
\begin{itemize}
\item
$R_1$ is a pseudodifferential operator on $\ZZbsc$ (in the b-calculus near $\partial_0 \Diagbsc$, and in the scattering calculus near $\partial_\infty \Diagbsc$), supported near $\Diagbsc$ and vanishing to second order at $\zf$; 
\item 
$R_2 \in I^{-1/2}(\ZZb, (\Nscstar \Diagb, L_+); 
\scOh)$ is an intersecting Legendre distribution of order $-1/2$, supported near $\partial_\infty \Diagb$; 
\item
$R_3 \in I^{-1/2, p; r_{\lb}, r_{\rb}}(\ZZb, (L_+, L^\sharp); \scOh)$ is a Legendre distribution associated to the intersecting pair of Legendre submanifolds with conic points $(L_+, L^\sharp)$, with $p = (n-2)/2$, $r_{\lb} = r_{\rb} = (n-1)/2$, supported near $\bfc$; and 
\item $R_4$ is supported away from $\bfc$ and is such that $e^{-ir} e^{-ir'} R_4$ is polyhomogeneous conormal on $\ZZb$ vanishing to order  $2$ at $\zf$, $n/2$ at $\lbo$ and $\rbo$, and $(n-1)/2$ at $\lb$ and $\rb$.
\end{itemize} 
\end{theo}

The proof of this theorem will occupy the rest of this section. 

\subsection{Parametrix construction}
To construct the kernel of $(\Delta + V_0 r^{-2}- (1 + i0))^{-1}$, we follow the method of \cite{HV2}: we first define a parametrix $\tilde G$ on the space $\ZZbsc$ and show that it gives a good approximation in the sense that 
$$
(\Delta + V_0 r^{-2}- 1) \tilde G = \Id + \tilde E
$$
with $\tilde E$ relatively `small'. We then correct $\tilde G$ by a finite rank term to obtain a new parametrix $G$ such that $\Id + E = (\Delta_{\conic} + V_0 r^{-2} -1) G$ is invertible, to obtain  $(\Delta + V_0 r^{-2}- (1 + i0))^{-1} = G (\Id + E)^{-1}$. This is all done in a calculus of operators that gives us very good control over the behaviour of the kernel at the boundary of the space $Z^2_{b,\sca}$, allowing us to prove Theorem~\ref{conic}. 

To construct $\tilde G$, we use the construction near $\sca$ and $\bfc$ from \cite{HV2}, which applies verbatim, as this construction is all local near infinity. 
Let us recall that this construction is made in four stages. First, we take an interior parametrix, i.e.\ a distribution $G_1$ conormal to and supported close to $\Diagbsc(Z) \subset \ZZbsc$ whose full symbol is the inverse of the full symbol of $\Delta_{\conic} - 1$. If we apply $\Delta_{\conic} - 1$ to such an interior parametrix we are left with an error term that, in a neighbourhood of $r = r' = \infty$, is smooth and supported close to $\Diagbsc$. If we view the error term on $Z^2_b$, then it is Legendrian with respect to $\Nsfstar \Diagb$ (see Section 4.1 of \cite{HV2}). In the second stage, this error is solved away microlocally with an intersecting Legendre distribution on $Z^2_b$ lying in $I^{-1/2}(\ZZb, (\Nscstar (\partial_\infty \Diagb), L_+), \Omegabsc)$, associated to the conormal bundle of the boundary of $\Delta_b(Z)$ and to the outgoing half of the `propagating Legendrian' $L_+$ described in the previous section. This gives us a parametrix $G_2$ with error te!
 rm $E_2$ that is Legendrian with respect to $L_+$ and microsupported away from $\partial_\infty \Nscstar \Diagb$ (see Section 4.2 of \cite{HV2}). 
 In the third stage, the error $E_3$ is solved away using a Legendrian conic pair associated to $(L, L^\sharp)$, giving a parametrix $G_3$ with error term $E_3$ Legendrian with respect to $L^\sharp$ only; thus, at this stage, the errors at $L$ have been solved away completely (see Sections 4.3 and 4.4 of \cite{HV2}).  In the fourth stage, the error term $E_3$ is solved away to infinite order at $\bfc$ and at $\lb$ (we recall that we can solve away to infinite order at $\lb$ but not $\rb$ since we apply the operator $\Delta_{\conic} + V_0 r^{-2} - 1$ in the left variables, so we obtain a Taylor series calculation which is easily solved order by order at $\lb$, while at $\rb$ we are left with a global problem which we cannot hope to solve). This yields a parametrix $G_4$ with error term $E$ rapidly vanishing at the boundary of $\ZZbsc$ except at $\rb$ where it has the form $A(r, y, y') (r')^{-(n-1)/2} e^{ir'} |d\gconic d\gconic'|^{1/2}$, where $A = O(r^{-\infty})$ as $r \to \i!
 nfty$ (see Section 4.5 of \cite{HV2}).

We take the kernel $\tilde G$ to be equal to $G_4$ in a neighbourhood of $\sca \cup \bfc \cup \Diagbsc$ of $\ZZbsc$, and supported away from $\lbo$ and $\rbo$. We now need to specify the parametrix near the boundary hypersurfaces $\zf, \lbo, \rbo$. 

At $\zf$, where $r = r' = 0$, we use the b-calculus. Any b-pseudodifferential operator on half-densities has a `normal operator', that is, the restriction of the kernel of the operator to the `front face' (here the face $\zf$), which has a natural interpretation as a dilation-invariant operator on a half-cylinder (here $\partial M \times (0, \infty)_\sigma$, where $\sigma = r'/r$). In the present case, we write $\Delta_{\conic} + V_0 r^{-2} - 1 = r^{-1} \big( \Pbconic - r^2 \big) r^{-1}$, so the b-operator of interest is $\Pbconic - r^2$, and its
normal operator is precisely $\Pbconic$, given by \eqref{Pb}, which is manifestly dilation-invariant. In the b-calculus, the normal operator of the inverse of a b-elliptic operator is the inverse of the normal operator \cite{APS}. 
We therefore specify that $\tilde G$ vanishes to second order at $\zf$, and the restriction of $(rr')^{-1} \tilde G $ to $\zf$ is equal to $\Pbconic^{-1}$. (We remark that this inverse exists due to assumption \eqref{posop}.) This has a distributional expansion in terms of the eigenfunctions on $\partial M$ as 
\begin{equation}\label{metric-cone-zf}
 \sum_{j=0}^\infty\Pi_{E_j}(y,y') \frac1{2\nu_j} \Big( (r/r')^{\nu_j} H(r'-r)+(r'/r)^{\nu_j} H(r-r')\Big) \left|\frac{dr dydr'dy'}{rr'}\right|^\demi .
\end{equation}
Thus $\tilde G = (r r') \Pbconic^{-1} + O(\rho_{\zf}^3)$ will be polyhomogeneous conormal at $\lbo$ and $\rbo$ with index set 
\begin{equation}
\mcA_{\lbo} = \mcA_{\rbo} = \{ (\nu_j + 1, 0) \mid j = 0, 1, 2, \dots \}
\label{lborbo}\end{equation}
with $\nu_j$ as in \eqref{metric-cone}; in particular, $\min \mcA_{\lbo} = \min \mcA_{\rbo} = \nu_0 + 1$. We also observe that this specification of $\tilde G$ near $\zf$ is compatible with the interior parametrix. This follows
from the fact that the full singularity (modulo $C^\infty$) at the diagonal, both for the interior parametrix and for \eqref{metric-cone-zf}, is uniquely determined by the full symbol of the operator. Explicitly, we can construct a kernel near $\zf$ as follows: we take our interior parametrix, which is supported close to the diagonal, and let $E_{\zf}$ denote the difference between this parametrix (restricted to $\zf$) and \eqref{metric-cone-zf}. As explained above, the difference is $C^\infty$. We extend this $C^\infty$ half-density function in some smooth manner from $\zf$ to $C(Y)$, and add this to our interior parametrix. The result agrees with our specifications both at the diagonal and at $\zf$.

Next we specify what happens at $\lbo$ and $\rbo$. To do this, we note that in the expansion \eqref{metric-cone}, the terms are vanishing more and more rapidly at $\lbo$ and at $\rbo$ as $j \to \infty$, since 
$J_\nu(z) = O(z^{\nu})$ as $z \to 0$. Therefore, we can form a Borel sum at these boundary hypersurfaces. To do this, choose boundary defining functions $\rho_{\lbo}, \rho_{\rbo}$ for these boundary hypersurfaces (for example we could take $\rho_{\lbo} = r'\langle r \rangle/r$, and $\rho_{\lbo} = r\langle r' \rangle/r'$). 
Then we specify  that $\tilde G$ is equal to
\begin{equation}
\bigg( \frac{\pi}{2i} (r r') \sum_j \Pi_{E_j}(y,y') \Ha^{(1)}_{\nu_j}(r)J_{\nu_j}(r') \varphi\big( \frac{\rho_{\lbo}}{ \epsilon_j }\big) + O(\rho_{\lbo}^\infty) \bigg) \Big|\frac{dr \, dr'}{r r'}\Big|^{1/2}
\label{cone-lbo}\end{equation}
near $\lbo$, for some $\varphi \in C_c^\infty[0, \infty)$ equal to $1$ near $0$, and some sequence $\epsilon_j$ tending to zero sufficiently fast, and 
\begin{equation}
\bigg( \frac{\pi}{2i} (r r')  \sum_j \Pi_{E_j}(y,y') \Ha^{(1)}_{\nu_j}(r')J_{\nu_j}(r)  \varphi\big( \frac{\rho_{\rbo}}{ \epsilon_j }\big)+ O(\rho_{\rbo}^\infty) \bigg) \Big|\frac{dr \, dr'}{r r'}\Big|^{1/2}
\label{cone-rbo}\end{equation}
near $\rbo$\footnote{Note the confusing fact that $\lbo$ is the face where $r' = 0$, which $\rbo$ is the face where $r = 0$!}. (We remark that in these formulae the $\Pi_{E_j}(y,y')$ terms contain half-density factors in the $(y, y')$ variables.) 
To check that this is compatible with the behaviour specified at $\zf$, we take the leading behaviour of these expressions at $\zf$. To do this we need the leading behaviour of Bessel and Hankel functions at $r=0$, given by \cite{AS}
\begin{equation}\begin{gathered}
J_\nu(z) = \frac1{\Gamma(\nu + 1)} \big( \frac{z}{2} \big)^\nu + O(z^{\nu + 1}), \\
\Ha^{(1)}_\nu(z) = \frac1{i\pi} \Gamma(\nu) \big( \frac{z}{2} \big)^{-\nu} + O(z^{-\nu + 1}).
\end{gathered}\label{Bessel-asymptotics}\end{equation}
This implies that at the leading behaviour of \eqref{cone-lbo} at $\zf$ is 
$$
(rr') \sum_j \Pi_{E_j} \frac1{2\nu_j} \big(\frac{r'}{r} \big)^{\nu_j},
$$
which is equal to \eqref{metric-cone} modulo $O(\rho_{\lbo}^\infty)$, 
and the leading behaviour of \eqref{cone-rbo} at $\zf$ is 
$$
(rr') \sum_j \Pi_{E_j} \frac1{2\nu_j} \big(\frac{r}{r'} \big)^{\nu_j},
$$
which is equal to \eqref{metric-cone} modulo $O(\rho_{\rbo}^\infty)$. 
This proves that all our specifications at $\zf, \lbo, \rbo$ are compatible. 

We next observe that the asymptotic formulae for Hankel functions for large argument, namely
$$
\Ha_{\nu}^{(1)}(r) = r^{-1/2} e^{ir - i\nu \pi/2 + i\pi/4} h_{\nu}(r), \quad r \geq 1, 
$$
where $h_{\nu}(r)$ is a classical symbol of order zero, i.e.\ with an expansion as $r \to \infty$ in nonpositive integral powers of $r$, 
implies that, near $\lbo \cap \lb$,  \eqref{cone-lbo} is of the form $r^{-(n-1)/2} e^{ir}|dg_{\conic} d{\gcyl'}|^{1/2}$ times a polyhomogeneous conormal function with $C^\infty$ index set at $\lb$ and index set $\mcA_{\lbo}$ at $\lbo$. A similar statement is valid for \eqref{cone-rbo} near $\rbo \cap \rb$. 

\begin{remark}\label{summary}
So far, we have found a parametrix $\tilde G$ which is the sum of a number of pieces: 
\begin{itemize}
\item a pseudodifferential operator, i.e.\ a kernel conormal at $\Diagbsc$ and supported close to $\Diagbsc$ (this is in the scattering calculus near $\partial_\infty \Diagbsc$ and in the b-calculus near $\partial_0 \Diagbsc$); 
\item
an intersecting Legendre distribution supported close to $\partial_\infty \Diagb$;
\item
 a conic Legendre pair supported near $\bfc$; and 
 \item
 a kernel which is supported away from $\bfc$ and is $e^{ir}e^{ir'}$ times a polyhomogeneous conormal half-density, with  index sets $\mcA_{\lbo} = \mcA_{\rbo}$ at $\lbo, \rbo$, and one-step index sets $2$ at $\zf$,  $(n-1)/2$ at $\lb, \rb$. 
  \end{itemize} 
In particular, our parametrix $\tilde G$ satisfies the conditions of Theorem~\ref{conic}. It remains to find the correction term and show that it also satisfies the conditions of Theorem~\ref{conic}.
\end{remark}

\subsection{Correction term and true resolvent}\label{correction}
Define $\tilde E = (\Delta_{\conic} + V_0r^{-2} - 1) \tilde G - \Id$. Then $\tilde E$ is $e^{ir}$ times a kernel that is conormal on $\ZZb$ and vanishes to order $1$ at $\zf$, $\infty$ at $\rbo, \lbo, \lb$ and $\bfc$ and to order $(n-1)/2$ at $\rb$. Thus, $\tilde E$ is a compact operator acting on $\ang{r}^{-l} L^2(Z)$ for any $l > 1/2$. It is not necessarily the case that $\Id + \tilde E$ is invertible on any of these spaces, however. To arrange this, for some (and then, it turns out, every) $l$, we add, following \cite{HV2}, a finite rank term to $\tilde G$, of the form 
$$
\sum_{i=1}^N \phi_i \ang{\psi_i, \cdot}.
$$
Here $N$ is the common value of the dimension of the kernel and cokernel of $\Id + \tilde E$ on $\ang{r}^{-l} L^2(Z)$ (where $l$ is a fixed real number $> 1/2$). We choose $\psi_i$ to span the null space of $\Id + \tilde E$ and $\phi_i$ to span a subspace supplementary to the range of $\Id + \tilde E$. Note that, due to the rapid vanishing of the kernel of $\tilde E$ as $r \to \infty$, if $\psi = -\tilde E\psi$, then $\psi$ vanishes rapidly at $r \to \infty$. Also, since $\tilde E$ vanishes to first order at $r = 0$, $\psi$ must vanish to infinite order at $r=0$ also. Hence each $\psi_i \in \CIdot(Z)$. 

To choose the $\phi_i$, we prove an analogue of Lemma 6.1 in \cite{HV2}:

\begin{lem} Let $l > 1/2$, and let $\CIdot(Z)$ denote smooth functions on $Z$ vanishing to infinite order at the boundary. Then the image of $\Pconic - 1$ on $\CIdot(Z) + \tilde G (\CIdot(Z))$ is dense in $\ang{r}^{-l} L^2(Z)$.
\end{lem}

\begin{proof} 
 This is proved in a similar way as Lemma 6.1 in \cite{HV2}. Let $\mathcal{M}$ be the subspace spanned by $(\Pconic - 1)(\CIdot(Z)$ and  $(\Pconic - 1)(\tilde G (\CIdot(Z)))$, and let $f$ be a function in 
 $\ang{r}^{-l} L^2(Z)$ orthogonal (in the inner product on $\ang{r}^{-l} L^2(Z)$) to $\mathcal{M}$. We shall prove that $f=0$. 
 
 With $\ang{\cdot, \cdot}$ denoting the inner product on $L^2$, we have 
\begin{equation}\begin{gathered}
\ang{ \ang{r}^{l} f, \ang{r}^{l} (\Pconic - 1) u} = 0 \quad \forall \ u \in \CIdot(Z) \\
 \implies 
 \ang{  (\Pconic - 1)(\ang{r}^{2l} f), u} = 0 \quad \forall \ u \in \CIdot(Z) 
 \end{gathered}\end{equation}
 which implies, setting $h = \ang{r}^{2l} f$, that $(\Pconic - 1) h = 0$. Now we apply the same argument setting this time $u = \tilde G v$, where $v \in \CIdot(Z)$, and the operator identity $(\Pconic - 1)\tilde G = \Id + \tilde E$,  to deduce that 
$ (\Id - \tilde E^*) h = 0$, or equivalently $h = \tilde E^* h$. But $\tilde E^*$ vanishes to order $1$ at $\zf$ and to infinite order at $\lbo, \rbo$, which shows that $h(r,y)$ vanishes to infinite order at $r=0$. Similarly,  $\tilde E^*$ maps $\ang{r}^{2l} L^2(Z)$ to $r^{-(n-1)/2} e^{-ir} \CI(Z)$ for $r \geq 1$, so we deduce that $h$ has this behaviour for $r \to \infty$. Let $h = r^{-(n-1)/2} e^{-ir} h_0(y) + O(r^{-(n+1)/2})$. Then applying Green's identity to $h$ and its complex conjugate, we find that 
$$\begin{gathered}
0 = \int_Z \Big( ((\Pconic - 1)h) \overline{h} - h  (\Pconic - 1)\overline{h} \Big)\, r^{n-1} dr dy \\
=  \bigg( \lim_{r \to \infty} \int_Y \big( (\partial_r h)\overline{h}  - h \partial_r \overline{h} \big) r^{n-1} dy - \lim_{r \to 0} \int_Y \big( (\partial_r h)\overline{h} - h \partial_r \overline{h} \big) r^{n-1} dy \bigg) \\
= -2i \int_Y |h_0(y)|^2 \, dy.
\end{gathered}$$
This shows that $h_0 = 0$. This implies that actually $h = O(r^{-(n+1)/2})$ as $r \to \infty$, so $h \in L^2(Z)$. But $\Pconic + V_0 r^{-2}$ has a dilation symmetry, so it has no point spectrum. Therefore $h = 0$, which finishes the proof. 
 \end{proof}
 
Having established this density result, it follows that we can find $\phi_i$, $i = 1 \dots N$, spanning a space supplementary to the range of $\Id + \tilde E$, each of which is the sum of a function in $\CIdot(Z)$ and one in $\tilde G (\CIdot(Z))$. By the mapping properties of $\tilde G$, such functions are smooth in $(0, \infty)$ and having the form $r^{-(n-1)/2} e^{ir} C^\infty(Z)$ as $r \to \infty$, while polyhomogeneous with index set $\mcA_{\lbo}$ as $r \to 0$.  Thus, when we add this finite rank term to $\tilde G$, it does not change any of properties of $\tilde G$ as listed in Remark~\ref{summary}. Let  the resulting kernel be denoted $G$. We then have $(\Pconic - 1)G' = \Id + E$, with $E$ having the same structure as above and with $\Id + E$ invertible on $\ang{r}^{-l} L^2$. Notice that the same $G$ works for all $l > 1/2$. 

Let $(\Id + E)^{-1} = \Id + S$. Then, since $E$ is Hilbert-Schmidt on $\ang{r}^{-l} L^2$, and we can write $S = -E + E^2+ ESE$, $S$ is also Hilbert-Schmidt on $\ang{r}^{-l} L^2$. To analyze finer properties of $S$, we introduce the following spaces of operators:

\begin{defn} 
Let $\Psi_j^0$, $j = 0, 1, 2, \dots$, be the algebra of operators (acting on half-densities) whose kernels are smooth on $\ZZbsc$,  vanishing to order $j$ at $\zf$ and vanishing to infinite order at all other boundary hypersurfaces; thus, we can think of these operators as b-pseudodifferential operators of order $-\infty$, vanishing to order $j$ at $\zf$. 
 Also, let $\Psi^\infty$ be the algebra of operators (acting on half-densities) whose kernels are smooth on $C(Y)$, and on $\ZZbsc$ are of the form ${r'}^{-(n-1)/2} e^{ir'} \CI(\ZZbsc)$  near $\rb$ and vanish to infinite order at all other boundary hypersurfaces. 
\end{defn}

It is straightforward to check the composition properties (the first following from composition properties of the b-calculus)
\begin{equation}\begin{aligned}
\Psi^0_j \circ \Psi^0_k &\subset \Psi^0_{j+k}\\
\Psi^0_j \circ \Psi^\infty &\subset \Psi^\infty \\
\Psi^\infty \circ \Psi^0_j &\subset \Psi^0_k \ \forall \ k, \text{ i.e.\ } \Psi^\infty \circ \Psi^0_j \subset \CIdot(\MMkb) \\
\Psi^\infty \circ \Psi^\infty &\subset \Psi^\infty.
\end{aligned}\label{comp-Psi}\end{equation}

In terms of these algebras, we can write $E \in \Psi^0_1 + \Psi^\infty$. 
Iterating the identity $$S = -E + E^2+ ESE,$$ we obtain 
\begin{equation}
S = - E + E^2 - E^3 + \dots + E^{4N} + E^{2N} S E^{2N}.
\label{SES}\end{equation}

Applying \eqref{comp-Psi} iteratively, we see that $E^j \in \Psi^0_j + \Psi^\infty$. Therefore, $- E + E^2 - E^3 + \dots + E^{4N} \in \Psi^0_1 + \Psi^\infty$. Next we analyze $E^{2N} S E^{2N}$. 

\begin{lem}
The operator $E^{2N} S E^{2N}$ has a kernel of the form
\begin{equation}
\Big( \frac{r}{\ang{r}} \Big)^N \Big( \frac{r'}{\ang{r'}} \Big)^N \ang{r'}^{-(n-1)/2} e^{ir'} C^{N}(Z \times Z)
\label{E2NSE2N}\end{equation}
as a multiple of the half-density $|d\gconic d\gconic'|^{1/2}$. Here $C^N(Z \times Z)$ denotes the space of functions on $Z \times Z$ with $N$ continuous derivatives. 
\end{lem}
 
\begin{proof}
In this proof, all kernels are understood to be multiples of the Riemannian half-density $|d\gconic d\gconic'|^{1/2}$. 

First, we know that $S$ is Hilbert-Schmidt on the space 
$\ang{r}^{-l} L^2(Z)$, $l > 1/2$, so its kernel is in the space 
$$
\ang{r}^{-l} \ang{r'}^{l} L^2(Z \times Z).
$$
However, rearranging the identity $(\Id + E)(\Id + S) = \Id$, we find that
$S = -(E + ES)$. Since the kernel of $E$ vanishes to infinite order as $r \to \infty$, we find that this is true for $S$ as well. In particular, we see that 
$$
S \in \ang{r}^{-l} \ang{r'}^{l} L^2(Z \times Z).
$$

Next, we have $E^{2N}(z,z'') \in \Psi^0_{2N} + \Psi^\infty$. If we differentiate this kernel $N$ times, it still vanishes to order $N$ at $\zf$, and to infinite order as $r \to \infty$. Therefore we can say that $E^{2N}$ has a kernel which is of the form 
$$
\Big( \frac{r}{\ang{r}} \Big)^N \ang{r}^{-N} C^N \big(Z; \ang{r'}^l L^2(Z') \big);
$$
that is, it is $C^N$ in the first variable, vanishing to order $N$ as $r \to 0$ and infinite order as $r \to \infty$, as a  $L^2$ function (weighted by $\ang{r'}^l)$ in the second variable. (The prime on $Z'$ in the formula above indicates the right factor; lack of prime indicates the left factor.)

Equally, we can describe $E^{2N}$ as being in the space
$$
\Big( \frac{r'}{\ang{r'}} \Big)^N \ang{r'}^{-(n-1)/2} C^N \big(Z'; \ang{r}^{-l} L^2(Z) \big);
$$
that is,  a $C^N$ function of the second variable, vanishing to order $N$ as $r' \to 0$ and order $(n-1)/2$ as $r' \to \infty$, with values in $L^2$ (weighted by $\ang{r}^{-l})$ in the first variable. 

The composition $E^{2N} S E^{2N}$ is therefore, using the descriptions above and applying the Cauchy-Schwartz inequality,
in the space \eqref{E2NSE2N}.
\end{proof}

It follows from the Lemma and \eqref{SES} that $S$ lies in the sum of the spaces $\Psi^0_1 + \Psi^\infty$ and \eqref{E2NSE2N} for every $N$. But the intersection of these spaces is just $\Psi^0_1 + \Psi^\infty$,
so we conclude that $S \in \Psi^0_1 + \Psi^\infty$.

The exact outgoing resolvent kernel on $Z$ is $G' + G'S$. 
We need to determine the nature of the correction term $G'S$. Recall that $G'
$ is the sum of a pseudodifferential operator $G_1$ and a distribution $G - G_1$ that is Legendrian at $\bfc, \lb, \rb$ and polyhomogeneous at the remaining boundaries. It is not hard to check that 
$$ 
G_1 \circ \Psi^0_1 \subset \Psi^0_1, \quad G_1 \circ \Psi^\infty \subset \Psi^\infty.
$$
Therefore $G_1 \circ S \in \Psi^0_1 + \Psi^\infty$. 

The composition of $G - G_1$ with $S$ can be analyzed using Melrose's Pushforward Theorem \cite{cocdmc}. To do this, we view the composition as the result of lifting the kernels of $G - G_1$ and $S$ to the b-`triple space' 
\begin{multline*}
\ZZZb = \big[ Z^3; (\partial_0 Z)^3; \partial_0 Z \times \partial_0 Z  \times Z; \partial_0 Z \times Z \times \partial_0 Z  ; Z \times \partial_0 Z \times \partial_0 Z ; \\ (\partial_\infty Z)^3; \partial_\infty Z \times \partial_\infty Z  \times Z; \partial_\infty Z \times Z \times \partial_\infty Z  ; Z \times \partial_\infty Z \times \partial_\infty Z \big]
\end{multline*}
in which all of the boundary hypersurfaces of codimension 2 and 3 that meet the diagonal are blown up; see for example \cite[Section 23]{scatmet}. This space has three stretched projections $\pi_L, \pi_C, \pi_R$ to $\MMb$ according as they omit the left, centre or right variable, respectively. 
The product of kernels $A$ and $B$ on $\MMb$ can be represented
as
$$
A \circ B = (\pi_C)_* \Big( \pi_R^* A  \cdot \pi_L^* B \Big).
$$
Let $A = e^{-ir} (G - G_1)$ and $B =  S e^{-ir'} $. Although the kernel of $A$ is Legendrian at some boundary faces, we claim that the  product of
$\pi_R^* A$ and $\pi_L^* B$ on $\ZZZb$ is 
polyhomogeneous conormal. To see this, notice that the centre variable of $\ZZZb$ is simultaneously the right variable of $A$ and the left variable of $B$. The places where $A$ is Legendrian is at $\bfc$ and $\rb$, which is where the right variable of $A$ goes to infinity, but since the kernel of $B$ is rapidly decreasing when the left variable of $B$ goes to infinity, the Legendrian behaviour is killed when these kernels are multiplied on $\ZZZb$. 
The pushforward theorem \cite[Theorem 5]{cocdmc} then shows that $e^{-ir} (G - G_1) S e^{-ir'}$ is polyhomogeneous on $\ZZb$ with index sets starting at $1$ at $\zf$, $(n-1)/2$ at $\lb$ and $\rb$,  $\infty$ at $\bfc$, and $\nu_0$ at $\lbo$ and $\rbo$.  Therefore $G S$ is a kernel satisfying the conditions of $R_4$ in the statement of Theorem~\ref{conic}. This completes the proof of Theorem~\ref{conic}.


\section{Low energy resolvent construction}\label{lerc} 
We now construct the low energy asymptotics of the outgoing resolvent $R(\lambda + i0) = (P - (\lambda + i0)^2)^{-1}$ of $P =\Delta+V$, 
for an asymptotically conic manifold $(M,g)$, with $g$ as in 
\eqref{metricconic} and with potential function $V$ as in \eqref{hyp2}, \eqref{hyp3}.

It is convenient to split this construction into two parts, the `b'-part and the `scattering' part. To do this we choose a cutoff function $\chi$ such that $\chi(t) = 1$ for $t \leq 1$ and $\chi(t) = 0$ for $t \geq 2$. We then look for two kernels $\Gsc$ and $\Gb$, solving the equations
\begin{equation}
(P - \lambda^2) (\Gsc) = \chi(x/\lambda), \quad 
(P - \lambda^2) (\Gb) = 1 - \chi(x/\lambda)
\end{equation}
where the functions on the right hand side act as multiplication operators. We shall continue to use the notation $\rho = x/\lambda$, $\rho' = x'/\lambda$, $\sigma = x/x' = \rho/\rho'$. 

\subsection{Construction of $\Gb$}\label{sect:constructionGb}

We start with $\Gb$, which is an approximation to $R(\lambda + i0) (1-\chi(\rho'))$. This is supported away from $\bfc$ and $\rb$, see figure \ref{support}. Our ansatz is that $e^{-i/\rho} \Gb$ is conormal at the diagonal $\Diagb$ and polyhomogeneous conormal at the remaining faces. We specify $\Gb$ by giving a certain number of  compatible models for $e^{-i/\rho} \Gb$  at each of these faces. We use the boundary defining function $\lambda$ for (the interior of)  $\zf, \lbo, \rbo, \bfo$ and write $(\Gb)^k_\bullet$ for the coefficient of $\lambda^k$ in the expansion of $\Gb$ at these faces. For the leading order coefficient we have $(\Gb)^k_\bullet = \lambda^{-k} \Gb |_{\bullet}$ (note that the operation of restriction of a half-density has the effect of cancelling a factor of $d\lambda/\lambda$). 

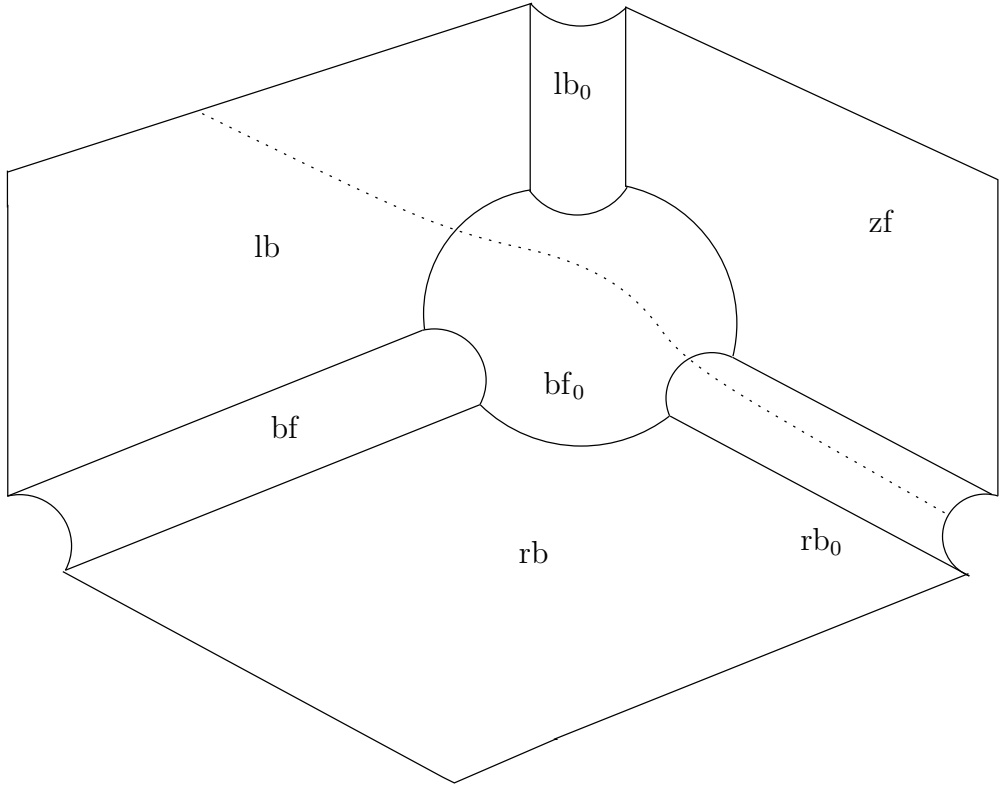
\begin{figure}[ht!]
\begin{center}
\input{support.pstex_t}
\caption{The support of $G_b$ at the boundary is located on the right of the dashed line.}
\label{support}
\end{center}
\end{figure}

\subsubsection{Terms at $\zf$}
Similarly to the previous section, and following \cite{GH1}, we write $\Delta+V=xP_bx$ for some $b$-elliptic operator in the sense of \cite{APS}
\[P_b=-(x\pl x)^2+\Delta_{\pl M}+(n/2-1)^2+V_0(y) + xW, \quad W\in {\rm Diff}^2_b(M)\] 
where $\Delta_{\pl_M}$ denotes the Laplacian on  the boundary $\pl M$ equipped with the metric $h(0)$ and ${\rm Diff}^m_b(M)$ denotes
the space of $m$-th order b-differential operators, i.e.\ those obtained from the enveloping algebra of the Lie algebra of smooth vector fields
tangent to $\pl M$. (Here we are writing derivatives with respect to the flat connection annihilating the half-density $|dg_b|^{1/2} = |x^n dg|^{1/2}$.)

The theory of b-elliptic operators given by Melrose \cite[Sec. 5.26]{APS} (see also the discussion in \cite{GH1}) shows that there is a generalized inverse $Q_b$, which is a b-pseudodifferential operator of order $-2$, for the operator $P_b$ on $L^2_b$, such that 
\begin{equation*}
P_b Q_b = Q_b P_b = \Id -\Pi_b.
\end{equation*}
where  $\Pi_b$ is orthogonal projection on the $L_b^2$ kernel of $P_b$. A zero mode of $P_b$ would be either a zero mode, or a zero-resonance, of $P$. However, assumption \eqref{hyp3} is that $P$ has no resonance nor eigenvalue at $0$. Hence we have 
$P_b Q_b = \Id$. 
The kernel $Q_b$ is conormal at the b-diagonal $\Delta_{k, \sca} \cap\zf$, 
uniformly up to $\zf \cap \bfo$ (as a multiple of the half-density $|dg_b dg_b'|^{1/2}$), and is polyhomogeneous conormal  to the three boundary faces of 
$\zf$ (ie. $\rbo$, $\lbo$ and $\bfo$). The index sets giving the exponents and logs in the expansion at the faces are 
$\mc{E}(Q_b)=(\mc{E}_{\bfo}(Q_b), \mc{E}_{\rbo}(Q_b), \mc{E}_{\lbo}(Q_b))$ where 
$\mc{E}_{\rbo}(Q_b)=\mc{E}_{\lbo}(Q_b)$ are a logarithmic extension  of the index set 
$$
\big\{ \big(\nu_j + k, 0\big) \mid , k\in\nn_0, \  \nu_j^2 \in  \operatorname{spec} \Delta_{\partial M} + V_0 + (n/2 - 1)^2 \big\}
$$
and \[ \mc{E}_{\bfo}=\nn_0\x \{ 0 \}.\]  
We set
$({\Gb})_{\zf}^0 =  x^{-1} Q_bx^{-1}$.

\subsubsection{Term at $\bfo$}
As noted above, the face $\bfo$ of $\MMksc$ is canonically the same as $Z^2_{b,\sca}$ where $Z=(0,\infty)_\rho \x \pl M$. 
Following Subsection 3.4 of \cite{GH1}, the operator $P-\la^2$ vanishes at order $2$ at $\bfo$ as a b-differential operator 
on $M^2_{k,b}$ and the induced operator at $\bfo$, 
$I_{\bfo}(\la^{-2}(\Delta + V-\la^2))$, is given by the operator $\Pconic -1$ acting on the left variable on $Z^2_{b,\sca}$.  We therefore set 
$({\Gb})_{\bfo}^{-2}$ to be $1 - \chi(\rho')$ times its inverse constructed in Section \ref{exactcone}: 
\begin{equation}\label{bfo}\begin{gathered}
(\Gb)^{-2}_{\bfo}:=  (1 - \chi(\rho')) (\Pconic - (1+i0))^{-1}  = 
   \frac{i\pi  (1-\chi(\rho'))}{2\rho\rho'} \\ \times  \sum_{j=0}^\infty\Pi_{E_j}(y,y')\Big( J_{\nu_j}(\frac{1}{\rho}) \Ha^{(1)}_{\nu_j}(\frac{1}{\rho'})H(\rho-\rho')+
J_{\nu_j}(\frac{1}{\rho'})\Ha^{(1)}_{\nu_j}(\frac{1}{\rho})H(\rho'-\rho)\Big).
\end{gathered}\end{equation}
This matches with $G_{\zf}^{0}$ at $\zf\cap\bfo$, since, as shown in the previous section, the operator $\Pconic - 1$ can be written as $\rho \, \Pbconic \, \rho$ where $\Pbconic$ is as in \eqref{Pb}, and the normal operator of $\Pbconic$ (in the sense of the b-calculus) agrees with the normal operator $N(P_b)$ of $P_b$ defined at $\zf$. We have seen that both the normal operators of $x (\Gb)^0_{\zf} x$
and of $\rho^{-1} (\Gb)^{-2}_{\bfo} \rho^{-1}$ are the inverse of $N(P_b)$, which is precisely the matching conditions for these two models. 

\subsubsection{Terms at $\rbo$ and $\lbo$} 
We take as the leading terms at $\rbo$ and $\lbo$, the models given in \cite[Section 4.5]{GH1} with $k$ replaced by $i\lambda$, which amounts to replacing the modified Bessel function $K_{\nu_j}$ in \cite{GH1}  by the Hankel function $\Ha_{\nu_j}^{(1)}$. We also have to multiply by $1-\chi(\rho')$ (which only affects $\rbo$). Therefore we take, for all $\nu_j \leq 1$, 
\begin{equation}\begin{gathered}
({\Gb})_{\lbo}^{\nu_j - 1} = \frac{i \pi}{2} (x' \rho)^{-1} v_j(y,z') \Ha_{\nu_j}^{(1)}(1/\rho)  \Big| \frac{d\rho dy}{\rho}dg'_b \Big|^{1/2},  \\
({\Gb})_{\rbo}^{\nu_j - 1} = \frac{i \pi}{2}(1-\chi(\rho')) (x \rho')^{-1} v_j(z,y')  \Ha_{\nu_j}^{(1)}(1/\rho')  \Big| \frac{d\rho' dy'}{\rho'}dg_b \Big|^{1/2}
\end{gathered}\label{Gbrbo}\end{equation}
where $v_j(z,y')$ is the unique function on $M \times \pl M$ such that 
\begin{equation}
P_b v_j = 0, \quad  v_j(x,y,y') = \frac{\Pi_{E_j}(y, y')}{2^{\nu_j} \Gamma(\nu_j + 1)} x^{-\nu_j} + O(x^{-\nu_j - 1} \log x), \quad x \to 0. 
\label{vjdefn}\end{equation}
The existence and uniqueness of $v_j$, and the matching of  terms \eqref{Gbrbo} with the leading models at $\zf$ and $\bfo$ is shown in \cite{GH1}. 

\subsubsection{Terms at $\lb$}\label{termsatlb}
Let $\tau$ be the half-density in \eqref{lblbo}. 
Near $\lb$, we choose $G$ such that $e^{-i/\rho}G$ is polyhomogeneous. More precisely, we choose $G$ 
of the form $a e^{i/\rho} \rho^{(n-1)/2} |\tau|^{1/2}$
 where $a$ is polyhomogeneous, with index sets 
$\mc{F}_{\bfo} = \mc{E}_{\bfo}(Q_b) - 2, \mc{F}_{\lbo} = \mc{E}_{\lbo}(Q_b) - 1$ at $\bfo, \lbo$ and with the $C^\infty$ index set $0$ at $\lb$, and with leading behaviour at $\bfo$ and at $\lbo$ chosen to match the models already specified at those boundary hypersurfaces. This is possible since  the models at adjacent faces $\bfo$ and $\lbo$ are both of the form $e^{i/\rho}$ times a polyhomogeneous half-density; for $\bfo$ this follows from the Legendrian description of the kernel in Section~\ref{exactcone}, while for $\lbo$ it follows from standard asymptotics of Hankel functions as their
argument tends to infinity, as observed above Remark~\ref{summary}. 

A standard computation in scattering theory (see \cite[(4.30) -- (4.31)]{HV2} for example) shows that if we apply $(P - \lambda^2)$ to a kernel of the form $$\tilde a e^{i/\rho} \rho^{(n-1)/2 + k} |\tau|^{1/2},$$ with $\tilde a$ polyhomogeneous and with index set $0$ at $\lb$, the result is a kernel of the form $$\lambda^2  b e^{i/\rho} \rho^{(n-1)/2+ k + 1} |\tau|^{1/2},$$ where $b \,  |_{\lb} = -2ik a \, |_{\lb}$. Using this iteratively, we can solve away the error term at $\lb$ to infinite order. In fact, applying this with $k=0$ shows that the result of applying $P - \lambda^2$ to $$a e^{i/\rho} \rho^{(n-1)/2} |\tau|^{1/2}$$ with $a$ is as above vanishes to $3$ orders better at $\lbo$ and $\bfo$, and $2$ orders better at $\lb$. (We gain three orders at $\lbo$ and $\bfo$ since the operator itself vanishes to second order, and the leading models at these faces are killed by the corresponding induced operator, which leads to a gain of an additional order.) This error term c!
 an be solved away iteratively at $\lb$ with terms of the form 
$$\tilde a e^{i/\rho} \rho^{(n+1)/2 + k} |\tau|^{1/2}, \quad k = 2, 3, \dots
$$ 
and where $\tilde a$ vanishes $1$ order better at $\bfo$ and one order better at $\lbo$ compared to $a$, i.e.\ the correction terms do not affect the leading models at $\bfo$ and $\lbo$ at all. In this way we can remove the error term at $\lb$ completely, i.e.\ so that it vanishes to infinite order there.

We conclude that we can construct a kernel $\Gb$ such that 
$$
(P - \lambda^2) (\Gb) - (1 - \chi(\rho)) = E_b \in \phgc_{\mcE}(\MMkb; \Omega_{k,b}^{1/2}(\MMkb)), 
$$
such that $\min \mcE_{\zf} = 1$, $\min \mcE_{\bfo} = 1$, $\min \mcE_{\lbo} = \nu_0 + 2$, $\min \mcE_{\rbo} = \nu_0$, $\mcE_{\lb} = \emptyset$, and vanishing  in a neighbourhood of $\bfc, \rb$. 

\subsection{Construction of $\Gsc$}
This is supported away from $\lbo \cup \zf$. Initially, we work in a smaller region, which is supported close to $\bfc$ --- say in the region where $\rho, \rho' < \epsilon$. In these coordinates, the metric can be written 
$$
g = 
\lambda^{-2} (d\rho^2/\rho^4 + h(\lambda \rho)/\rho^2) \equiv \lambda^{-2} dg_\lambda.
$$
 Our operator $P - \lambda^2$ can be written in the $(\rho, y)$ coordinates as $\lambda^{2} (\Delta_\lambda +V_0 \rho^2 + \lambda \rho^3 W - 1)$ in the $\rho$ coordinates, where $\Delta_\lambda$ is the Laplacian with respect to the metric $g_\lambda$. Thus, our equation $(P - \lambda^2) G = \Id$ is equivalent to 
$$
(\Delta_\lambda +V_0 \rho^2 + \lambda \rho^3 W - 1) G = \lambda^{-2} \Id. 
$$
Since this operator has coefficients depending smoothly on $\lambda$ down to $\lambda = 0$ in this region, we can perform the first part of the parametrix construction in \cite{HV2} (that part in Section 4) 
uniformly in $\lambda$, obtaining a Legendre distribution polyhomogeneous in $\lambda$ with index set $-2$ at $\bfo$. We give just a sketch of this construction here, referring to \cite{HV2} for full details. 

\subsubsection{Pseudodifferential term}
We begin by choosing a pseudodifferential operator $\Gsc_1 \in \Psi^{-2, (-2, 0,0); *}(M; \Omegakbh)$,  in the calculus defined in \cite{GH1}, 
that solves away the singularity along the diagonal, in the equation
$$
(P - \lambda^2) G_1 =\chi(\rho'). 
$$
See Section 4.1 of \cite{HV2} and Section 3.1 of \cite{GH1}. 

\subsubsection{Intersecting Legendre term}
Next we consider the error term $\Esc_1 = (P - \lambda^2) \Gsc_1$. On the space $\MMkb$, it is a Legendre distribution associated to $\Nscstar \Diagb$. We can solve this away (using Proposition~\ref{ex-int} in place of Proposition 3.2 from \cite{HV2}) by adding to $\Gsc_1$ an intersecting Legendre distribution $$\Gsc_2 \in I^{-1/2, \infty, \infty; \mcB}(\MMkb, (\Nscstar \Diagb, L^{\bfc}); \Omegab),$$ obtaining a parametrix $\Gsc_2$ with an error term $E_2 \in I^{-1/2, \infty, \infty; \mcE}(\MMkb, L^{\bfc}; \Omegab)$ that is Legendre with respect to $L^{\bfc}_+$, microsupported away from $\Nscstar \Diagb$, and supported away from $\lb$ and $\rb$. Here, $\mcB_{\bfo} = -2$, but $\mc{E}_{\bfo} = 0$ (we gain two orders at $\bfo$ because the operator $P - \lambda^2$ vanishes to second order there). 
See Section 4.2 of \cite{HV2}. 

\subsubsection{Conic Legendre term}
We can solve away the error term $\Esc_2$ by adding to $\Gsc_2$ a  Legendre distribution in the space $I^{m, p ; r_{\lb} , r_{\rb} ; \mcA}(\MMkb, (L^{\bfc}, L^\sharp)$, with $m=-1/2$, $p = (n-2)/2$, $r_{\lb} = r_{\rb} = (n-1)/2$,  associated to the conic pair of Legendre submanifolds $(L^{\bfc}, L^\sharp)$ (using Proposition~\ref{ex-conic} in place of Proposition 3.5 from \cite{HV2}). We then obtain a parametrix $\Gsc_3$ with  error term $\Esc_3$  Legendre with respect to $L^\sharp$ only: $\Esc_3 \in I^{p , r_{\lb} + 2 , r_{\rb} \mc{E}; \mc{B}}(\MMkb, L^\sharp; \Omegab)$. See Sections 4.3 and 4.4 of \cite{HV2}. 

\subsubsection{Correction at $\bfo$}
Now we make a step that is absent from the argument in \cite{HV2}; we correct 
the leading behaviour at $\bfo$ to the exact conic resolvent. We observe that the pseudodifferential singularities of the parametrix constructed above as well as those at $\Nsfstar \Diagb$ and at the propagating Legendrian $L^{\bfc}_+$ are uniquely determined. By contrast, the singularities at $L^\sharp$ are not uniquely determined (although the \emph{singularities} of the symbol on $L^\sharp$ where $L^\sharp$ meets $L^{\bfc}_+$ are determined --- this subtle point is explained in \cite{MZ}). Thus, the difference
$$
F_{\bfo}^{-2} := (\Gsc_3)_{\bfo}^{-2} - (\Pconic - (1 + i0)^2)^{-1},
$$
is Legendre with respect to $L^\sharp$ only. (To clarify the notation in this expression, the superscripts $-2$ are the coefficients of $\lambda^{-2}$ at $\bfo$, while the superscript $-1$ is a power.) By Proposition~\ref{bfo-rest}, this is the boundary value of a term $F \in I^{p; r_{\lb}, r_{\rb}; \mcA}(\MMkb, L^\sharp_+; \Omegakbh(\MMkb))$. Let $\Gsc_4 = \Gsc_3 + F$. 
Now the error term $\Esc_4 = (P - \lambda^2) (\Gsc_4)$ is better: it vanishes to order $1$ at $\bfo$, i.e.\ we may now take $\mc{E}_{\bfo} = 1$.  

\subsubsection{Leading terms at $\rbo$} 
To match with $(\Gsc_4)^{-2}_{\bfo}$, we define $(\Gsc)^{\rho}_{\nu_j - 1}$, for all $j$ such that $\nu_j \leq 1$, to be given by the second line of \eqref{Gbrbo}, with the factor $1 - \chi(\rho')$ replaced by $\chi(\rho')$.

\subsubsection{Solving away outgoing errors at $\bfc$ and $\lb$}
We now solve away all outgoing errors at $\bfc$ and $\lb$ as in Section 4.5 of \cite{HV2} and Section~\ref{termsatlb} above, by adding to $\Gsc_4$ a suitable Legendre distribution associated to the outgoing Legendrian $L^\sharp$, obtaining $\Gsc_5$. Since the error term $\Esc_4$ already vanishes to order $1$ at $\bfo$, the correction terms will be at order $-1$ at $\bfo$ and therefore do not affect the leading behaviour of $\Gsc_4$ at $\bfo$ (which are at order $-2$) at all. The new error term  $\Esc_5$ is such that $e^{-i/\rho'} \Esc_5$ is polyhomogeneous on $\MMkb$ and vanishes to order order $1$ at $\bfo$, order $\infty$ at $\lb$ and $\bfc$,  order $\nu_0$ at $\rbo$ and $(n-1)/2$ at $\rb$.

\subsection{Correction term and true resolvent}
Let our parametrix $G$ be given by $G = \Gb + \Gsc_5$. Then the error term $E = (P - \lambda^2) G - \Id$ is such that $e^{-i/\rho'} E \in \mathcal{A}_{\mc{E}}(\MMkb; \Omegakbh(\MMkb))$ is polyhomogeneous on $\MMkb$ such that $\min \mc{E}_{\zf} = 1$,
$\min \mc{E}_{\bfo} = 1$, $\min \mc{E}_{\lbo} = \nu_0 + 2$,
$\min \mc{E}_{\rbo} = \nu_0$, 
$\min \mc{E}_{\rb} = (n-1)/2$, and $\mc{E}_{\bfc} = \mc{E}_{\lb} = \emptyset$. 
If we express this in terms of a half-density of the form $|dg_b dg_b' d\lambda/\lambda|^{1/2}$, which lifts to $\MMkb$ to be a smooth nonvanishing b-half-density, then the orders of vanishing are 
as above except at $\rb$, where it changes to $-1/2$. 
 (Note: by changing to a b-half-density the order of vanishing is reduced by $n/2$ at $\lb$ and $\rb$, and by $n$ at $\bfc$; however, because we already have infinite order vanishing at $\lb$ and $\bfc$, only the change at $\rb$ is visible.) Since b-half-densities are square-integrable precisely when the order of vanishing is positive, it follows that the error term $E$ is Hilbert-Schmidt on $x^l L^2(M)$ for every $l > 1/2$ and each $\lambda > 0$, with the Hilbert-Schmidt norm tending to zero as $\lambda \to 0$. In particular, it is compact, and invertible for sufficiently small $\lambda$. 
Consequently the true resolvent $R(\la) = (P - (\lambda + i0)^2)^{-1}$ is given by $G(\Id + E)^{-1}$.  

We next analyze the structure of the correction term. Define $S$ (for sufficiently small $\lambda$) by  
\begin{equation}
(\Id + E)^{-1} = \Id + S;
\label{Sdefn}\end{equation}
 then the correction term is $GS$. 
Let us define $E_{\phg}$ to be the kernel $E$ conjugated by $e^{i/\rho}$: $E_{\phg} = e^{i/\rho} E e^{-i/\rho'}$. 
Also let $S_{\phg} = e^{i/\rho} S e^{-i/\rho'}$. 
We have, for any $N$, 
$$
S = \sum_{j=1}^{2N} (-1)^j E^j + E^N S E^N \implies S_{\phg} = \sum_{j=1}^{2N} (-1)^j E_{\phg}^j + E_{\phg}^N S_{\phg} E_{\phg}^N. 
$$
By the results of \cite{GH1}, 
$E_{\phg}^j$ is polyhomogeneous conormal with index family $\mcE^{(j)}$ where index family $\mcE^{(j+1)}$ is given at $\bfo, \lbo, \rbo, \zf$ inductively by
\begin{equation}\begin{aligned}
(\mc{E}^{(j)}_{\lbo}+\mc{E}_{\zf})&\bar{\cup}(\mc{E}^{(j)}_{\bfo}+\mc{E}_{\lbo}) \text{ at } \lbo, \\
(\mc{E}^{(j)}_{\rbo}+\mc{E}_{\bfo})&\bar{\cup}(\mc{E}^{(j)}_{\zf}+\mc{E}_{\rbo}) \text{ at } \rbo, \\
(\mc{E}^{(j)}_{\bfo}+\mc{E}_{\bfo})&\bar{\cup}(\mc{E}^{(j)}_{\lbo}+\mc{E}_{\rbo}) \text{ at } \bfo, \text{ and }  \\
(\mc{E}^{(j)}_{\zf}+\mc{E}_{\zf})&\bar{\cup}(\mc{E}^{(j)}_{\rbo}+\mc{E}_{\lbo}) \text{ at } \zf.
\end{aligned}\label{Ejformulae}\end{equation}
From \eqref{Ejformulae} it is straightforward to prove by
induction that 
\begin{equation}\begin{gathered}
\min \mc{E}^{(j)}_{\lbo} \geq \nu_0 + 2 + j,  \quad 
\min \mc{E}^{(j)}_{\bfo} \geq 1+j, \\
\min \mc{E}^{(j)}_{\rbo} \geq \nu_0+j, \quad 
\mc{E}^{(j)}_{\zf} \geq j,   \\
\mc{E}^{(j)}_{\bfc} = \mc{E}^{(j)}_{\lb} = \emptyset, \quad \mc{E}^{(j)}_{\rb} = (n-1)/2. 
\end{gathered}\label{Ejinduction}\end{equation}
It follows that the index family $\Ebar$ defined by $\Ebar_\bullet = \cup_j \mc{E}^{(j)}_\bullet$ is well-defined. We show

\begin{lem} The kernel $S_{\phg}$ is polyhomogeneous on $\MMkb$ with index family $\Ebar$. 
\end{lem}

\begin{proof} Let $Q_N$ be the differential operator 
$$
Q_N = \chi(\rho') \prod_{j=0}^{N-1} \big(\rho'\frac{d}{d\rho'} - \frac{n-1}{2} - j \big) , \quad \rho' = \frac{x'}{\lambda},
$$
where $\chi(\rho')$ is a smooth function equal to $1$ for $\rho' \leq 1$ and $0$ for $\rho' \geq 2$. 
This operator has the property that it maps a function of $\rho'$ of the form 
${\rho'}^{(n-1)/2} \CI(\rho')$ into a function of the form ${\rho'}^{(n-1)/2 + N} \CI(\rho')$, that is, it kills the first $N$ terms of the expansion of a function in 
${\rho'}^{(n-1)/2} \CI(\rho')$ at $\rho' = 0$. 
Using the definition of polyhomogeneous conormality given in \cite{cocdmc}, it is enough to show, for every positive integer $N$, that $S_{\phg}$ can be written as a sum $S_{\phg, N, 1} + S_{\phg, N, 2}$, where $S_{\phg, N, 1} \in \mathcal{A}_{\Ebar}(\MMkb; \Omegakbh)$, and $Q_N S_{\phg, N, 2}$ is conormal (as opposed to polyhomogeneous conormal) on $\MMkb$ with respect to a multiweight $\frak{r} = \frak{r}_N$, all of whose entries tend to infinity with $N$. 

To do this, we write 
\begin{equation}
S_{\phg} = \sum_{j=1}^{2N} (-1)^j E_{\phg}^j + E_{\phg}^{N} S_{\phg} E_{\phg}^{N} := S_{\phg, N, 1} + S_{\phg, N, 2}.
\label{SEt}\end{equation}
Clearly, $S_{\phg, N, 1}$ is polyhomogeneous conormal with respect to the index set $\Ebar$. We claim that $Q_N S_{\phg, N, 2}$ is conormal with respect to multiweights\footnote{Note that the rate of decay of $E_{\phg}^N S_{\phg} E_{\phg}^N$ is $N + O(1)$ at all boundary hypersurfaces except for $\rb$, as $N \to \infty$, so the purpose of $Q_N$ is to force an improvement of the rate of decay at $\rb$.} $r_{\lbo} = r_{\rbo} = r_{\bfo} = r_{\zf} = r_{\lb} = r_{\rb} = N$, $r_{\bfc} = 2N$.\footnote{It would actually be true with $r_{\lbo} = \nu_0 + N + 2$, $r_{\rbo} = \nu_0 + N$, $r_{\bfo} = 1 + N$, $r_{\zf} = N$, $r_{\bfc} = r_{\lb} = \infty$, $r_{\rb} = (n-1)/2 + N$ but it is sufficient, and easier, to show the weaker claim.}  We can take
$$
\Big(  r_{\lbo} r_{\rbo} r_{\bfo} r_{\zf} r_{\lb} r_{\rb} \Big)^N \rho_{\bfc}^{2N} = \lambda^N \langle \frac{x}{\lambda} \rangle ^N \langle \frac{x'}{\lambda} \rangle ^N,
$$

To show the conormality of $Q_N S_{\phg, N, 2}$ with respect to these multiweights, we consider $m$ vector fields $W_1, \dots, W_m$ on $\MMkb$, tangent to the boundary, where $m \in \nn$ is arbitrary. We must show that 
\begin{multline}
W_1 \dots W_m \big( Q_N S_{\phg, N, 2} \big) \in \Big(  r_{\lbo} r_{\rbo} r_{\bfo} r_{\zf} r_{\lb} r_{\rb} \Big)^N \rho_{\bfc}^{2N} L^\infty(\MMkb) \\ =\lambda^N \langle \frac{x}{\lambda} \rangle ^N \langle \frac{x'}{\lambda} \rangle ^N L^\infty(\MMkb).
\label{tobeproved}\end{multline}

We next observe that vector fields on $\MMkb$  tangent to the boundary are generated, over $\CI(\MMkb)$, by b-vector fields on $M$, lifted to $\MMkb$ by either the left or the right projection, and by $\lambda \partial_\lambda$. So it is sufficient to prove \eqref{tobeproved} when the $W_i$ are generating vector fields as just described. Notice that, writing $S_{\phg, N, 2} = E_{\phg}^{N} S_{\phg} E_{\phg}^{N}$, the $W_i$ lifted from $M$ by the left, resp. right, factor act on the left, resp. right, factor of $E_{\phg}^{N}$.  Notice also that the operator $Q_N$ acts only on the right factor of $E_{\phg}^{N}$ and increases the order of vanishing at $\rb$ to order $(n-1)/2 + N$. However, the vector field $\lambda \partial_\lambda$ acts on all three factors of $S_{\phg, N, 2}$, including the middle factor $S_{\phg}$.

We have already seen that $S_{\phg}$ is, for each fixed $\lambda > 0$, a Hilbert-Schmidt operator on $x^l L^2(M)$, $l > 1/2$, with uniformly bounded Hilbert-Schmidt norm. The same is true for $(\lambda \partial_\lambda)^j S_{\phg}$ for every $j$. In fact, using the identity 
$$
(\Id + E_{\phg}) (\Id + S_{\phg}) = \Id \implies E_{\phg} + E_{\phg}S_{\phg} + S_{\phg} = 0,
$$
we compute 
$$
\lambda \partial_\lambda S_{\phg} = - (\Id + S_{\phg}) \Big( \lambda \partial_\lambda E_{\phg} \Big) (\Id + S_{\phg}) ,
$$
from which it follows (using the polyhomogeneity of $E_{\phg}$) that 
$\lambda \partial_\lambda S_{\phg}$ is Hilbert-Schmidt  on $x^l L^2(M)$, $l > 1/2$, with uniformly bounded Hilbert-Schmidt norm. Proceeding inductively we can deduce this for $(\lambda \partial_\lambda)^j S_{\phg}$ for every $j$.

Now we write $z, z', w, w'$ for points in $M$, with $x(z), x(w)$, etc, denoting the corresponding boundary defining functions, and express the kernel of $Q_N S_{\phg, N, 2}$ as 
\begin{multline*}
Q_N S_{\phg, N, 2}(\lambda, z, z') = 
\int\limits_{M \times M}  \Big( x(w)^{l} E_{\phg}^{N}(\lambda, z, w) \big( x(w')^{-l} (Q_N)_{z'} E_{\phg}^{N}(\lambda, w', z') \big) \Big)  \\ \times \Big( x(w)^{-l} x(w')^{l} S_{\phg}(\lambda, w, w') \Big)  \, dg(w) \, dg(w'), \quad l > \frac1{2}. 
\end{multline*}
We then apply the vector fields $W_i$ (assumed without loss of generality to be either b-vector fields on $M$ lifted from the left or right factors, or the vector field $\lambda \partial_\lambda$) to 
this expression, and multiply by the factors 
\begin{equation}
\lambda^{-N} \langle \frac{x}{\lambda} \rangle ^{-N} \langle \frac{x'}{\lambda} \rangle ^{-N}
\label{rhofactors}\end{equation}
 from \eqref{tobeproved}. Using Cauchy-Schwarz (taking advantage of the Hilbert-Schmidt property of kernels $x(w)^{-l} x(w')^{l} S_{\phg}(\lambda, w, w')$, etc), and using the conormality of $E_{\phg}^N$ to absorb the factors from \eqref{rhofactors}, we see  that \eqref{tobeproved} is satisfied. 

\end{proof}

Now we analyze $GS$. We have 
\begin{equation}
 e^{-i/\rho} GS e^{-i/\rho'} = \Big( e^{-i/\rho} G e^{-i/\rho} \Big) S_{\phg}.
\label{GScomp}\end{equation}
 Note that $e^{-i/\rho} G e^{-i/\rho} $ is not polyhomogeneous at $\bfc$ (that is, as both $\rho$ and $\rho'$ tend to zero), but is at all other boundary hypersurfaces. However, after composing with $S_{\phg}$, this non-polyhomogeneity is killed by the rapid decrease of the kernel of $S_{\phg}$ at $\bfc$ and $\lb$, i.e.\ as its left $\rho$ variable tends to zero (just as in the discussion in the last paragraph of Section~\ref{exactcone}). Hence we can apply \cite[Proposition 2.10]{GH1} to the composition \eqref{GScomp}. We find that 
 $e^{-i/\rho} GS e^{-i/\rho'}$ is polyhomogeneous and vanishes to order  
 (at least) $1$ at $\zf$ and  $\bfo$, order $\nu_0 + 1$ at $\lbo$ and $\rbo$, $(n-1)/2$ at $\lb$ and $\rb$ and $n-1$ at $\bfc$. 
In particular, this correction term vanishes to one order higher than $G$ at $\bfo, \zf, \lbo, \rbo$. Therefore, writing $R_{\bullet}^k$ for the coefficient of $\lambda^k$ in the expansion of the resolvent at  $\bullet$, we have, with $v_0$ given by \eqref{vjdefn},
\begin{equation}\begin{aligned}
R_{\bfo}^{-2} &= G_{\bfo}^{-2} = \big(P_{\conic} - (1 + i0)^2\big)^{-1}; \\
R_{\zf}^0 &= G_{\zf}^{-1} = (x x')^{-1}  P_b^{-1}; \\
R_{\lbo}^{\nu_0 - 1} &= G_{\lbo}^{\nu_0 - 1} = \frac{i \pi}{2} (x' \rho)^{-1} v_0(y,z') \Ha_{\nu_0}^{(1)}(1/\rho)  \Big| \frac{d\rho dy}{\rho}dg'_b \Big|^{1/2},  \\
R_{\rbo}^{\nu_0 - 1} &= G_{\rbo}^{\nu_0 - 1} =\frac{i \pi}{2}(x \rho')^{-1} v_0(z,y')  \Ha_{\nu_0}^{(1)}(1/\rho')  \Big| \frac{d\rho' dy'}{\rho'}dg_b \Big|^{1/2}.
\end{aligned}\label{resmodels}\end{equation}

 
 \section{Spectral measure}
 In this section we study the spectral measure $dE_{P_+^{1/2}}(\lambda)$ of the operator $P_+^{1/2}$ and prove Theorem~\ref{mainsm}. The spectral measure is related to the resolvents $R(\lambda \pm i0)$ by
\begin{equation}
dE_{P_+^{1/2}}(\lambda) = \frac{d}{d\lambda} E_{P_+^{1/2}}(\lambda) \, d\lambda = \frac{\lambda}{\pi i} \Big( R(\lambda + i0) - R(\lambda - i0) \Big) \, d\lambda, \quad \lambda > 0.
\label{subtraction}\end{equation}
The resolvent kernel $R(\lambda \pm i0)$ is invariant under involution, i.e.\ $R(\lambda \pm i0)(z,z') = R(\lambda \pm i0)(z',z)$, and the formal adjoint of $R(\lambda + i0)$ is $R(\lambda - i0)$, i.e.\ 
$$
R(\lambda - i0)(z,z') = \overline{R(\lambda + i0)(z', z)}.
$$
It follows that the spectral measure can be expressed 
\begin{equation}
dE_{P_+^{1/2}}(\lambda)(z,z') = \frac{2\lambda}{\pi} \Im \big( R(\lambda + i0)(z,z') \big) \, d\lambda, \quad \lambda > 0.
\label{smim}\end{equation}
We now discuss cancellations that occur when the two resolvent kernels are subtracted. 

\subsection{Behaviour at diagonal}\label{behavdiag}
The diagonal singularity of the resolvent kernel is completely determined by the full symbol of $P$. Consequently, the diagonal
singularity cancels in the expression \eqref{subtraction}, and the spectral measure is smooth across the diagonal. Another way to see this is that the spectral measure satisfies an elliptic equation $(P - \lambda^2) dE(\lambda) = 0$, so it cannot have any local singularities.

Moreover, the difference between the outgoing ``$R_2$" piece and the incoming  ``$R_2$" piece (in the terminology of Theorem~\ref{mainres}), which are intersecting Legendre distributions associated to $(\Nsfstar \Diagb, L^{\bfc}_+)$ and $(\Nsfstar \Diagb, L^{\bfc}_-)$ respectively, is a Legendrian associated to the propagating Legendrian $L^{\bfc}$ alone. This is because the left Hamilton vector field \eqref{Vl} is nonzero at $L^{\bfc} \cap \Nsfstar \Diagb$, and tangent to $L^{\bfc}$. Since Legendrian regularity propagates on $L^{\bfc}$ along the Hamilton vector field, the spectral measure is Legendre across $\Nsfstar \Diagb$, i.e.\ it is a Legendre distribution on $\MMkb$ associated to the intersecting pair of Legendre submanifolds with conic points $(L^{\bfc}, L^\sharp_+ \cup L^\sharp_-)$.

\subsection{Leading asymptotics at $\bfo, \lbo,\rbo$}
We start by giving the leading asymptotics of the spectral measure at the boundary hypersurfaces $\bfo, \lbo,\rbo$. 
At $\bfo$ (and also at $\lbo,\rbo$) we have already determined the leading asymptotic of the outgoing resolvent --- see \eqref{resmodels}. Since the imaginary part of $i \Ha_{\nu}^{(1)}(r)$ is $J_{\nu}(r)$, we see from \eqref{smim} that the leading asymptotic of the spectral measure $dE_{P_+^{1/2}}(\lambda)$ at $\bfo$ is at order $-1$, and is given by ${\rm Im}(G_{\bfo}^{-2})$, that is 
\begin{equation}\label{sbfo}
(dE)_{\bfo}^{-1} = dE_{\Pconic^{1/2}}(1) = 
r r'  \sum_{j=0}^\infty\Pi_{E_j}(y,y')J_{\nu_j}(r)J_{\nu_j}(r')  \left|\frac{dr \, dr'}{rr'}\right|^\demi .
\end{equation}
Similarly, at $\lbo$, $2/\pi$ times the imaginary part of the resolvent kernel yields, as the leading asymptotic of the spectral measure at $\lbo$,
\begin{equation}\label{slbo}
(dE)_{\lbo}^{\nu_0} = (x' \rho)^{-1} v_0(y,z') J_{\nu_j}(1/\rho)  \Big| \frac{d\rho dy}{\rho}dg'_b \Big|^{1/2},  
\end{equation}
where $v_0$ is as in \eqref{vjdefn}. The $\rbo$ leading asymptotic has the same form by symmetry of the resolvent kernel.
\begin{equation}\label{srbo}
(dE)_{\rbo}^{\nu_0} = (x \rho')^{-1} v_0(y',z) J_{\nu_j}(1/\rho')  \Big| \frac{d\rho' dy'}{\rho'}dg_b \Big|^{1/2}. 
\end{equation}
 
\subsection{Expansion at $\zf$}\label{exp}
We will show that the expansion of the resolvent kernel at $\zf$ is real up to the term at order  $\la^{2\nu_0}$. 
An immediate consequence is a determination of the rate of vanishing of the  spectral measure at $\zf$:

\begin{prop}\label{cancellation} The spectral measure $dE_{P_+^{1/2}}(\lambda)$ for the operator $P_+^{1/2}$ vanishes to order $2\nu_0 + 1$ at $\zf$. More particularly, it has expansion at $\zf$ of the form
\[dE_{P_+^{1/2}}(\lambda)= \Big(\la^{2\nu_0+1} w(z)  w(z') |dg dg'|^{1/2} + O(\la^{\min(2\nu_0+2,2\nu_1+1)}) \Big)d\la \]
where $w$ is a solution of $P w = 0$, and $ w \sim x^{n/2 - 1 - \nu_0} W(y)$ as $x\to 0$ for some smooth function $W$ on $\pl M$.
\end{prop}

\begin{proof} 
We first show that we may choose the parametrix $G$ so that its expansion at $\zf$ up to order $\la^{2\nu_0}$  is real. This is certainly true for the pseudodifferential part of $G$; indeed, the real part of any pseudodifferential parametrix is also a pseudodifferential parametrix, since $P$ has real coefficients.  
Now consider the expansion of ${\rm Im}(G_{\bfo}^{-2})$ at the face $\zf$. 
We see from the expansion of modified Bessel functions $J_{\nu}(z)$ at $z=0$ and \eqref{sbfo}
that the imaginary part of $G_{\bfo}^{-2}$ has index set $\mc{Z}$ at $\zf$ 
of the form $\mc{Z}=\cup_{\nu_j\leq \nu_0+1/2}(2\nu_j+2) \cup \mc{Z}'$ where $\mc{Z}'\geq 2\nu_0+3$ and has no log terms; in particular if 
$\pl M=S^{n-1}$ and  $V_0=0$, one has $\mc{Z}=n+\nn_0$. 
Similarly, we see from \eqref{srbo} and \eqref{slbo} that $G_{\lbo}^{\nu_0 - 1}$ and $G_{\rbo}^{\nu_0 - 1}$ have a real expansion at $\zf$ up to order $\nu_0 + 1$ and their imaginary part has index set at $\zf$ of the form $\mc{Y}=\cup_{\nu_j\leq \nu_0+1}(\nu_j+1)\cup \mc{Y}'$ 
where $\mc{Y}'\geq \nu_0+2$ is an index set with no log terms and  $\mc{Y}=n/2+\nn_0$ if $\pl M=S^{n-1}$ and  $V_0=0$. Also, $G_{\zf}^0$ is real, and the higher order terms $G_{\zf}^\alpha$, for $\alpha > 0$, can be chosen arbitrarily provided that they are compatible with terms specified at $\bfo, \lbo, \rbo$. 
It follows that we may choose $G$ so that the expansion at $\zf$ is real up to order $\lambda^{2\nu_0}$ (in the sense that $\lambda^{2\nu_0}$ is the first non-real term occurring), actually  with imaginary part having index set at $\zf$ 
of the form $X=\cup_{\nu_j\leq \nu_0+1/2}2\nu_0\cup \mc{X}'$ where $\mc{X}'\geq 2\nu_0+1$ has no log terms and 
$\mc{X}=n-2+\nn_0$ if $\pl M=S^{n-1}$ and $V_0=0$. We will choose $G$ so that its imaginary part has expansion at $\zf$
\begin{equation}\label{imG}
{\rm Im}(G)=\la^{2\nu_0}A(z,z')+O(\la^{\min(2\nu_0+2,2\nu_1)})
\end{equation}
for some smooth kernel $A$ on $\zf$ satisfying $PA=0$, which can be done as long as the expansion 
matches with the imaginary parts of $G^{\nu_0-1}_{\rbo},G^{\nu_0-1}_{\lbo},G_{\bfo}^{-2}$ at $\rbo,\lbo,\bfo$. 
To determine the kernel $A(z,z')$ on $\zf$, we note that the eigenvalue $\nu_0^2$ for 
$\Delta_{\partial M} + V_0 + (n/2 - 1)^2$ is simple, and the eigenfunction does not change sign, so there is a unique positive normalized eigenfunction $W(y) |dh|^{1/2}$ for this eigenvalue. Therefore, using Melrose's b-calculus as outlined in \cite[Section 2.1]{GH1}, there is a unique half-density $\tilde w |dg_b|^{1/2}$ such that $P_b \tilde w = 0$ and $\tilde w(x, y) = x^{-\nu_0} W(y) + O(x^{-\nu_0 + \epsilon})$ as $x \to 0$ for some $\epsilon > 0$; moreover, this is the only solution to $P_b u = 0$ (up to scaling) in the space $x^{-\nu_0 - \epsilon} L^2_b$, if $\epsilon > 0$ is sufficiently small, and we have $v_0(z, y') = \tilde w(z) W(y)$. We then set
$$
A(z,z') = (xx')^{-1} \tilde w(z) \tilde w(z') |dg_b dg_b'|^{1/2}.
$$
and from the expansion of \eqref{sbfo}, \eqref{srbo}, \eqref{slbo} at $\zf$, we see that the expansion of $\la^{2\nu_0}A(z,z')$ 
matches with the asymptotic of \eqref{sbfo}, \eqref{srbo}, \eqref{slbo} at $\zf$. Moreover, if $\nu_1\geq \nu_0+1$, 
the expansion of \eqref{sbfo}, \eqref{srbo}, \eqref{slbo} at $\zf$ involve no power of $\la$ between $\la^{2\nu_0}$ and 
$\la^{\min(2\nu_0+2,2\nu_1)}$, therefore $G$ can be chosen so that ${\rm Im}(G)$ satisfies \eqref{imG}.

Since $P$ has real coefficients, the error term $E=(P-\la^2)G-\Id$ has the same property as $G$, i.e.\
it has a real expansion at $\zf$ in powers of $\la$ up to $\la^{2\nu_0}$, and since $PA=0$, 
the expansion at $\zf$ of $(P-\la^2){\rm Im}(G)={\rm Im}(E)$ is actually of order $\min(2\nu_0+2,2\nu_1)$.

We next claim that the same is true for the powers $E^j$. To see this, 
we use \eqref{Ejformulae} to show inductively that the expansion at $\zf$ is real to order $\min(2\nu_0+2,2\nu_1)$, since the terms at $\lbo$ and $\rbo$ do not affect the expansion at $\zf$ until order $2\nu_0 + 2$. 

Since $S$ in \eqref{Sdefn} is given by a finite sum of the form $\sum_{j=1}^{2N} (-1)^j E^j$ up to a term that vanishes to high order $\sim N$ at $\zf$, the same property holds true for $S$. By the same arguments, $GE^j$ has a real kernel up to order $\min(2\nu_0+1,2\nu_1)$. 

Finally we show, using exactly the same method as for the powers $E^j$, that the composition $GS$ has a real expansion at $\zf$ in powers of $\la$ up to $\la^{2\nu_0}$, and since $GS \sim G\sum_{j\geq 1}(-1)^jE^j$ (as an expansion at $\zf$), we see that the resolvent $R(\lambda)$ itself has the property that it has a real expansion at $\zf$ up to order $2\nu_0$ with the leading term of $dE_{P_+^{1/2}}(\la)$  
given by (transforming back to the original scattering metric $g$)
\begin{equation}
(dE)_{\zf}^{2\nu_0+1} =  w(z)  w(z') |dg dg'|^{1/2},
\label{smzf}\end{equation}
where $w |dg|^{1/2} := x^{-1} \tilde w |dg_b|^{1/2}$ satisfies $P  w = 0$, and $ w = x^{n/2 - 1 - \nu_0} W(y) + O(x^{n/2 - 1 - \nu_0 + \epsilon})$,  $x \to 0.$ For example, if the potential function $V$ vanishes, then $W$ is constant, $\nu_0 = (n/2 - 1)$, and  $ w$ is just a constant function. 
Moreover the next asymptotic term in the expansion of $dE_{P_+^{1/2}}(\la)$ is at order $\min(2\nu_0+1,2\nu_1)$.
\end{proof}

\begin{proof}[Proof of Theorem~\ref{mainsm}] Theorem~\ref{mainsm} follows directly from Theorem~\ref{mainres} and the cancellations proved in Section~\ref{behavdiag} and Proposition~\ref{cancellation}.
\end{proof}

\begin{proof}[Proof of Corollaries~\ref{cor:waves} and \ref{propagators}] 
Using the spectral measure, we write 
\begin{equation}
\indic_{(0, \infty)}(P) \chi(P)F_t(\sqrt{P})=\int_{\rr^+} \chi(\la^2)F_t(\la)dE_{P_+^{1/2}}(\la)
\label{F_t(P)}\end{equation}
with $F_t(\la)=e^{it\la^2}$, $\cos(t\la)$ or $\sin(t\la)/\la$. Then the leading asymptotic of the kernel $\chi(P)F_t(\sqrt{P})(z,z')$ 
as $t\to \infty$ for $z,z'\in M^0$ is straightforward by using Theorem \ref{mainth}: the leading term 
$\la^{2\nu_0+1} w(z) w(z')$ of $dE_{P_+^{1/2}}(\la)$ at $\la=0$
produces the leading term in \eqref{waveexp1} or \eqref{propexp} as $t\to \infty$, while the error $O(\la^{\min(2\nu_0+2,2\nu_1+1)}))$ 
contributes at most an error term in \eqref{waveexp1} or \eqref{propexp} as $t\to \infty$, by using the fact that $dE_{P_+^{1/2}}(\la;z,z')$ is smooth in $z,z'$ and 
polyhomogeneous conormal at $\la=0$. Moreover, 
putting $F_t(\lambda) = e^{-i\lambda t}$, and using  \cite[Example 7.1.17]{Ho}, we find that 
$$
\int_0^\infty \chi(\lambda) e^{\pm it\lambda} \lambda ^{2\nu_0 + 1} \, d\lambda = \Gamma(2\nu_0 + 2) e^{\pm i \pi (\nu_0 + 1)} t^{-2(\nu_0 + 1)} + O(t^{-\infty}), \quad t \to \infty,
$$
which gives the constants in \eqref{waveexp1}. 
\end{proof}

\begin{remark}\label{manyends} If we do not assume that the boundary $\partial M$ is connected, then our analysis works much as above. However, if $\partial M$ has more than one component, the lowest eigenvalue $\nu_0$ of the operator $\Delta_{\partial M} + V_0 + (n/2 - 1)^2$ need not be simple, and the leading asymptotic $(dE)_{\zf}^{2\nu_0+1}$ would then have rank equal to the multiplicity of $\nu_0$. Similarly, the expansion \eqref{propexp} of the propagator as $t \to \infty$ would have rank equal to the multiplicity of $\nu_0$.
\end{remark} 

\section{Index of Notation}


\begin{tabbing}
\bf{Notation} \hskip 25pt \= \bf{Description/definition of notation} \hskip 60pt \= \bf{Reference} \\
 \> \> \\
$(M^\circ, g)$ \> asymptotically conic manifold \> Section~\ref{sect:Geometricsetting}\\
$M$ \> compactification of $M^\circ$ to manifold with boundary \> Section~\ref{sect:Geometricsetting} \\
$x$; $r$ \> $x$ is a boundary defining function on $M$, $x = 1/r$ \> Section~\ref{sect:Geometricsetting} \\
$g$ \> scattering metric on $M^\circ$ \> Section~\ref{sect:Geometricsetting} \\
$h(x)$ \> family of metrics on $\partial M$ \> Section~\ref{sect:Geometricsetting} \\
$V$ \> potential function on $M$ \> Section~\ref{sect:Geometricsetting} \\
$V_0$ \> leading asymptotic of $V$ at $\partial M$, equal to $x^{-2} V |_{\partial M}$ \> Section~\ref{sect:Geometricsetting} \\
$\Delta_g$ \> positive Laplacian on $(M, g)$ \> Section~\ref{sect:Geometricsetting} \\
$\Delta_{\partial M}$ \> positive Laplacian with respect to $h(0)$ on $L^2(\partial M)$ \> Section~\ref{sect:Geometricsetting} \\
$\nu_j^2$ \> eigenvalue of $\Delta_{\partial M} + (n-2)^2/4 + V_0$ on $L^2(\partial M)$ \> Section~\ref{sect:Geometricsetting} \\
$P$ \> $\Delta_g + V$ \> Section~\ref{sect:Geometricsetting} \\
$P_+$ \> positive spectral part of $P$, $\indic_{(0, \infty)}(P) \circ P$ \> Section~\ref{sect:Geometricsetting} \\
$R(\sigma)$ \> $(P - \sigma^2)^{-1}, \ \Im \sigma > 0$ \> Section~\ref{sect:Geometricsetting} \\
$R(\lambda + i0)$ \> boundary value of $R(\sigma)$ for $\Re \sigma = \lambda \in \RR$, $\Im \sigma \downarrow 0$ \> \\
$[X; Y]$ \> real blowup of a manifold with corners $X$ at $Y$ \> Section~\ref{5.1} \\
$\MMkb$ \> blowup of $[0,1] \times M \times M$ \> Section~\ref{5.1} \\
$\MMksc$ \> further blowup of $\MMkb$ \> Section~\ref{5.1} \\
$\zf$, $\lb$, $\rb$, $\bfc$,  \> boundary hypersurfaces of  $\MMkb$ \> Section~\ref{5.1} \\
$\lbo$, $\rbo$, $\bfo$ \> \>  \\
$\Delta_{k,b}$ \> closed lifted diagonal in $\MMkb$ \> Section~\ref{5.1} \\
$E$, $\mathcal{E}$ \> index set, index family \> Section~\ref{sect:phg} \\
$E_1 + E_2$, $E_1 \extunion E_2$ \> \ \ Addition and extended union of index sets \> Section~\ref{sect:phg} \\
Integral, one-step  \> \ \ Properties of index sets \> Section~\ref{sect:phg} \\
$\min E$ \> minimal exponent in an index set \> Section~\ref{sect:phg} \\
$\Vb(M), \Vsc(M)$ \> b-vector fields, resp. scattering vector fields on $M$ \> Section~\ref{sect:compcotbundle} \\
$\rho$ \> $x/\lambda$ \> Section~\ref{sect:compcotbundle} \\
$\mathcal{V}_{k,b}(\Mkb)$ \>  certain Lie algebra of vector fields  on $\Mkb$ \> Section~\ref{sect:compcotbundle} \\
$\mathcal{V}_{k,b}(\MMkb)$ \> analogous Lie algebra of vector fields of $\MMkb$ \> Section~\ref{sect:compcotbundle} \\
$\Tkb \Mkb$ \> compressed tangent bundle  of $\Mkb$ \> Section~\ref{sect:compcotbundle} \\
$\Tkb \MMkb $ \> compressed tangent bundle of $\MMkb$ \> Section~\ref{sect:compcotbundle} \\
$\Tkbstar \Mkb $ \> compressed cotangent bundle  of $\Mkb$ \> Section~\ref{sect:compcotbundle} \\
$\Tkbstar \MMkb $ \> compressed cotangent bundle of $\MMkb$ \> Section~\ref{sect:compcotbundle} \\
$\nu, \mu, \nu', \mu', T$ \> coordinates on the fibres of $\Tkbstar \MMkb$ \> Section~\ref{sect:compcotbundle} \\
$\Omegakb(\MMkb)$ \> compressed density bundle of $\MMkb$ \> Section~\ref{sect:densities} \\
$\sigma$ \> $x/x'$ \> Section~\ref{facs} \\
$Z_\bullet$ \> base of fibration on boundary hypersurface $\bullet $ of $\MMkb$ \> Section~\ref{facs} \\
$\Nsfstar Z_\bullet$ \> bundle over $Z_\bullet$, admitting contact structure \> Section~\ref{facs} \\
$L$, $\Lambda$ \> Legendre submanifold \> Section~\ref{sec:legsub} \\
$(\Lambda_0, \Lambda_1)$ \> intersecting pair of Legendre submanifolds \> Section~\ref{sec:legsub} \\
$(\Lambda, \Lambda^\sharp)$ \> conic Legendre pair \> Section~\ref{sec:legsub} \\
$\hat \Lambda$ \> blowup of $\Lambda$ to $[\Nscstar Z_{\bfc}; \textrm{span} \Lambda^\sharp]$ \> \eqref{lambdah} \\
$\Nscstar \Diagb$ \> boundary of the conormal bundle to the diagonal  \> Section~\ref{sec:legsub} \\
$\Lbf$ \> propagating Legendrian \> Section~\ref{sec:legsub} \\
$L^\sharp_\mp$ \> incoming/outgoing Legendrian \> Section~\ref{sec:legsub} \\
$\sigma_l(P - \lambda^2)$ \> symbol of $P - \lambda^2$ acting in the left factor \> Section~\ref{sec:legsub} \\
$\sigma_r(P - \lambda^2)$ \> symbol of $P - \lambda^2$ acting in the right factor \> Section~\ref{sec:legsub} \\
$V_l$ \> Hamilton vector field of $\sigma_l(P - \lambda^2)$ \> Section~\ref{sec:legsub} \\
$V_r$ \> Hamilton vector field of $\sigma_r(P - \lambda^2)$ \> Section~\ref{sec:legsub} \\
$ I^{m, r_{\lb}, r_{\rb}; \mcA}(\MMkb, \Lambda; \Omegakbh)$  \> \> \\  \> space of Legendre distributions 
\> Defn~\ref{defn:legdist} \\
$ I^{m, r_{\lb}, r_{\rb}; \mcA}(\MMkb, (\Lambda_0, \Lambda_+); \Omegakbh)$ \> \> \\ \> space of intersecting Legendre distributions \> Defn~\ref{defn:intlegdist} \\
$ I^{m, p; r_{\lb}, r_{\rb}; \mcA}(\MMkb, (\Lambda, \Lambda^\sharp); \Omegakbh)$ \> \> \\ \> space of distributions assoc. to Legendre conic pair \> Defn~\ref{defn:coniclegdist} \\
$S^{[m]}(\Lambda)$ \> line bundle over $\Lambda$ \> \eqref{S[m]defn} \\
$\mathcal{L}_{V_l}$ \> Lie derivative of $V_l$ on half-densities \> Section~\ref{sect:symbolcalc} \\
$g_{\conic}$ \> $dr^2 + r^2 h$ on $(0, \infty) \times Y$ \> Section~\ref{exactcone} \\
$P_{\conic}$ \> homogeneous Schr\"odinger operator on metric cone \> \eqref{conicSchr} \\
$P_{b, \conic}$ \> $r P_{\conic} r$ \> \eqref{Pb} \\
$\gcyl$; $|d\gcyl|$ \> $r^{-2} g_{\conic}$; $|dr/r dh|$ \> Section~\ref{exactcone}  \\
$Z$ \> $(0, \infty) \times Y$ \> Section~\ref{sect:coniccomp} \\
$Z^2_b$ \> b-double space of $Z$ \> \eqref{Z2bdefn} \\
$Z^2_{b, \sca}$ \> further blowup of $Z^2_b$ \> \eqref{Z2bscdefn} \\
$g_b$; $|dg_b|$ \> $r^{-2} g$; $r^{-n} |dg|$ \> Section~\ref{sect:constructionGb} \\
$T_\bullet^k$ \> coefficient of $\lambda^k$ in expansion of $T$ at $\bullet$ \> Section~\ref{sect:constructionGb} \\

\end{tabbing}

\end{document}

%% file: Mk2b2.pstex_t
\begin{picture}(0,0)%
\includegraphics{Mk2b2.pstex}%
\end{picture}%
\setlength{\unitlength}{4144sp}%
\begingroup\makeatletter\ifx\SetFigFont\undefined%
\gdef\SetFigFont#1#2#3#4#5{%
  \reset@font\fontsize{#1}{#2pt}%
  \fontfamily{#3}\fontseries{#4}\fontshape{#5}%
  \selectfont}%
\fi\endgroup%
\begin{picture}(5886,4677)(4340,-4848)
\put(7379,-3525){\makebox(0,0)[lb]{\smash{{\SetFigFont{12}{14.4}{\rmdefault}{\mddefault}{\updefault}{\color[rgb]{0,0,0}$\rb$}%
}}}}
\put(5812,-1698){\makebox(0,0)[lb]{\smash{{\SetFigFont{12}{14.4}{\rmdefault}{\mddefault}{\updefault}{\color[rgb]{0,0,0}$\lb$}%
}}}}
\put(5912,-2778){\makebox(0,0)[lb]{\smash{{\SetFigFont{12}{14.4}{\rmdefault}{\mddefault}{\updefault}{\color[rgb]{0,0,0}$\bfa$}%
}}}}
\put(9225,-1794){\makebox(0,0)[lb]{\smash{{\SetFigFont{12}{14.4}{\rmdefault}{\mddefault}{\updefault}{\color[rgb]{0,0,0}$\zf$}%
}}}}
\put(7532,-2061){\makebox(0,0)[lb]{\smash{{\SetFigFont{12}{14.4}{\rmdefault}{\mddefault}{\updefault}{\color[rgb]{0,0,0}$\bfo$}%
}}}}
\put(5034,-3076){\makebox(0,0)[lb]{\smash{{\SetFigFont{12}{14.4}{\rmdefault}{\mddefault}{\updefault}{\color[rgb]{0,0,0}$x'/x$}%
}}}}
\put(7585,-623){\makebox(0,0)[lb]{\smash{{\SetFigFont{12}{14.4}{\rmdefault}{\mddefault}{\updefault}{\color[rgb]{0,0,0}$\lbo$}%
}}}}
\put(9270,-3063){\makebox(0,0)[lb]{\smash{{\SetFigFont{12}{14.4}{\rmdefault}{\mddefault}{\updefault}{\color[rgb]{0,0,0}$\rbo$}%
}}}}
\put(7584,-1029){\makebox(0,0)[lb]{\smash{{\SetFigFont{12}{14.4}{\rmdefault}{\mddefault}{\updefault}{\color[rgb]{0,0,0}$\lambda/x$}%
}}}}
\put(8633,-2633){\makebox(0,0)[lb]{\smash{{\SetFigFont{12}{14.4}{\rmdefault}{\mddefault}{\updefault}{\color[rgb]{0,0,0}$\lambda/x'$}%
}}}}
\end{picture}%

%% file: M2ksc.pstex_t
\begin{picture}(0,0)%
\includegraphics{M2ksc.pstex}%
\end{picture}%
\setlength{\unitlength}{4144sp}%
\begingroup\makeatletter\ifx\SetFigFontNFSS\undefined%
\gdef\SetFigFontNFSS#1#2#3#4#5{%
  \reset@font\fontsize{#1}{#2pt}%
  \fontfamily{#3}\fontseries{#4}\fontshape{#5}%
  \selectfont}%
\fi\endgroup%
\begin{picture}(5869,4655)(4357,-4826)
\put(8948,-2784){\makebox(0,0)[lb]{\smash{{\SetFigFontNFSS{12}{14.4}{\rmdefault}{\mddefault}{\updefault}{\color[rgb]{0,0,0}$\rbo$}%
}}}}
\put(7661,-2391){\makebox(0,0)[lb]{\smash{{\SetFigFontNFSS{12}{14.4}{\rmdefault}{\mddefault}{\updefault}{\color[rgb]{0,0,0}$\bfo$}%
}}}}
\put(7635,-691){\makebox(0,0)[lb]{\smash{{\SetFigFontNFSS{12}{14.4}{\rmdefault}{\mddefault}{\updefault}{\color[rgb]{0,0,0}$\lbo$}%
}}}}
\put(9548,-2018){\makebox(0,0)[lb]{\smash{{\SetFigFontNFSS{12}{14.4}{\rmdefault}{\mddefault}{\updefault}{\color[rgb]{0,0,0}$\zf$}%
}}}}
\put(6968,-3511){\makebox(0,0)[lb]{\smash{{\SetFigFontNFSS{12}{14.4}{\rmdefault}{\mddefault}{\updefault}{\color[rgb]{0,0,0}$\rb$}%
}}}}
\put(5848,-1378){\makebox(0,0)[lb]{\smash{{\SetFigFontNFSS{12}{14.4}{\rmdefault}{\mddefault}{\updefault}{\color[rgb]{0,0,0}$\lb$}%
}}}}
\put(5868,-3031){\makebox(0,0)[lb]{\smash{{\SetFigFontNFSS{12}{14.4}{\rmdefault}{\mddefault}{\updefault}{\color[rgb]{0,0,0}$\bfa$}%
}}}}
\put(6468,-2565){\makebox(0,0)[lb]{\smash{{\SetFigFontNFSS{12}{14.4}{\rmdefault}{\mddefault}{\updefault}{\color[rgb]{0,0,0}$\sca$}%
}}}}
\end{picture}%

%% file: Z2bsc.pstex_t
\begin{picture}(0,0)%
\includegraphics{Z2bsc.pstex}%
\end{picture}%
\setlength{\unitlength}{4144sp}%
\begingroup\makeatletter\ifx\SetFigFontNFSS\undefined%
\gdef\SetFigFontNFSS#1#2#3#4#5{%
  \reset@font\fontsize{#1}{#2pt}%
  \fontfamily{#3}\fontseries{#4}\fontshape{#5}%
  \selectfont}%
\fi\endgroup%
\begin{picture}(2200,2179)(5861,-2825)
\put(7815,-993){\makebox(0,0)[lb]{\smash{{\SetFigFontNFSS{12}{14.4}{\rmdefault}{\mddefault}{\updefault}{\color[rgb]{0,0,0}$\zf$}%
}}}}
\put(8046,-1956){\makebox(0,0)[lb]{\smash{{\SetFigFontNFSS{12}{14.4}{\rmdefault}{\mddefault}{\updefault}{\color[rgb]{0,0,0}$\rbo$}%
}}}}
\put(6630,-793){\makebox(0,0)[lb]{\smash{{\SetFigFontNFSS{12}{14.4}{\rmdefault}{\mddefault}{\updefault}{\color[rgb]{0,0,0}$\lbo$}%
}}}}
\put(6457,-2413){\makebox(0,0)[lb]{\smash{{\SetFigFontNFSS{12}{14.4}{\rmdefault}{\mddefault}{\updefault}{\color[rgb]{0,0,0}$\sca$}%
}}}}
\put(5876,-1580){\makebox(0,0)[lb]{\smash{{\SetFigFontNFSS{12}{14.4}{\rmdefault}{\mddefault}{\updefault}{\color[rgb]{0,0,0}$\lb$}%
}}}}
\put(5885,-2394){\makebox(0,0)[lb]{\smash{{\SetFigFontNFSS{12}{14.4}{\rmdefault}{\mddefault}{\updefault}{\color[rgb]{0,0,0}$\bfa$}%
}}}}
\put(7021,-2761){\makebox(0,0)[lb]{\smash{{\SetFigFontNFSS{12}{14.4}{\rmdefault}{\mddefault}{\updefault}{\color[rgb]{0,0,0}$\rb$}%
}}}}
\end{picture}%

%% file: support.pstex_t
\begin{picture}(0,0)%
\includegraphics{support.pstex}%
\end{picture}%
\setlength{\unitlength}{4144sp}%
\begingroup\makeatletter\ifx\SetFigFontNFSS\undefined%
\gdef\SetFigFontNFSS#1#2#3#4#5{%
  \reset@font\fontsize{#1}{#2pt}%
  \fontfamily{#3}\fontseries{#4}\fontshape{#5}%
  \selectfont}%
\fi\endgroup%
\begin{picture}(5886,4677)(4340,-4848)
\put(7379,-3525){\makebox(0,0)[lb]{\smash{{\SetFigFontNFSS{12}{14.4}{\rmdefault}{\mddefault}{\updefault}{\color[rgb]{0,0,0}$\rb$}%
}}}}
\put(5812,-1698){\makebox(0,0)[lb]{\smash{{\SetFigFontNFSS{12}{14.4}{\rmdefault}{\mddefault}{\updefault}{\color[rgb]{0,0,0}$\lb$}%
}}}}
\put(5912,-2778){\makebox(0,0)[lb]{\smash{{\SetFigFontNFSS{12}{14.4}{\rmdefault}{\mddefault}{\updefault}{\color[rgb]{0,0,0}$\bfa$}%
}}}}
\put(9045,-3448){\makebox(0,0)[lb]{\smash{{\SetFigFontNFSS{12}{14.4}{\rmdefault}{\mddefault}{\updefault}{\color[rgb]{0,0,0}$\rbo$}%
}}}}
\put(7525,-2515){\makebox(0,0)[lb]{\smash{{\SetFigFontNFSS{12}{14.4}{\rmdefault}{\mddefault}{\updefault}{\color[rgb]{0,0,0}$\bfo$}%
}}}}
\put(7585,-728){\makebox(0,0)[lb]{\smash{{\SetFigFontNFSS{12}{14.4}{\rmdefault}{\mddefault}{\updefault}{\color[rgb]{0,0,0}$\lbo$}%
}}}}
\put(9451,-1548){\makebox(0,0)[lb]{\smash{{\SetFigFontNFSS{12}{14.4}{\rmdefault}{\mddefault}{\updefault}{\color[rgb]{0,0,0}$\zf$}%
}}}}
\end{picture}%

%% file: spmeas.rv.bbl
\begin{thebibliography}{99} 

\bibitem{AS} 
{M.~Abramowitz and I.A. Stegun,}
\newblock {\em  Handbook of mathematical functions with formulas, graphs, and
              mathematical tables,}  
\newblock {National Bureau of Standards Applied Mathematics Series}, {55},
U.S. Government Printing Office, Washington, D.C. (1964).



\bibitem{BH0} J.-F.  Bony, D. H\"afner, \emph{Decay and non-decay of the local energy for the wave equation in the De Sitter - Schwarzschild metric}, Comm. Math. Phys. \textbf{282}, no. 3, 697-719.

\bibitem{BH1}  J.-F.  Bony, D. H\"afner, \emph{The semilinear wave equation on asymptotically Euclidean manifolds},  arXiv:0810.0464.

\bibitem{BH2}  J.-F.  Bony, D. H\"afner, \emph{Low frequency resolvent estimates for long range perturbations of the Euclidean Laplacian},  Math. Res. Lett.  \textbf{17}  (2010),  no. 2, 301--306. 

\bibitem{BH3} J.-F. Bony, D. H\"afner, \emph{Local energy decay for several evolution equations on asymptotically 
euclidean manifolds},  arXiv:1008.2357. 


\bibitem{Bouclet} J.-M. Bouclet, \emph{Low energy behaviour of powers of the resolvent of long range perturbations of the Laplacian}, Proc. Centre Math. Appl. Aust. Nat. Univ. \textbf{44}, 2010, 115--127. 

\bibitem{Bou} J.-M. Bouclet \emph{Low frequency estimates and local energy decay for asymptotically euclidean Laplacians}, 
arXiv:1003.6016. 

\bibitem{BS} J. Br\"uning and R. Seeley, \emph{The resolvent expansion for second order regular singular operators}, J. Funct. Anal. \textbf{73}, no. 2, (1987), 369--429.


\bibitem{Callias} C. Callias, \emph{The heat equation with singular coefficients. I.},  Comm. Math. Phys. \textbf{88}, no. 3 (1983), 357--385.

\bibitem{CT} J. Cheeger and M. Taylor, \emph{On the diffraction of waves by conical singularities I, II}, Comm. Pure Appl. Math. \textbf{35} (1982), 275--331, 487--529. 

\bibitem{DSS1} R. Donninger, W. Schlag, A. Soffer,  \emph{A proof of Price's Law on Schwarzschild black hole manifolds for all angular momenta}, arxiv arXiv:0908.4292

\bibitem{DSS} R. Donninger, W. Schlag, A. Soffer, \emph{On pointwise decay of linear waves on a Schwarzschild black hole background}, arXiv:0911.3179. 

\bibitem{GH1} C. Guillarmou, A. Hassell, 
\emph{Resolvent at low energy and Riesz transform for Schrodinger operators on asymptotically conic manifolds, I }.
Math. Ann. \textbf{341} (2008), 859--896.

\bibitem{GH2} C. Guillarmou, A. Hassell, 
\emph{Resolvent at low energy and Riesz transform for Schrodinger operators on asymptotically conic manifolds, II}, Annales de l'Institut Fourier \text{59} (2009), 1553 -- 1610.


\bibitem{GHS2} C. Guillarmou, A. Hassell, A. Sikora, \emph{Restriction and spectral multiplier theorems on asymptotically conic manifolds}, arXiv:1012.3780.

\bibitem{asatet} A. Hassell, R. Mazzeo and R. B. Melrose, \emph{Analytic surgery and the accumulation of eigenvalues}, Commun. in Anal. and Geom. \textbf{3} (1995), 115-222. 

\bibitem{HV1} A.~Hassell and A.~Vasy, \emph{The spectral projections and the resolvent for scattering metrics}, J. d'Analyse Math. \textbf{79} (1999), 241--298. 


\bibitem{HV2} A.~Hassell and A.~Vasy, \emph{The resolvent for Laplace-type operators on asymptotically conic spaces.},  Ann. Inst. Fourier (Grenoble)  \textbf{51}(5) (2001), 1299-1346.


\bibitem{HW} A. Hassell and J. Wunsch, \emph{The semiclassical resolvent and the propagator for non-trapping scattering metrics}, Adv. Math. \textbf{217} (2008), no. 2, 586--682. 

\bibitem{Ho} Lars H{\"o}rmander. \newblock {\emph The analysis of linear partial differential operators. {I}}. \newblock Springer-Verlag, Berlin, 1983, 1990. 

\bibitem{JK} A. Jensen, T. Kato, \emph{Spectral properties of Schr\"odinger operators and time-decay of the wave functions}, 
Duke Math. J. \textbf{46} (1979), 583-611.


\bibitem{Li} H.-Q. Li, \emph{La transform\'ee de Riesz sur les vari\'et\'es coniques}, J. Funct. Anal. \textbf{168}
(1999), no 1, 145-238.


\bibitem{APS} R.B. Melrose, \emph{The Atiyah-Patodi-Singer index theorem} 
(AK Peters, Wellesley, 1993).


\bibitem{Kyoto} R.B. Melrose, \emph{Pseudodifferential operators, corners and singular limits}, in Proceedings of the International Congress of Mathematicians, Vol. I, II (Kyoto, 1990), 217--234, Math. Soc. Japan, Tokyo, 1991. 

\bibitem{MS} R.B. Melrose, A. Sa Barreto, \emph{Zero energy limit for scattering manifolds}, unpublished note.

\bibitem{scatmet}
R.B. Melrose, Spectral and scattering theory for the Laplacian on
  asymptotically Euclidian spaces, in {\it Spectral and Scattering Theory}, M. Ikawa, ed., Marcel Dekker, 1994.
  
\bibitem{cocdmc}
R.B. Melrose, \emph{Calculus of conormal distributions on manifolds with corners}, Int. Math. Res. Not.  3 (1992), 51--61.

\bibitem{MZ}
R. B. Melrose and M. Zworski, \emph{Scattering metrics and geodesic flow
  at infinity}, Invent. Math. \textbf{124} (1996), no.~1-3, 389--436.
 

\bibitem{Murata} M. Murata, \emph{Asymptotic expansions in time for solutions of Schr\"odinger-type equations}, 
\textbf{49}, 10-56 (1982), 10-56.

\bibitem{Pr1} R.H. Price. \emph{Nonspherical perturbations of relativistic gravitational collapse. II. Integer-spin, zero-rest-mass 
fields.}  Phys. Rev. D (3),  (1972) 5:24392454.

\bibitem{Pr2} R.H. Price, L.M. Burko. \emph{Late time tails from momentarily stationary, compact initial data in 
Schwarzschild spacetimes.} Phys. Rev. D, 70(8) (2004) 084039.


\bibitem{SBZ} A. S\'a Barreto, M. Zworski, \emph{Distribution of resonances for spherical black holes}, Math. Res. Lett. \textbf{4} 
(1997), no. 1, 103--121.  

\bibitem{Ta} D. Tataru, \emph{Local decay of waves on asymptotically flat stationary space-times},  arXiv:0910.5290; Amer. J. Math., to appear. 

\bibitem{VW} A. Vasy and J. Wunsch, \emph{Positive commutators at the bottom of the spectrum},  J. Funct. Anal. \textbf{259} (2010), no. 2, 503--523. 

\bibitem{XPW} X-P. Wang, \emph{Asymptotic expansion in time of the Schr\"odinger group on conical manifolds},
to appear, Annales Inst. Fourier (2006). 

\bibitem{Yafaev} D. Yafaev, \emph{Scattering theory: some old and new problems}, Springer Lecture Notes in Mathematics, vol. 1735, 2000. 


\end{thebibliography}
